\documentclass[a4paper,12pt]{article}
\usepackage[utf8]{inputenc}

\usepackage{amssymb}
\usepackage{amsthm}
\usepackage{amsmath}
\usepackage{mathrsfs}
\usepackage{enumerate}
\usepackage{color}
\usepackage{graphicx}
\usepackage{psfrag}
\usepackage{constants}
\usepackage{authblk}

\usepackage[hypertexnames=false]{hyperref} 
\hypersetup{
  colorlinks   = true,    
  urlcolor     = blue,    
  linkcolor    = blue,    
  citecolor    = green      
}

\def\Ind{\mathbf{1}}

\newtheorem{theorem}{Theorem}[section]
\newtheorem{definition}[theorem]{Definition}
\newtheorem{proposition}[theorem]{Proposition}
\newtheorem{lemma}[theorem]{Lemma}
\newtheorem{remark}[theorem]{Remark}
\newtheorem{corollary}[theorem]{Corollary}
\newtheorem{convention}[theorem]{Convention}
\newtheorem{notation}[theorem]{Notation}
\newtheorem{assumption}[theorem]{Assumption}

\def\eps{\varepsilon}
\def\phi{\varphi}
\def\IR{\mathbb{R}}

\def\IT{\mathbb{T}}

\def\IZ{\mathbb{Z}}
\def \diff {\mathrm{d}}

\def\sigalg{\mathcal{F}} 

\def\cM{\mathcal{M}} 
\def\cR{\mathcal{R}}

\def\cC{\mathcal{C}}
\def\cone{\mathscr{C}} 

\numberwithin{equation}{section}
\title{Equidistribution for standard pairs in planar dispersing billiard flows}

\author[1]{P\'eter B\'alint}
\author[2]{P\'eter N\'andori}
\author[1]{Domokos Sz\'asz}
\author[1,3]{Imre P\'eter T\'oth}
\affil[1]{Institute of Mathematics, Budapest University of Technology and
  Economics, Egry J\'ozsef u. 1, H-1111 Budapest, Hungary}
\affil[2]{  Department of Mathematics, University of Maryland, 4176 Campus Drive,
  College Park, MD 20742, USA}
\affil[3]{  MTA-BME Stochastics Research group, Egry J\'ozsef u. 1, H-1111
  Budapest, Hungary}

\begin{document}

\maketitle


\begin{abstract}
We prove exponential correlation decay in dispersing billiard flows on the 2-torus assuming finite horizon and lack of corner points. With applications aimed at describing heat conduction, the highly singular initial measures are concentrated here on $1$-dimensional submanifolds (given by standard pairs) and the observables are supposed to satisfy a generalized H\"older continuity property. The result is based on the exponential correlation decay bound of Baladi, Demers and Liverani \cite{BDL16} obtained for H\"older continuous observables in these billiards. The model dependence of the bounds is also discussed.
\end{abstract}

\tableofcontents
\section{Introduction}

\subsection{General introduction}\label{sec:general-intro}

Decay of correlations is a most useful property when establishing probabilistic laws for stochastic processes. What is more - and this is the main reason of our interest - it has a pivotal role in non-equilibrium statistical physics since it also controls convergence to equilibrium. Usually correlation decay bounds are originally established when the system starts from a nice measure absolutely continuous with respect to the smooth invariant measure of the system. In contrast, our goal here is to present a {\it correlation bound for planar Sinai billiard flows in the case where the initial measure is determined by a standard pair}. Briefly saying a {\it standard pair} is a smooth unstable curve together with a nice probability density on it hence determining a singular measure in the phase space (the precise definition will be given in Section \ref{sec:stpair}). The tool of standard pairs was initiated in \cite{D04,D05} where it appeared as a much effective and flexible variant of Markov approximations of hyperbolic
dynamical systems. Since then, it has been widely used when tackling a variety of problems (see \cite{BChD11,BChD16,ChD08,ChD09B,ChD09C,CDP14,DL11,DN16A,DN16B,DdS12,DSzV08,DSzV09}).
In particular, in diverse approaches studying the Fourier law of heat conduction for Hamiltonian models, it occurs that correlation bounds - and the resulting convergence to equilibrium - for {\it a hyperbolic {\rm flow} starting from precisely a standard pair} seem inevitable (see e. g. \cite{DL11,DN16B,BNSzT16}). \footnote{\cite{DL11} uses the method of standard pairs for the study of systems with \textit{slow-fast degrees of freedom} whereas \cite{DN16B,BNSzT16} apply it to treat the \textit{rare interaction limit}.} {Although} on the one hand, our prime motivation was the aforementioned application (for instance \cite{DN16B} also uses our result), we, on the other hand, also mean our work as a contribution to and a reference on the {methodology of standard pairs applied to flows}.

Our result is based on correlation bounds when the initial measure is a nice smooth one. Indeed, for planar Sinai billiards {\it exponential correlation bounds for the billiard ball map} had been known since the late 90's (see \cite{Y98} -- finite horizon case, \cite{Ch99} -- infinite horizon case). However, it took time until, in 2007, Chernov \cite{Ch07} could derive a {\it stretched exponential correlation bound for the flow} in case of so-called generalized H\"older continuous observables. (Earlier \cite{M07} had obtained superpolynomial correlation bounds for functions smooth in the flow direction). Quite recently, Baladi, Demers and Liverani \cite{BDL16} were finally able to {\it achieve an exponential bound} for H\"older observables.\footnote{ \cite{BDL16} also contains more details on earlier results providing correlation decay for planar dispersing billiard flows.}
 It is essential to note here that the aforementioned results were drawn up in terms of the invariant Liouville measure as an initial measure and they can be naturally extended to cases when the initial law is absolutely continuous with respect to it with a nice density.

As said above it is, however, substantial for some actual applications (see \cite{DN16B,BNSzT16} that  the initial measure be given by a standard pair (and, moreover, the observable be H\"older in a weaker, so-called generalized sense, only). Such a result is the goal of this paper and will be the claim of our Main Theorem \ref{thm:main}.
Under these circumstances a natural way to derive a correlation bound starting from a standard pair is 1) to smear the singular measure given by the standard pair in a small neighbourhood of its unstable curve component and 2) apply a known result to this smooth initial measure. This idea works well for smooth dynamical systems, for instance in case of geodesic flows \cite{DL11}.  However, for singular systems, like billiards, there arises a substantial difficulty, and actually the bulk of the paper fights exactly this obstacle. We discuss the problem in detail in Section~\ref{sec:about-the-proof}. We also note that though the Banach spaces used by \cite{BDL16} also contain singular measures nevertheless the standard pairs we need and consider here do not belong to those Banach spaces.

Here, we prove our results for the simplest possible class of Sinai billiards: for planar dispersing billiards with finite horizon and no corner points. However, in order to discuss possible extensions, we present a setup which is slightly more general, and also allows corner points. It is worth mentioning that in the recent work \cite{CPZ17} a similar approach, i. e. taking  iterates of standard pairs, is successfully used for constructing SRB measures for multidimensional, non-uniformly hyperbolic maps.

The billiard
ball map has been extensively studied in the literature, see e.g. \cite{ChM06} and \cite{ChD09A}. In the preliminary
Section \ref{sec:prelim}, we extend several notions and definitions {\it from the billiard ball  map to the billiard flow}
(with the exception of Section \ref{sec:Holder}, where various notions of H\"older continuity are recalled).
Section \ref{sec:proof} contains the proof of our Main Theorem \ref{thm:main}. In particular, certain substantial foliations are introduced in Sections \ref{sec:conditional}-\ref{sec:product}.
The approximating density is defined in Section \ref{sec:construction}. Then its regularity properties
are studied and adjusted to the foliations in Sections \ref{sec:regularity}  and \ref{sec:smoothing}.
The proof is completed in Section~\ref{sec:using}.
Finally, a possible strengthening and two possible generalizations of our result are discussed in Section~\ref{sec:extensions}.
In particular, as we learnt from \cite{ChD09A}, standard pairs are most appropriate for the perturbative analysis of billiard-like systems therefore we pay special attention to the dependence on the model of the constants appearing in the bounds.\footnote{We note, however, that the proof of the limit transition to a Markov jump process in the model of \cite{BNSzT16} does not need these uniform - in the appearing models - bounds of the constants but they are necessary for more complicated models.}   This question is discussed in Subsection \ref{sec:extensions} (see also Subsection \ref{sec:const-convention}). The aforementioned and desirable generalizations are discussed in Subsections \ref{sec:ext-corner} and \ref{sec:no-eps0}.

About the Appendix: Subsection \ref{sec:Ahom} is devoted to demonstrate three properties,  formulated in Subsection \ref{sec:holonomy}, of holonomy maps along central-stable manifolds. In subsection \ref{sec:genholder-EDC} we extend exponential correlation decay known for H\"older observables to those for generalized H\"older observables. In particular we use Corollary \ref{cor:osc_F-genholder-norm} of a Theorem \ref{thm:osc_f-genHolder} by IP T\'oth, which is proved in his separate paper \cite{T17}. Finally Subsection \ref{sec:Holder-extension-1p} formulates statements about the extension of H\"older continuous function to larger sets.

\subsection{Setup and Main Theorem}

\subsubsection{Billiard table}\label{sec:billiardtable}

Our main discussion is restricted to hyperbolic billiards with smooth scatterers. Nevertheless,
keeping for later purpose we introduce the model a bit more generally, i. e. we permit the boundaries of the scatterers to consist of smooth pieces (called walls, cf. \cite{ChM06}, Section 2.1). Specifically, let $\IT^2$ denote the two-dimensional torus $\IT^2=\IR^2/\IZ^2$. Let $Q_0\subset \IT^2$ be open and connected and let $Q=\overline{Q_0}$ be its closure. Assume that
the boundary of $Q$ is a finite union of curves in $\IT^2$:
\[\partial Q=\Gamma=\Gamma_1\cup\dots\cup \Gamma_r.\]
The curves $\Gamma_i$ are assumed to be $C^3$-smooth of finite length, and intersect only at the endpoints:
\[\Gamma_i\cap \Gamma_j\subset \partial\Gamma_i \cup \partial\Gamma_j\quad\text{for every $i\neq j$}.\]
The closed set $Q\subset\IT^2$ is called the \emph{billiard table} or \emph{billiard flow configuration space}.
The $\Gamma_i$ are called \emph{walls} (while the boundaries of the the connected components of $\IT^2 \setminus Q_0$ are called \emph{scatterers}).
The billiard flow describes the motion of a point particle (called the billiard particle) that moves in $Q$ with unit velocity, so \emph{the phase space of the billiard flow} is
\[M=Q\times S^1\subset \IT^3\]
where $S^1$ is the unit circle. Geometrically, we view it as the set of unit velocity vectors: $S^1=\{v\in\IR^2 \,|\, |v|=1\}$. From a metric perspective, it is better to {view} it as the $1$-torus $S^1=\IT^1=\IR/\IZ\cong \{e^{i2\pi \xi}\,|\,z\in\IR\}$.
The billiard particle moves uniformly (with constant {speed}) in $Q$ until it hits the boundary $\partial Q$. When it reaches the boundary, it bounces back under the rules of elastic collision -- so the velocity remains unit, and its direction changes like the direction of a reflected light ray in geometric optics: the angle of reflection equals the angle of incidence.

For $r\in Q$ and $v\in S^1$ and $x=(r,v)\in M$ we call $r$ the position, $v$ the velocity, and $x$ the phase point of the particle. For a particle with phase point $x=(r,v)\in M$, let $\Phi^t(x)$ give the phase point of the same particle after it moves for time $t$. For $x = (r, v) \in M$ denote $\pi_Q x = r$. \{$\Phi^t:M\to M | t \in \mathbb R$\}  is called the billiard flow and this definition is unambiguous if we assume that, in addition, the trajectories are continuous from the right. \footnote{Another way to make it unambiguous is to identify the left and right limit points when the orbit hits the boundary, cf. \cite{KSSz90,ChD09A}. Both ways are equivalent.}
For $x\in M$, let $\tau(x)$ denote the time of free flight for $x$ until the first collision:
\[\tau(x)=\inf\{t>0\,|\,\pi_Q(\Phi^t (x)) \in \Gamma\}.\]

\subsubsection{Assumptions and statement of the theorem}\label{sec:statement}

We will use the following assumptions:

\begin{assumption}[No corner points]
\label{asm:no-corner}
 All the walls $\Gamma_i$ are closed $C^3$-smooth curves, so they have no boundary and they are disjoint.
\end{assumption}
Consequently the walls and the scatterers are identical.

\begin{assumption}[Dispersing planar billiard]
\label{asm:dispersing}
 All walls are strictly convex when viewed from the outside of $Q$.
\end{assumption}
Consequently all scatterers are strictly convex.

\begin{assumption}[Finite horizon]
\label{asm:finite-horizon}
 There are no phase points {which} can fly {indefinitely} without a collision (or equivalently the free flight time $\tau(x)$ is finite for every $x$ and consequently it is also bounded).
\end{assumption}

Our main result is proven under the assumptions \ref{asm:dispersing}, \ref{asm:finite-horizon} and \ref{asm:no-corner}.\footnote{According to the terminology of \cite{ChM06}, billiards satisfying assumptions \ref{asm:dispersing}, \ref{asm:finite-horizon} and \ref{asm:no-corner} belong to category A.}
Most importantly, our main reference \cite{BDL16} only covers this case. However, our long term aim is to drop Assumption \ref{asm:no-corner} and cover billiards with corner points, under the much weaker Assumption \ref{asm:no-cusp}, which will be discussed in Section~\ref{sec:ext-corner}.

The more often used -- and this was, indeed, traditionally the most favoured and more convenient -- description of billiard dynamics relies upon the discrete phase space
\[\mathcal M = \{(q, v)| q \in \partial Q, |v|\in \IR^2, |v|=1, \langle n(q),v \rangle \ge 0\}\subset M,\]
where $n = n(q)$ is the normal vector of $\partial Q$ at $q$ pointing inwards into $Q$ (that is, out of the scatterers).
Then the related billiard ball map $T: \cM \to \cM$ is defined by $Tx = \Phi^{\tau(x)}(x)$.  (We note that - in the sense of our convention made in Section \ref{sec:billiardtable} - actually $M = (Q_0 \times S^1) \cup \mathcal M$.) We also note that a convenient coordinate for the velocity component $(q, v) \in \mathcal M $ is the angle
$\varrho\in\left[-\frac{\pi}{2},\frac{\pi}{2}\right]$ between $v$ and $n$ implying $\cos \varrho = \langle n,v\rangle$. It is sometimes convenient to use $(q, \varrho)$ instead of $(q, v)$, so -- with some abuse of notation -- we can write
\begin{equation}\label{eq:map-phasespace-with-position-and-angle}
\mathcal M = \{(q, \varrho)| q \in \partial Q, -\frac{\pi}{2} \le \varrho \le \frac{\pi}{2}\}.
\end{equation}

The connection between the two phase spaces is provided by the projection $\Pi:M\to \cM$ that assigns to each point $x\in M$ the
point of the previous collision. That is, if $\tau_-(x)=\min\{t\ge 0 \mid \Phi^{-t}(x)\in\cM\}$, then $\Pi x=\Phi^{-\tau_-(x)}x$. The invariant, Liouville measure $\mu$ for the billiard flow is given by
\[\mu=\frac{1}{Leb(M)} Leb_M \quad\text{, so}\quad\diff \mu = \frac{1}{2\pi Vol (Q)}\diff r \diff v,\]
whereas the natural invariant measure $\nu$ for the billiard ball map is given by
\[\diff\nu = \frac{1}{2|\partial Q|}\langle n,v\rangle \diff q \diff v\quad (= \frac{1}{2|\partial Q|}\cos \varrho \diff q \diff v).\]
In general we are going to follow the terminology of (and several facts from) \cite{ChM06} often without particular reference.

Now we present the main theorem of the paper. Most of the precise terminology will be introduced later, cf. Definition \ref{def:stp}. For now assume that $(W^u,\phi)$ is a sufficiently regular standard pair where $\phi$ is a nice probability density with respect to the length measure $m_{W^u}$ on the $u$-curve $W^u$.  The essence is that if $F$ is a sufficiently regular observable on $M$, then $\int_{W^u} \phi(x)F(\Phi^t x)\diff x \to \int_{W^u} \phi(x) \diff x \int_M F(x) \diff x$ as $t\to\infty$, exponentially fast.

Assume we are given a standard pair $(W^u, \phi)$. In addition, for technical reasons we assume that there is some $\eps_0>0$ such that the $2\eps_0$-neighbourhood of $W^u$ is disjoint from the boundary of $M$. This condition is formulated in Definition~\ref{def:gooducurve}, and discussed more in Section~\ref{sec:no-eps0}.

\begin{theorem}\label{thm:main} Suppose the billiard table $Q$ satisfies assumptions \ref{asm:dispersing}, \ref{asm:finite-horizon} and \ref{asm:no-corner}. Let $0<\Theta_\phi<1$, $0<\alpha_F \le 1$ and $\eps_0>0$. Then there exist $\cC<\infty$ and $a>0$ with the following properties:

Let $(W^u,\phi)$ be a standard pair with a dynamically $\Theta_\phi$-H\"older $\phi$ (cf. Definition \ref{def:dHc}). Assume that the $2\eps_0$-neighbourhood of $W^u$ is disjoint from the boundary of $M$. Let $F:M\to \IR$ be generalized $\alpha_F$-H\"older continuous (cf. Definition \ref{def:GHC}). Then, for every $t\ge 0$,
 \[\left| \int_{W^u} \left(F\circ \Phi^t\right) \phi \diff m_{W^u} - \int_M F\diff\mu \right| \le \cC \cdot ||\phi||_{\Theta_\phi;dH}\cdot var_{\alpha_F}F\cdot  e^{-a t}.\]
Here
 \begin{itemize}
  \item $||\phi||_{\Theta_\phi;dH}$ is the dynamical H\"older norm of $\phi$ (see (\ref{eq:dHn})).
  \item $var_{\alpha_F} F$ is the generalized H\"older seminorm of $F$ (see (\ref{eq:var})).
  \item $a=a(Q,\cR_u,\Theta_\phi,\alpha_F)<\infty$ and $\cC=\cC(Q,\cR_u,\Theta_\phi,\alpha_F)<\infty$ depend on the billiard table $Q$, the regularity of good $u$-curves quantified in $\cR_u$, and the regularity classes of $\phi$ and $F$ given by $\Theta_\phi$ and $\alpha_F$. In other words, they depend on $W^u$, $\phi$ and $F$ through the aforementioned parameters, only.
 \end{itemize}
Moreover, $\cC$ depends on $Q$ only through $\cR_Q$ from (\ref{eq:RQ}) and, furthermore, through $\cC_{BDL}(Q,\alpha)$ from Theorem~\ref{thm:BDL} with some $\alpha=\alpha(\cR_Q,\cR_u,\Theta_\phi)>0$.
Similarly, $a$ depends on $(Q,\cR_u,\Theta_\phi)$ only through $a'(Q,\alpha)$ from Theorem~\ref{thm:BDL} with the same $\alpha=\alpha(\cR_Q,\cR_u,\Theta_\phi)>0$. That is,
\[\cC=\cC(\cR_Q,\cR_u,\alpha_F,\cC_{BDL}(Q,\alpha(\cR_Q,\cR_u,\Theta_\phi)))\]
and
\[a=a(\alpha_F, a'(Q,\alpha(\cR_Q,\cR_u,\Theta_\phi))).\]
\end{theorem}

\subsection{Some words about the proof}\label{sec:about-the-proof}

As mentioned before,
{we will use the recent result of
Baladi, Demers and Liverani on the
exponential correlation bound for H\"older observables.} In fact, Corollary 1.3 (of Theorem 1.2) in their work \cite{BDL16}) says the following:

\begin{theorem}\label{thm:BDL}
Consider a billiard like the one introduced above. Assume $0 < \alpha \le 1$. Then there exist $a' = a'(Q,\alpha) > 0$ and $\cC_{BDL}=\cC_{BDL}(Q,\alpha) < \infty$ such that for any
$F,G:M\to \IR$ $\alpha$-H\"older functions with $\int_M F\diff\mu=0$ and any $t\ge 0$ one has
 \[\left| \int_M (F\circ \Phi^t) G \diff \mu  \right| \le \cC_{BDL}||F||_{\alpha;H}   ||G||_{\alpha;H} e^{- a't}. \]
Here $||.||_{\alpha;H}$ denotes the $\alpha$-H\"older norm defined in (\ref{eq:Holder-norm}).
(BDL stands for Baladi-Demers-Liverani.)
\end{theorem}

Once a bound for correlations under the initial Liouville measure for H\"older observables is known, then, for obtaining a similar one for less regular objects (standard pairs vs. generalized H\"older observables) it is a natural idea -- as this was also sketched in \cite{Ch07} -- to smear both the initial measure and the observable to improve regularity, and then use Theorem~\ref{thm:BDL}.
This is exactly what we will do.

We first generalize Theorem~\ref{thm:BDL} to the case when $F$ is only generalized H\"older continuous. This is done in Section~\ref{sec:genholder-EDC} with a standard approximation argument. The essence is that  in $L^1(\mu)$ a generalized H\"older continuous function can easily be approximated by a truly H\"older one. This is sometimes done by the method of mollification (see, for instance, \cite{E15}).

The second step -- namely, replacing $G\diff\mu$ with a measure given by a standard pair, is by far less trivial. No $L^1$ approximation makes sense, so we have to carefully make use of the regularity of the integrand. However, the integrand $F\circ \Phi^t$ has very bad regularity properties when $t$ is big -- except in the central-stable directions. Consequently the mollification (in other words the smearing of the density component of the standard pair) should act along the central-stable direction. However, not every point has a long enough central-stable manifold. Indeed, for any unstable curve $W^u$ and for any neighbourhood $U$ of $W^u$, those points of the unstable curve, along which smooth pieces of central-stable manifolds fully cross $U$ (i. e. have no boundary points inside $U$), form a Cantor-like subset only. Similar is the situation with the union of the aforementioned smooth pieces inside the given neighbourhood. The bulk of the paper fights exactly this difficulty.
(In fact, in case of geodesic flows, this idea of smearing works much more simply, since this difficulty does not arise there and e.g. \cite{DL11} completes the proof in half a page.)

We note that doing the two steps the other way round -- i.e. first allowing standard pairs instead of $G$ and then allowing $F$ to be generalized H\"older instead of H\"older, would not work (or at least not {naively}). Indeed, in the second step, approximating a generalized H\"older $F$ with a H\"older one in $L^1(\mu)$ does not help: their integrals w.r.t. the singular measure can be very different.

Finally we note that the title of our work follows the terminology of \cite{ChD09A} where correlation bounds, in case when the initial measure is determined by a standard pair, are called {\it equidistribution properties}.

\section{Preliminaries}\label{sec:prelim}

Below we suppose that Assumptions 1.1, 1.2 and 1.3 hold.

\subsection{Singularities, homogeneity layers, central-stable manifolds}

As well-known, although dispersing billiards are hyperbolic dynamical systems, they possess singularities that necessarily cut the invariant manifolds thus making the mathematical treatment harder. In our case the so-called primary singularities {correspond to grazing collisions}. For the billiard ball map they are
\begin{align*}
S_0 =& \{x = (q, v) \in \mathcal M | \langle n(q),v \rangle = 0\} \\
S_n =& \{x = (q, v) \in \mathcal M | T^{-n}x \in S_0\}\qquad  n \in \IZ
\end{align*}

These set are submanifolds of $\mathcal M$ which can only terminate on each other or on the boundary of $\mathcal M$:
\begin{lemma}\label{lem:sing-continuation}
$\partial S_n \subset \bigcup_{0\le k<n} S_k$.
\end{lemma}

A detailed analysis of singularity curves also providing the proof of the lemma was provided in Sinai's classical paper \cite{S70}.

{In order to control the unbounded expansion in the vicinity of the tangencial collisions, it is
useful to introduce secondary singularities} (cf. \cite{ChM06}, chapter 5, in particular Definitions
5.8 and 5.11). These partition neighbourhoods of primary singularities into so-called homogeneity layers. As a consequence one uses local manifolds (unstable, stable or central-stable ones) as pieces contained in a single layer and call them local homogeneous - unstable, stable or central-stable - manifolds. (For instance, local central-stable manifolds (and their holonomy maps, see Section \ref{sec:holonomy}) will play a central role in our argument.)
Specifically, the phase space $\cM$ is partitioned into homogeneity layers
\begin{eqnarray}
\nonumber \mathbb{H}_k &=& \left\{(q,\varrho)\in \cM \,\mid \, \frac{\pi}{2} -\frac{1}{k^2}\le \varrho< \frac{\pi}{2} -\frac{1}{(k+1)^2}\right\}\ (k>k_0), \\
\nonumber \mathbb{H}_{-k} &=& \left\{(q,\varrho)\in \cM \,\mid \, \frac{\pi}{2} -\frac{1}{k^2}\le -\varrho< \frac{\pi}{2} -\frac{1}{(k+1)^2}\right\}\ (k>k_0), \\
\nonumber \mathbb{H}_{0}&=& \left\{(q,\varrho)\in \cM \,\mid \, |\varrho|\le \frac{\pi}{2} -\frac{1}{(k_0)^2}\right\}.
\end{eqnarray}
where $k_0 > 0$ is an appropriately fixed integer.
Just like in analogous constructions for the billiard map
(see for instance Appendix A in \cite{ChD09A}), we have to introduce the homogeneous manifolds (see Definition~\ref{def:homcsman} below) to guarantee the
required regularity properties.

As said, for  obtaining appropriate distortion control, the boundaries of the homogeneity layers are regarded as artificial (or secondary) singularities.
A \textit{homogeneous local stable manifold} for the map $T:\cM\to\cM$ is a local stable manifold $\gamma$ such that for any $n\ge 0$
$T^n\gamma$ belongs to a single homogeneity strip. Here we define a possible extension for the billiard flow.

\begin{definition}\label{def:homcsman}
A homogeneous local central-stable manifold for the billiard flow $\Phi^t:M\to M$ is a local central-stable manifold $W\subset M$ such that
 $\gamma=\Pi W$ is a homogeneous local stable manifold for the map $T:\cM\to \cM$.
\end{definition}

\subsection{Global constants and regularity parameters}

In the billiards literature, numbers that depend on the billiard table $Q$ only, are often called ``global constants''  and are simply denoted by $C$. Their precise value is usually not important and not studied. However, in some cases it is good to know if such a ``constant'' depends on $Q$ only through some regularity parameters, like bounds on free flight time, scatterer curvature, etc. We will keep track of such dependence. Moreover, the notion of ``unstable curve'' plays a key role in our study. The definition of this notion includes a number of arbitrary choices of regularity constants. These could be chosen as global constants -- i.e. depending on $Q$ only, -- but for the sake of applicability, we will keep track of these choices as well. As a result, what we will call ``constants'' actually depend on both $Q$ and a number of further input parameters. The precise form of the dependence is unimportant, but we record what they depend on.

\subsubsection{Regularity parameters of the billiard table}\label{sec:regpar}

Let $\tau_{min}>0$ and $\tau_{max}>\infty$ be lower and upper bounds for the free flight. Let $\kappa_{min}>0$ and $\kappa_{max}<\infty$ be lower and upper bounds for the curvature of the scatterers. Let $\kappa'_{max}<\infty$ be an upper bound for the derivative of the curvature (as a function of the base point with respect to the arc length parametrization). Let $K_{max}<\infty$ be an upper bound for the number of scatterers. Let $A_{min}$, $A_{max}$ be lower and upper bounds for the area of $Q$ (NB: we think of situations where our result is applied to a family of billiards and we need uniform bounds). Let $d_Q$ be an upper bound for the diameter of the configuration space.
The data
\begin{equation}\label{eq:RQ}
\cR_Q:=\{\tau_{min},\tau_{max},\kappa_{min},\kappa_{max},\kappa'_{max},K_{max},A_{min},A_{max},d_Q\}
\end{equation}
describe the regularity of the billiard table $Q$ for our purposes. In most of our calculations, constants that depend only on the billiard table $Q$, will actually depend on $Q$ only through these regularity parameters. Note that the bounds $\tau_{min},\tau_{max},\kappa_{min},\kappa_{max},\kappa'_{max},K_{max},d_Q$ need not be sharp, so the estimates we give are uniform for the class of billiard tables satisfying the same bounds.

\begin{remark}
So far as we understand the proof of Theorem \ref{thm:BDL} in \cite{BDL16}, in its present form, does not provide that $a'=a'(Q,\alpha)$ and $\cC_{BDL}=\cC_{BDL}(Q,\alpha)$ only depend on $Q$ through $\cR_Q$. However, we believe these are true and it would also be desirable to know them. This is discussed in more detail in Section~\ref{sec:extensions}.
\end{remark}

\subsubsection{Hyperbolicity, cone fields and regularity parameters of $u$-curves}

When studying hyperbolicity of planar dispersing billiards is discrete time, there is a natural choice of stable and unstable cone fields:
$\cone_u(x):=\{(\diff q,\diff \phi)\in T_x \cM\,|\, \diff q\diff \phi>0\}\cup{0}$ for the unstable ones and
$\cone_s(x):=\{(\diff q,\diff \phi)\in T_x \cM\,|\, \diff q\diff \phi<0\}\cup{0}$ for the stable ones. These cone fields are strictly invariant, meaning
that $DT(\cone_u(x))\subset \mathop{int}(\cone_u(Tx))\cap \{ 0\}$ and $DT^{-1}(\cone_s(x))\subset \mathop{int}(\cone_s(T^{-1}x))\cap \{ 0\}$. However, for applications it is often convenient to use smaller cone fields, e.g. by applying the (derivative of the) dynamics to the above. We will stick to the simplest definition above, and use this unstable cone field to define $u$-curves below. Using a  smaller cone field would result in a more restrictive definition of $u$-curves, so our result still applies.

A discrete time $u$-curve or \emph{unstable curve}, is defined in \cite{ChM06}:

\begin{definition}\label{def:discrete_time_u-curve}
 A discrete time $u$-curve is a twice differentiable curve $w\subset \cM$ such that its tangent vector is in the unstable cone $\cone_u(x)$ at any point $x$ belonging to $w$, and its curvature is at most some $B_{max}<\infty$ everywhere.
\end{definition}

$B_{max}=B_{max}(\cR_Q)$ is chosen big enough to make sure that $u$-curves evolve into $u$-curves under the billiard map $T$, apart from being cut by singularities. However, one may want to choose $B_{max}$ bigger than what is necessary for this invariance. We allow that, and record when our ``constants'' may depend on $B_{max}$.

To define a $u$-curve in continuous time, we take a discrete time $u$-curve $w\subset \cM$ and let its points move with the flow for some place-dependent time:

\begin{definition}\label{def:u-curve}
A $u$-curve is a curve $W\subset M$ obtained as
\[W:=\{\Phi^{t(x)}(x)\,|\, x\in w\},\]
where $w$ is a discrete time $u$-curve and the flight time function $t:w\to R^+$ is chosen so that the following regularity properties are satisfied:
\begin{enumerate}[(1)]
 \item $W\subset M$ is also a twice differentiable curve.
 \item The angle of the tangent vector of $W$ with the flow direction is at least some $\alpha_{min}>0$.
 \item The curvature of $W$ is at most some $0<\Gamma_{max}<\infty$ everywhere.
 \item The length of $W$ (to be denoted as $L=L(W^u)$) is at most some $L_{max}\le\frac{1}{100\Gamma_{max}}$.
\end{enumerate}
 \end{definition}
We note that the notion of $u$-curve for the flow in \cite{DN16B} {slightly} differs from ours. Our notion of $u$-curve depends on the data
\[\{B_{max},\alpha_{min},\Gamma_{max},L_{max}\}.\]

For several reasons, we restrict to $u$-curves that are sufficiently far from the scatterers as expressed in the following definition.
For example, the proof of certain regularity properties (discussed in Section~\ref{sec:holonomy}) is easier this way.

\begin{definition}\label{def:gooducurve}
 Fix some $\eps_0\ge 10 L_{max}$. Given a $u$-curve $W^u$, let $dist(W^u,\partial M)$ be its distance from the boundary of $M$. We say that $W^u$ is a good $u$-curve if $dist(W^u,\partial M)>\eps_0$.
\end{definition}

So our notion of good $u$-curve depends on the data
\begin{equation}\label{eq:Ru}
\cR_u:=\{B_{max},\alpha_{min},\Gamma_{max},L_{max},\eps_0\},
\end{equation}
and it is assumed that $10 L_{max}\le\eps_0$.

\subsubsection{Convention on the notation for constants}\label{sec:const-convention}

What we call  ``constants'' or ``global constants'' in this paper, are numbers that depend on $Q$ and $\cR_u$ only. We usually denote these by $C$, often with an index. This index will often refer to the role of the constant, like $C_h$ to the regularity constant of the holonomy in Theorem~\ref{thm:holonomy_dynholder}, of $C_{G;u}$ to the regularity constant of our function $G$ along $u$-curves in Proposition~\ref{prop:G_0-dynHolder}. In other cases, we just use $C_1,C_2,\dots$ to denote different numbers. We also use the unindexed $C$, which may denote different constants at each appearance -- even within a line. In all cases, anything denoted by $C$ depends on $(Q,\cR_u)$ only. Also, some global constants may be denoted by other letters for reasons of tradition, like $\lambda$ for the hyperbolic expansion factor in Theorem~\ref{thm:hyperbolicity}.

Actually, all the global constants that appear in this paper are known to depend on $Q$ only through the regularity parameters $\cR_Q$. (Unfortunately, this is not {known} for the coefficients of the correlation decay estimates denoted by $\cC$ and $a$, see later.)
Also, in many cases, a global constant is known to depend only on $Q$ and not on $\cR_u$. We keep track of these dependences, with future applications in mind. As mentioned above, equidistribution theorems like ours are sometimes applied to a class of models simultaneously, see e.g. \cite{ChD09A}. Then it is important to know if the same bounds hold for all models in the class.

We also note that in the literature of correlation decay, it is common to use the word ``constant'' and the notation $C$ for something which is not a global constant in our sense. A typical example is an exponential correlation decay statement for H\"older observables, of the form
\[Cov(F,G,t)\le C ||F||_\alpha ||G||_\alpha e^{-at}.\]
Here the ``constant'' $C$ does not depend on $F$, $G$ and $t$, but it does depend on $\alpha$ -- i.e. the class of regularity of the observables -- so it is not a global constant.
In this paper, such quantities will not be denoted by $C$, but by $\cC$.

\subsubsection{Hyperbolic properties}\label{sec:hyp-prop}

Below we state three theorems on the hyperbolic properties of the billiard flow. Our main reference for these statements is \cite{ChM06}. Although not formulated exactly as in our Theorems,
\cite{ChM06} contains some estimates from which these properties immediately follow. Below we point out these connections. Note also that similar properties are
discussed in \cite{BDL16}, too. Yet, the discussion of \cite{BDL16} does not literally apply, as our notion of u-curve is different from that of \cite{BDL16} since we allow
variations in the flow direction. Nonetheless, our $u$-curves are uniformly transversal to the flow direction,  see item (2) in Definition~\ref{def:u-curve}.
Accordingly, all the constants that appear in the statements below depend on the class of u-curves $\cR_u$, in particular on the choice of the constant
$\alpha_{min}$.

\begin{theorem}[Uniform hyperbolicity]\label{thm:hyperbolicity}
There are constants $\lambda=\lambda(\cR_Q,\cR_u)>1$ and $c_{hyp}=c_{hyp}(\cR_Q,\cR_u)>0$ such that if $W$ is a $u$-curve, $x,y\in W$ and $\Phi^t$ is smooth on $W$, then
\[dist_{\Phi^t W}(\Phi^t x, \Phi^t y)\ge c_{hyp} \lambda^t dist_{W}(x, y).\]
\end{theorem}

\begin{proof} In this whole section we assume that $F:M\to \IR$ is generalized $\alpha_F$-Hölder continuous (cf. Definition \ref{def:GHC}) and that $\int_M F\diff\mu =0$.
The analogous property for the billiard map is stated in Corollary 4.20 and Formula (4.19) in \cite{ChM06}. As formulated in Theorem~\ref{thm:hyperbolicity},
the statement follows from the definition of $\alpha_{min}$ and the expansion properties of dispersing wave fronts, in particular Formulas (3.35) and (4.9) in
\cite{ChM06}. See also Lemma 3.3 in \cite{BDL16}.
\end{proof}

\begin{theorem}[Transversality]\label{thm:transversality}
There is a constant $c_{tr}=c_{tr}(\cR_Q,\cR_u)>0$ such that any $u$-curve and any central-stable manifold intersecting it have an angle at least $c_{tr}$ at their intersection point.
\end{theorem}

\begin{proof}
The analogous property for the billiard map is stated in Formulas (4.14) and (4.21) in \cite{ChM06}. To discuss the property for the flow, note that central-stable manifolds are
two dimensional. Here we describe two linearly independent directions tangent to central stable manifolds such that the plane they span is uniformly transversal to $u$-curves.
Tangent vectors in the flow phase space are conveniently described in the Jacobi coordinates $(d\xi,d\eta,d\omega)$, see section 3.6 in \cite{ChM06}, or Section~\ref{sec:Ahom} in the present paper.
On the one hand, the flow direction $(0,1,0)$ is tangential to central stable manifolds, and uniformly transversal to u-curves by the definition of $\alpha_{min}$. On the other hand,
the stable direction for the billiard flow is $(1,0,-B_s)$ for some $B_s>0$, that is, stable manifolds are associated with \textit{convergent} wave fronts, see Formula (4.45) in \cite{ChM06}. u-curves,
in turn,  correspond to dispersing wave fronts that have tangent vector $(1,0,B_u)$ for some $B_u>0$. For the claimed uniform transversality, we need to see that the plane $\{(b,a,-b B_s)|(a,b)\in\mathbb{R}^2\}$ is uniformly transversal to the vector $(1,0,B_u)$. This follows as both $B_s$ and $B_u$ are
uniformly bounded away form $0$, and even though neither of these quantities is bounded from above, they cannot approach infinity simultaneously. In particular, $B_u$ can be arbitrarily large
just \textit{after} tangential collisions, while $B_s$ can be arbitrarily large just \textit{before} tangential collisions. See also \cite{BDL16}, in particular Remark 2.1 and the discussion following it.
\end{proof}

To state one more property we need some terminology: the set $\{t\ge 0|\Phi^{-t} S_0\}$, the singularity set for \textit{the flow} $\Phi^t$, $t\ge 0$ is a countable collection of smooth, one-codimensional (i.e.~two dimensional)
submanifolds. We will refer to these submanifolds as singularity manifolds.

\begin{theorem}[Alignment]\label{thm:alignment}
There is a constant $c_{al}=c_{al}(\cR_Q,\cR_u)>0$ such that if $W$ is a $u$-curve, $S_i$ is any singularity manifold which is a pre-image of a tangential collision, and $W$ and $S_i$ intersect, then they have an angle at least $c_{al}$ at their intersection point.
\end{theorem}

\begin{proof}
The argument is essentially the same as in the proof of Theorem~\ref{thm:transversality}. Just like central-stable manifolds, the singularity manifolds are two dimensional, and the flow direction is tangential to them. Furthermore, singularity manifolds can be associated to convergent wave fronts, see Proposition 4.41 (more precisely, its time reversal counterpart) in \cite{ChM06}. That is, for an appropriate choice of $B_{\mathop{Sing}}$, the vector $(1,0,-B_{\mathop{Sing}})$ is tangential to the singularity manifold. Here $B_{\mathop{Sing}}$ is uniformly bounded away from $0$ and can approach infinity just \textit{before} tangential collisions. The required transversality follows as in the case of Theorem~\ref{thm:transversality}.
\end{proof}

\begin{remark}\label{rem:aligncorner}
 Note that, for dispersing billiards with corner points, the above properties are more subtle as there is no lower bound on the free flight.

 However, uniform hyperbolicity and transversality (theorems \ref{thm:hyperbolicity} and \ref{thm:transversality}) extend to dispersing billiards with corner points, under the weaker Assumption~\ref{asm:no-cusp}, without any problem. Alignment is more problematic, because it fails at specific points of the singularity {set} that corresponds to a collision at the corner. However, in Theorem~\ref{thm:alignment} we only claim alignment for tangential singularities. This \emph{does} remain true under Assumption~\ref{asm:no-cusp}, and this is exactly what we need in this paper. (Alignment is used only once, in the proof of Proposition~\ref{prop:G_0-Holder-on-H_1}.) For further details, see section 9 in \cite{Ch99}.
\end{remark}

\subsection{Notions of H\"older continuity}\label{sec:Holder}

\subsubsection{H\"older continuity}

Let $f:X\to\IR$, where $(X,dist)$ is some metric space. Let $0<\alpha\in\IR$ and $0\le C<\infty$. The function $f$ is said to be H\"older continuous with exponent $\alpha$ and H\"older constant $C$ if for any $x,y\in X$
\begin{equation}\label{eq:Holder-def}
|f(x)-f(y)|\le C dist(x,y)^{\alpha}.
\end{equation}
We also say that $f$ is $\alpha$-H\"older with constant $C$, or that $f$ is H\"older with constants $(\alpha,C)$.
The best constant
\[|f|_{\alpha;H}:=\inf\{C\in\IR\,|\, \forall x,y\in X\, |f(x)-f(y)|\le C dist(x,y)^{\alpha} \}\]
is a seminorm on the space of $\alpha$-H\"older functions. Correspondingly, the $\alpha$-H\"older norm of $f$ is
\begin{equation}\label{eq:Holder-norm}
||f||_{\alpha;H}:=\sup_{x\in X} |f(x)| + |f|_{\alpha;H}.
\end{equation}

We will use this notion with $X=M$ or $X=W$ where $W$ is a $u$-curve or a central-stable manifold. We need the following easy quantitative properties:
\begin{lemma}\label{lem:Holder-Holder}
 Let $0<\alpha\le 1$ and let $f,g:X\to\IR$ be $\alpha$-H\"older. Then
 \begin{enumerate}[(i)]
  \item\label{it:Holder-product} $f g$ is also $\alpha$-H\"older and
   \[|fg|_{\alpha;H}\le \sup_X |f| |g|_{\alpha;H} + \sup_X |g| |f|_{\alpha;H}.\]
  \item\label{it:Holder-Holder} If $0<\alpha'\le\alpha$, then $f$ is also $\alpha'$-H\"older and
   \[|f|_{\alpha';H}\le diam(X)^{\alpha-\alpha'} |f|_{\alpha;H}.\]
  \item\label{it:1_over_f-Holder} If $\inf_X f>0$, then $\frac1f$ is also $\alpha$-H\"older and
   \[\left|\frac1f\right|_{\alpha;H} \le \frac{|f|_{\alpha;H}}{\inf_X f^2}.\]
 \end{enumerate}
\end{lemma}

\begin{proof}
 Trivial calculation using only the definition.
\end{proof}

\subsubsection{Generalized H\"older continuity}\label{sec:genholder-def}

Following Keller~\cite{Keller85}, Saussol~\cite{S00} and Chernov~\cite{Ch07}, we generalize the above notion so that (\ref{eq:Holder-def}) need not hold for every pair $(x,y)$, only ``on average'' w.r.t the natural invariant measure $\mu$.

For $x\in M$ and $0\le r\in\IR$, $B_r(x)$ denotes the ball of radius $r$ centred at $x$:
\[B_r(x):=\{y\in M\,:\, dist(x,y)\le r\}.\]
For a function $f:M\to\IR$ we use $(osc_r f):M\to\IR$ to denote its ``$r$ oscillation'':
\begin{equation}\label{eq:osc_r-def}
(osc_r f)(x):=\sup_{y\in B_r(x)} f(y)-\inf_{y\in B_r(x)} f(y).
 \end{equation}
For $0<\alpha\le 1$ the generalized $\alpha$-H\"older seminorm of $f$ is
\begin{equation}\label{eq:gH-seminorm-def}
|f|_{\alpha;gH}:=\sup_{r>0}\frac{1}{r^{\alpha}}\int_M (osc_r f)(x)\diff \mu(x).
\end{equation}

It is easy to see that $osc_r f$ is $\mu$-measurable.

\begin{definition}\label{def:GHC}
Let $f:M\to\IR$ be Borel measurable. We say that $f$ is generalized $\alpha$-H\"older if $|f|_{\alpha;gH}<\infty$.
\end{definition}

\begin{remark}\label{rem:sup_not_ess-sup}
This definition coincides with the one given by Chernov in \cite{Ch07}. It is also similar to what Saussol calls the ``quasi-H\"older property'' in \cite{S00} (which is a special case of the notion defined by Keller in \cite{Keller85}). However, it is not exactly the same. The difference is that Keller~\cite{Keller85} and Saussol~\cite{S00} use \emph{essential} supremum and infimum in the definition (\ref{eq:osc_r-def}) of the oscillation, so their definition does not notice the difference between functions that are equal almost everywhere - w.r.t some distinguished (e.g. Lebesgue) measure.
This is in accordance with using absolutely continuous measures only, when integrating $f$.

From our point of view, two functions, which are equal $\mu$-almost everywhere, may be very different. The measures we use for integration are given by standard pairs, so they are singular w.r.t. $\mu$ -- actually, concentrated on submanifolds. So, for us, the notion of oscillation with the true $\sup$ and $\inf$ is the good one.
\end{remark}

Since $M$ is compact, $f$ can only be generalized $\alpha$-H\"older if it is bounded, so
\begin{equation}\label{eq:var}
var_\alpha(f):=|f|_{\alpha;gH} + \sup_M f - \inf_M f <\infty
\end{equation}
as well. This is still not a norm, since $var_\alpha(const)=0$, but it is the good quantity to measure the regularity of $f$ for the purpose of our statements.

\begin{lemma}\label{lem:holder_genholder}
 If the function $f$ is $\alpha$-H\"older
 then it is also generalized $\alpha$-H\"older and
 \[|f|_{\alpha;gH}\le 2 |f|_{\alpha;H},\]
 \[var_{\alpha}(f) \le 2 |f|_{\alpha;H} +  \sup_M f - \inf_M f \le 2 ||f||_{\alpha;H}.\]
\end{lemma}

\begin{proof}
 This is immediate from the definition. We use that $\mu(M)=1$.
\end{proof}

\begin{lemma}\label{lem:genholder-genholder}
 If $0<\alpha'\le\alpha\le 1$ and $f:M\to\IR$ is generalized $\alpha$-H\"older, then it is also generalized $\alpha'$-H\"older and
 \[var_{\alpha'}(f) \le \cC(\cR_Q,\alpha,\alpha') var_\alpha(f)\]
 where $\cC(\cR_Q,\alpha,\alpha')=\max\{(diam(M))^{\alpha-\alpha'},1\}$.
\end{lemma}

\begin{proof}
Setting $R:=diam(M)$, clearly $osc_r f\le osc_R f$, so
\[|f|_{\alpha;gH}=\sup_{r>0}\frac{1}{r^{\alpha}}\int_M (osc_r f)(x)\diff \mu(x)=\sup_{0<r<R}\frac{1}{r^{\alpha}}\int_M (osc_r f)(x)\diff \mu(x).\]
So
\[|f|_{\alpha';gH}=\sup_{0<r<R}r^{\alpha-\alpha'}\frac{1}{r^{\alpha}}\int_M (osc_r f)(x)\diff \mu(x)\le R^{\alpha-\alpha'}|f|_{\alpha;gH}.\]
Now this implies
\[var_{\alpha'}(f):=|f|_{\alpha';gH} + \sup_M f - \inf_M f\le R^{\alpha-\alpha'} |f|_{\alpha;gH} + \sup_M f - \inf_M f \le \max\{R^{\alpha-\alpha'},1\}var_\alpha(f).\]
\end{proof}

\subsubsection{Dynamical H\"older continuity}

The billiard ball map can be extended to $M$ in a natural way: $T:M\to\mathcal{M}$, $T(x):=\Phi^{\tau(x)}(x)$.

\begin{definition}\label{def:sep-time}
Let $W$ be a $u$-curve. For any $x,y\in W$, their \emph{separation time} $s^+(x,y)$ is the smallest $n>0$ for which $T^n$ is not continuous on the subcurve of $W$ connecting $x$ and $y$.
\end{definition}

This is a convenient definition to describe the partitioning of $W$ during its time evolution due to the presence of singularities, but we note that $s^+(x,y)$ does not depend on the $u$-curve $W$ connecting them, as long as they can be connected with some $u$-curve. Also note that this separation time is connected to the discrete time billiard map. In this paper, no continuous version of the separation time will be used.

Since $W$ is expanded by $T$, $s^+(x,y)<\infty$ for every $x\neq y\in W$. If we fix some $\vartheta<1$, then
\[d_\vartheta(x,y):=\vartheta^{s^+(x,y)}\]
is a metric on $W$, and is called the dynamical distance.

In this section we concentrate on functions $f$ that are not defined on all of $M$, but instead, on a $u$-curve $W$ only.
Then the notion of dynamical distance makes sense for every $x,y\in W$, and the usual notion of H\"older continuity, w.r.t. this metric, is called \emph{dynamical H\"older continuity}: $f:W\to\IR$ is called dynamically H\"older continuous if there are constants $0\le C<\infty$ and $0<\alpha\le 1$, such that
\[|f(x)-f(y)|\le C d_\vartheta(x,y)^\alpha\]
for every $x,y\in W$. Since $d_\vartheta(x,y)^\alpha=(\vartheta^{s^+(x,y)})^\alpha = \theta^{s^+(x,y)}$ with $\theta:=\vartheta^\alpha$, the notion of dynamical H\"older continuity is independent of the choice of $\vartheta$ -- or, in other words, only the power $\theta:=\vartheta^\alpha$ has physical meaning. This justifies the following formal definition:
\begin{definition}\label{def:dHc}
 Let $W$ be a $u$-curve, $0<\theta<1$ and $C<\infty$. The function $f:W\to \IR$ is dynamically H\"older with constants $(C,\theta)$ (or dynamically $\theta-$H\"older) if for any $x,y\in W$
 \[|f(x)-f(y)|\le C \theta^{s^+(x,y)}.\]
 The dynamical $\theta$-H\"older seminorm of $f$ is defined as the best constant
 \begin{equation}\label{eq:dHn}
 |f|_{\theta;dH}:=\inf\{C\in\IR\,|\, \forall x,y\in W \,|f(x)-f(y)|\le C \theta^{s^+(x,y)} \}.
 \end{equation}
 The dynamical $\theta$-H\"older norm of $f$ is defined as
 \begin{equation}\label{eq:dynHolder-norm}
 ||f||_{\theta;dH}:=|f|_{\theta;dH} + \sup_W f.
 \end{equation}
\end{definition}

A few easy properties:

\begin{lemma}\label{lem:dynHolder-product}
 Let $W$ be a $u$-curve, $0<\theta<1$ and $f,g:W\to\IR$ dynamically $\theta$-H\"older. Then
 \begin{enumerate}[(i)]
  \item\label{it:dynHolder-product} $f g$ is also dynamically $\theta$-H\"older and
   \[|fg|_{\theta;dH}\le \sup_W |f| \cdot |g|_{\theta;dH} + \sup_W |g| \cdot |f|_{\theta;dH}.\]
  \item\label{it:dynHolder-dynHolder} If {$\theta'<\theta<1$}, then $f$ is also dynamically $\theta'$-H\"older and
   \[|f|_{\theta';dH}\le |f|_{\theta;dH}.\]
  \item\label{it:1_over_f-dynHolder} If $\inf_W f>0$, then $\frac1f$ is also dynamically $\theta$-H\"older and
   \[\left|\frac1f\right|_{\theta;dH} \le \frac{|f|_{\theta;dH}}{\inf_W f^2}.\]
 \end{enumerate}
\end{lemma}

\begin{proof}
 Trivial calculation using only the definition.
\end{proof}

A relation between H\"older continuity and dynamical H\"older continuity is shown in the following lemma.

\begin{lemma}\label{lem:Holder-dynHolder}
If $W$ is a $u$-curve, $0<\alpha\in\IR$ and $f:W\to\IR$ is H\"older continuous with exponent $\alpha$, then $f$ is also dynamically H\"older continuous with some constant $\theta=\theta(\alpha,\cR_Q,\cR_u)$ and
\[|f|_{\theta;dH} \le \cC(\alpha,\cR_Q,\cR_u) |f|_{\alpha;H}.\]
 \end{lemma}

\begin{proof}
We only need to estimate the Euclidean distance with the dynamical distance from above.
The dynamics is uniformly expanding along $u$-curves by uniform hyperbolicity (Theorem~\ref{thm:hyperbolicity}), but the length of possible smooth components of $T^n W$ is bounded: indeed, they are increasing curves in the discrete time phase space viewed as (\ref{eq:map-phasespace-with-position-and-angle}), and must terminate on negative time singularities. So $|T^n W|\le \Cl{const:H-dynH-1}(\cR_Q,\cR_u)$, and points $x$ and $y$ with a long separation time $s^+(x,y)$ have to be close: if $n<s^+(x,y)$, then
\[\Cr{const:H-dynH-1}\ge |T^n W|\ge dist_{T^n W} (T^n x,T^n y)\ge c_{hyp}\lambda^{n}dist_{W}(x,y).\]
Using this with $n=s^+(x,y)-1$, we get
\[dist_{W}(x,y)\le \frac{\Cr{const:H-dynH-1} \lambda}{c_{hyp}} \left(\frac{1}{\lambda}\right)^{s^+(x,y)}.\]
Choosing
\[\theta=\theta(\alpha,\cR_Q,\cR_u):=\left(\frac{1}{\lambda}\right)^\alpha<1 \quad , \quad \cC=\cC(\alpha,\cR_Q,\cR_u):=\left(\frac{\Cr{const:H-dynH-1} \lambda}{c_{hyp}}\right)^\alpha\]
this becomes
\[dist_{W}(x,y)^\alpha\le \cC \theta^{s^+(x,y)}.\]
So since $f$ is $\alpha$-H\"older,
\[|f(x)-f(y)|\le |f|_{\alpha;H} dist(x,y)^\alpha \le |f|_{\alpha;H}dist_{W}(x,y)^\alpha \le |f|_{\alpha;H} \cC \theta^{s^+(x,y)}.\]
This is exactly the statement to prove.
\end{proof}

A comparison in the other direction is not so easy: dynamical H\"older continuity obviously doesn't imply H\"older continuity -- in fact, it doesn't even imply continuity. However, a dynamically H\"older function can be made H\"older at the cost of modifying it on a small set. We will use this in our construction (in Proposition~\ref{prop:G_0-Holder-on-H_1}).

\begin{lemma}\label{lem:sep-time-constant-on-W^cs}
 If $W^u_1$ and $W^u_2$ are $u$-curves, $h:H\subset W^u_1\to W^u_2$ is the holonomy map along central-stable manifolds, $x,y\in H$ and $s^+$ is the separation time from Definition~\ref{def:sep-time}, then
 \[s^+(h(x),h(y))=s^+(x,y).\]
\end{lemma}

\begin{proof}
 $x$ and $h(x)$ are on the same central-stable manifold, so for every discrete time step of the collision map, their images are on the same continuity component of the map. Similarly for $y$ and $h(y)$. So $x$ and $y$ are separated exactly when $h(x)$ and $h(y)$ are separated.
\end{proof}

\subsection{Holonomy along central-stable manifolds}\label{sec:holonomy}

We state some regularity properties of homogeneous local central stable-manifolds and their holonomy
{which will pay a crucial role in our argument.}
In particular, we consider $u$-curves in the sense of Definition~\ref{def:u-curve}, and we want to ensure that there are many homogeneous local central-stable manifolds that connect these $u$-curves, and that the holonomy maps obtained by sliding along the central stable manifolds are sufficiently regular. For this purpose, we only consider \emph{good} $u$-curves in the sense of Definition~\ref{def:gooducurve}. Analogous properties in the map context are discussed in \cite{ChM06}, Chapter 5; hence our task is to reduce the flow statements to the map statements. The proofs are given in Section~\ref{sec:Ahom}.

For any $x\in M$, let $r^{c-s}(x)$ denote the inner diameter of the homogeneous local central stable manifold of $x$ -- meaning the supremum of the radii of those disks in the homogeneous local central-stable manifold, centred at $x$, which fit completely into the manifold without reaching its boundary.

\begin{theorem}\label{thm:long_cs_manifolds} There is a constant $C_W=C_W(\cR_Q,\cR_u)$ such that for any good $u$-curve $W^u$ and any $0<\eps$,
\[m_{W^u}(\{x\in W^u\,|\; r^{c-s}(x)\le \eps\})\le C_W \eps.\]
\end{theorem}

Consider two good $u$-curves $W^u_1$ and $W^u_2$ sufficiently close to each other. For $H_1\subset W^u_1$ denote by $h:H_1\to W^u_2$ the holonomy map defined with \emph{homogeneous central-stable} manifolds and let $H_2=h(H_1)$. Let finally $m_{H_1}$ and $m_{H_2}$ denote the corresponding arc length measures. Our next theorem claims that the holonomy maps between good $u$-curves along homogeneous central-stable manifolds are absolutely continuous with uniformly bounded and dynamically H\"older continuous densities.

\begin{theorem}\label{thm:holonomy_dynholder}
There exist constants $C_h=C_h(\cR_Q,\cR_u)<\infty${, $\Theta_h=\Theta_h(\cR_Q,\cR_u)<1$ and} a function $Jh:H_1\to\IR^+$ such that for any Borel $B\subset H_1$,
\begin{equation}\label{eq_Jh-def}
m_{H_2}(h(B))=\int_B Jh \diff m_{H_1}.
\end{equation}
Furthermore, this $Jh$ satisfies that for any $x,y\in H_1$
\[
|Jh(x)|\le C_h
\]
and
\[
|Jh(x)-Jh(y)|\le C_h \Theta_h^{s^+(x,y)}
\]
where $s^+(x,y)$ is the separation time of $x$ and $y$.
The function $Jh$ being, in fact, the Radon-Nikodym derivative of $m_{H_2}\circ h$ w.r.t. $m_{H_1}$  is called the Jacobian of the holonomy.
\end{theorem}

We will need one more property to ensure that the Jacobian of the holonomy varies sufficiently regularly along the central stable direction, as formulated in the following theorem.

\begin{theorem}\label{thm:holonomy_cs-regularity}
Let $W_1$ and $W_2$ be two good $u$-curves, and $p_1\in W_1$ and $p_2\in W_2$ two points on them that lie on the same homogeneous central stable manifold $W_{cs}$. We introduce the following quantities:
 \begin{itemize}
 \item $\delta$, the distance of $p_1$ and $p_2$ along $W_{cs}$ (in the natural Riemannian metric on $W_{cs}$ as a submanifolds of $M$);
 \item $\alpha$, the angle of the tangent vectors $T_{p_1}W_1$ and $T_{p_2}W_2$,
 \item $J_h(p_1)$, the Jacobian of the holonomy along central stable manifolds from $W_1$ to $W_2$, evaluated at $p_1$.
 \end{itemize}
 With these notations, there is a constant $C_{h2}=C_{h2}(\cR_Q,\cR_u)$ such that
 \[
 |J_h(p_1)-1|\le C_{h2}(\alpha+\delta^{1/3}).
 \]
\end{theorem}

\subsection{Standard pairs}\label{sec:stpair}

The method of standard pairs was introduced in \cite{D05}. Somewhat later in \cite{Ch06} it was already utilized for obtaining various stochastic properties of dispersing billiards whereas  in \cite{ChD09A} Chernov and Dolgopyat presented far-reaching novel applications of the method. A standard pair is a $u$-curve $W$ and a probability density $\phi$ on $W$, with certain regularity properties. It can be pictured as a measure on the phase space which is concentrated on a single $u$-curve -- i.e. it is highly singular w.r.t. Riemannian volume. The precise notion to use may depend on the application, and there are many slightly different versions in the literature.

In applications, time evolution of standard pairs plays a crucial role, so the regularity properties required are such that standard pairs evolve into standard pairs in some sense. In the present work we will not consider such a time evolution: instead, we will approximate a standard pair with an absolutely continuous measure. So, for the purpose of applicability, we choose the notion of standard pair as general as we can.

\begin{definition}\label{def:stp}
A standard pair is a pair $(W,\phi)$ where
\begin{itemize}
 \item $W\subset M$ is a $u$-curve in the sense of Definition~\ref{def:u-curve}.
 \item $\phi:W\to\IR^+$ is a probability density function w.r.t. arc length on $W$.
 \item $\phi$ is dynamically H\"older continuous with some constants $C_\phi<\infty$ and $\theta_\phi<1$
 \end{itemize}
\end{definition}

\section{Proof of the main theorem}\label{sec:proof}

In the whole section we assume that  $F:M\to \IR$ is generalized $\alpha_F$-Hölder continuous (cf. Definition \ref{def:GHC}) and that $\int_M F\diff\mu = 0$.

The proof is based on Theorem \ref{thm:BDL}, i.e. on Corollary 1.3 of Theorem 1.2 from \cite{BDL16}.

We will prove our theorem by approximating the singular measure $\tilde\phi$ concentrated on $W^u$ with an absolutely continuous one, which has some H\"older continuous density $G$ w.r.t. $\mu$. The choice of $G$ will depend on $t$, which we sometimes emphasize by writing $G=G_t$.

Our approximating density $G_t$ will be supported on some $\eps$-neighbourhood of $W^u$, where $\eps=\eps(t)\le\eps_0$ will be specified later. The delicate construction of $G$ uses two foliations of this $\eps$-neighbourhood, which correspond to some kind of ``product structure'', at least for the vast majority of points. Both foliations will be measurable partitions w.r.t. $\mu$, and the regularity properties of the factor and conditional measures will play a crucial role -- although not always exploited formally.

This approximate product structure, with some of the notation, is shown in Figure~\ref{fig:product_structure}.
\begin{figure}[hbt]
 \psfrag{eps}{$\eps$}
 \psfrag{W^s_x}{$W^{c-s}_x$}
 \psfrag{h_z(x)}{$h_z(x)$}
 \psfrag{W^s_y}{$W^{c-s}_y$}
 \psfrag{h_z(y)}{$h_z(y)$}
 \psfrag{W^u_z}{$W^u_z$}
 \psfrag{W^u}{$W^u$}
 \psfrag{x}{$x$}
 \psfrag{y}{$y$}
 \psfrag{z}{$z$}
 \psfrag{D}{$D$}
 \psfrag{L}{$L$}
 \centering
 \includegraphics[bb=0 0 549 229, width=15cm]{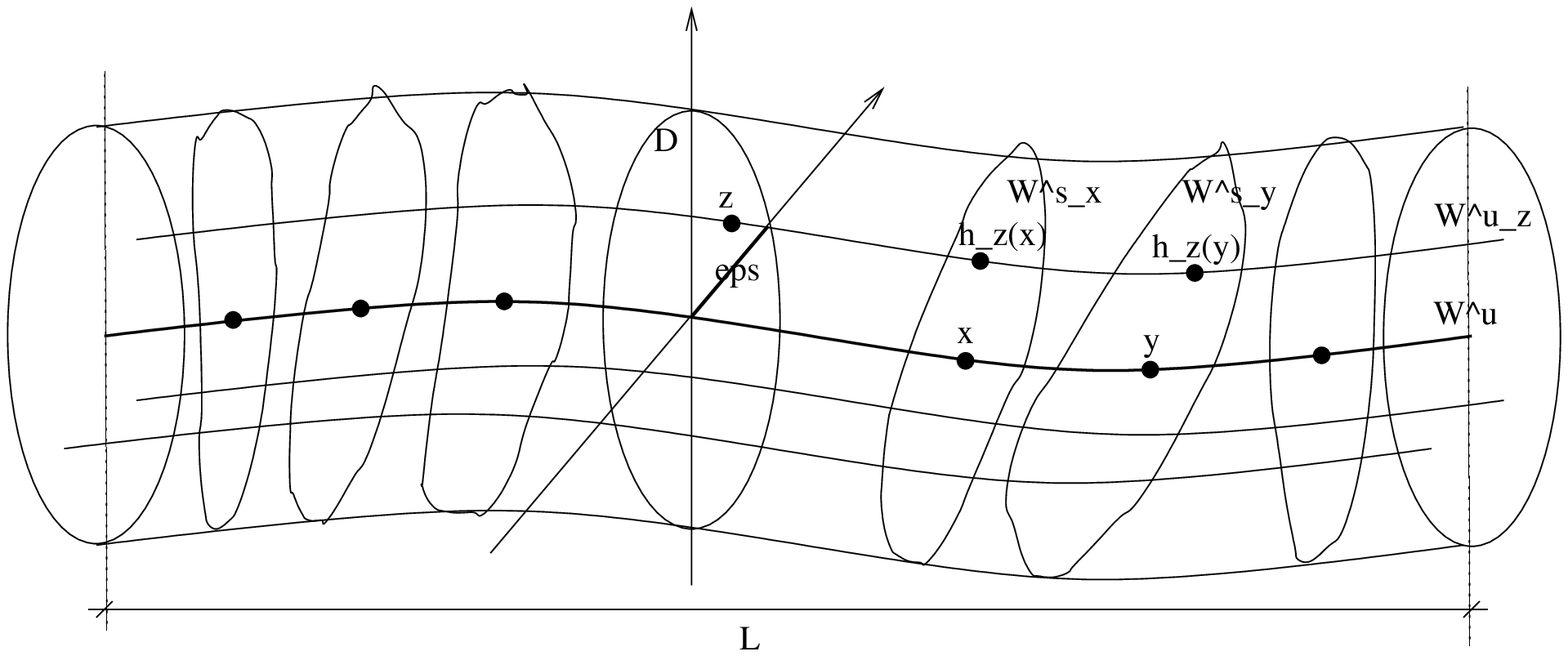}
 \caption{Product structure of a neighbourhood of a $u$-curve}
 \label{fig:product_structure}
\end{figure}

\subsection{Conditional measures}\label{sec:conditional}

During the construction, we will use foliations of (subsets of) the phase space with lower dimensional submanifolds. Such a foliation can also be viewed as a partition into subsets indexed with some index set. We use the notion of \emph{measurability} of such a partition in the usual sense:

\begin{definition}\label{def:kernel}
Let $(X_1,\sigalg_1)$ and $(X_2,\sigalg_2)$ be measurable spaces. The function $k:X_1\times \sigalg_2 \to [0,\infty]$ is called a kernel (form $X_1$ to $X_2$) if $k(x,.):\sigalg_2\to [0,\infty]$ is a measure for every $x\in X_1$ and $k(.,A):X_1\to[0,\infty]$ is measurable for every $A\in \sigalg_2$. It is called a probability kernel if also $k(x,X_2)=1$ for every $x\in X_1$.
\end{definition}

\begin{definition}\label{def:measure-kernel-composition}
Let $(X_1,\sigalg_1)$ and $(X_2,\sigalg_2)$ be measurable spaces, $\nu$ a measure on $(X_1,\sigalg_1)$ and $k:X_1\times \sigalg_2 \to [0,\infty]$ a kernel. The composition of $\nu$ and $k$ is the measure $\nu\otimes k$ on $(X_2,\sigalg_2)$ defined by
\[(\nu\otimes k)(B):=\int_{X_1} k(x,B)\diff \nu(x)\]
for every $B\in\sigalg_2$.
\end{definition}

\begin{definition}\label{def:measurable_partition}
Let $X_1$, $X_2$ be measurable spaces and $X_2=\mathop{\dot{\bigcup}}_{i\in X_1} E_i$ a partition of $X_2$. We say that the partition is measurable w.r.t. the measure $\mu$ on $X_2$ if there is a measure $\nu$ on $X_1$ and a kernel $k$ from $X_1$ to $X_2$ such that $\mu=\nu\otimes k$ and each measure $k(i,.)$ is concentrated on $E_i$. Then $\nu$ is called factor measure, and the measures $k_i(.):=k(i,.)$ are called conditional measures.
\end{definition}

If a partition is measurable, the factor measure and conditional measures are not unique. However, if $\mu$ is finite, then there is a canonical choice. Let $\pi:X_2\to X_1$ be the natural projection defined by $\pi(x)=i$ for $x\in E_i$, so $\pi^{-1}A=\cup_{i\in A}E_i$ for $A\subset X_1$. Then $\nu$ can be chosen to be the push-down of $\mu$ to $X_1$ by $\pi$:
\[\nu(A)=\pi_* \mu(A)=\mu(\pi^{-1}A)\]
for every measurable $A\subset X_1$. With this choice
\[\nu(A)=\mu(\pi^{-1}A)=(\nu\otimes k)(\pi^{-1}A)=\int_{X_1} k(x,\cup_{i\in A}E_i)\diff\nu(i)=\int_A k(i,X_2)\diff \nu(i),\]
so $k$ becomes a probability kernel.

In our discussion we always work with finite measures. This justifies the following convention:

\begin{convention}\label{conv:conditional-measure_always_probability}
When decomposing a measure into factor measure and conditional measures, we always choose the factor measure to be the push-down by the natural projection, and the conditional measures to be probabilities.
\end{convention}

\subsection{u-foliation}\label{sec:ufoliation}

First, we construct a foliation of a tube around $W^u$, about $\eps$ thick ($\eps < \eps_0$), with $1$-dimensional curves which are ``nearly parallel'' to $W^u$.
By ``nearly parallel'' we mean that if two points are sufficiently close, then the tangent vectors of the curves in those points are guaranteed to be arbitrarily close, even if the two points are not on the same curve.

This foliation can be chosen to have very nice regularity properties. Set $D\subset \IR^2$ to be the disk of radius $\eps$ centred at $0$. We place this disk orthogonally to $W^u$ at one of its points (say, the centre point of $W^u$), such that $D$ is centred at this intersection point.
Then, for any $z\in D$, $W^u_z$ will be the copy of $W^u$ shifted with $z$ in Euclidean space: $W^u_z=W^u+z$. Now we set $U:=\bigcup_{z\in D} W^u_z$ to be the tube formed by these shifted versions of $W^u$. The foliation $\{W^u_z\}_{z\in D}$ of $U$ has nice regularity properties. In particular, the partition is measurable, and the conditional measures on the $W^u_z$
are smooth. However, since $W^u$ is not (necessarily) straight, these conditional measures do not coincide with arc length, and we have to be careful about their precise regularity.

The specific form of the construction above is not important -- other smooth foliations would also do. However, with this special choice, some calculations can be simplified by introducing the following notation:

\begin{notation}\label{not:alpha_and_L}
Let $L$ be the length of the orthogonal projection of $W^u$  to the normal vector of $D$. Let $m_D$ be Lebesgue measure on $D$.
For $r\in U$ let $\psi(r)$ be the angle of the tangent vector of $W^u_z$ at $r$ with the normal vector of $D$.
\end{notation}

With this notation, $Leb(U)=L m_D(D)$ and $\mu(U)=\frac{L}{Leb(M)} m_D(D)$. The properties we need are summarized in the following lemma.

\begin{lemma}\label{lem:u-foliation}
 The foliation $\{W^u_z\}_{z\in D}$ of $U:=\bigcup_{z\in D} W^u_z$ has the following properties:
\begin{enumerate}
 \item $W^u=W^u_0$.
 \item $U$ is contained in the  $\eps$-neighbourhood of $W^u$.
 \item The partition $\{W^u_z\}_{z\in D}$ is measurable w.r.t. the invariant measure $\mu$ restricted to $U$ (i. e. to $\mu(.) = \mu(. \cap U)$), in the sense of Definition~\ref{def:measurable_partition}. Denote the factor measure on $D$ by $\mu^{factor}_D$ and the conditional (probability) measures on the foliation leaves $W^u_z$ by $\mu^{cond}_{W^u_z}$.

 \item $\mu^{factor}_D=\frac{L}{Leb(M)} m_D$.
 \item $\mu^{cond}_{W^u_z}$ is absolutely continuous w.r.t. arc length, with density
   \begin{equation}\label{eq:mu_cond_formula}
    \frac{\diff \mu^{cond}_{W^u_z}}{\diff m_z}(r)=\frac1L \cos\psi(r).
   \end{equation}
 \end{enumerate}
\end{lemma}

\begin{proof}
 All items are obvious from the construction. The factors $L$ and $\frac1L$ are in accordance with Convention~\ref{conv:conditional-measure_always_probability}.
\end{proof}

For a technical reason (mainly of convenience) we need the following

\begin{lemma}\label{lem:W-u-z_is_good_u-curve}
For every $z\in D$, $W^u_z$ is a good $u$-curve (as in Definition~\ref{def:gooducurve}).
\end{lemma}

\begin{proof}
By definition, the $u$-curve $W^u$ is obtained from a discrete time $u$-curve $w^u\subset\cM$ by lifting to $M$ with the flow.
Since $W^u$ is at least $2\eps_0$ far from $\partial M$, the shifted version $W^u_z$ is still at least $\eps_0$ far, so it is good, provided that it is a $u$-curve.

Now, by shifting $W^u$ into $W^u_z$, the curvature does not change, but the trace on $\cM$ does: $W^u_z$ is obtained from some $w^u_z$ in the same way as $W^u_z$ is obtained from $w^u$. We need that this $w^u_z$ is still a discrete time $u$-curve. But, strictly speaking, $w^u_z$ is not a shifted version of $w^u$, so the tangent vectors and curvature can change slightly.
At this point, being absolutely precise would result in overly complicated notation and no real ideas presented. One could introduce a more restricted class of $u$-curves into which $w^u$ is requested to belong (smaller cones, smaller curvature bound), and a less restricted class into which the nearby $w^u_z$ falls automatically if $\eps_0$ -- and thus $|z|$ -- is small enough. We omit these details.
\end{proof}

As mentioned in Section~\ref{sec:conditional}, the choice of the normalization for the factor measure and the conditional measures is somewhat arbitrary. In our choice we follow Convention~\ref{conv:conditional-measure_always_probability}. However, other choices of normalization would also be possible, and in some sense, only the ``product of the two'' has a physical meaning. Accordingly, what we will really use is the following immediate corollary of Lemma~\ref{lem:u-foliation}:

\begin{lemma}\label{lem:u-density-times-factor-density}
For any $z\in D$ and any $r \in W^u_z$
\[
\frac{\diff\mu^{factor}_D}{\diff m_D}(z) \frac{\diff \mu^{cond}_{W^u_z}}{\diff m_z}(r) = \frac{1}{Leb(M)} \cos\psi(r).
\]
\end{lemma}

These functions of $r$ on the different $W^u_z$ are obviously uniformly bounded and Lipschitz continuous:

\begin{lemma}\label{lem:u-density-regularity}
For any $z\in D$ and any $r_1,r_2\in W^u_z$
\[ \frac12 \le \cos\psi(r_1) \le 1 \]
and
\[ |\cos\psi(r_1)-\cos\psi(r_2)|\le \Gamma_{max} dist_{W^u_z}(r_1,r_2). \]
\end{lemma}

\begin{proof} Each $W^u_z$ is a shifted version of $W^u$, so it is obviously enough to prove the statement for $z=0$ (meaning $r_1,r_2\in W^u_0=W^u$). $W^u$ is a $u$-curve, so by Definition~\ref{def:u-curve} its curvature is at most $\Gamma_{max}$ and its length is at most $L_{max}\le\frac{1}{100\Gamma_{max}}$. So
\[
|\cos\psi(r_1)-\cos\psi(r_2)| \le |\psi(r_1)-\psi(r_2)| \le \Gamma_{max} dist_{W^u_z}(r_1,r_2).
\]
Choosing $r_2$ to be the point of $W^u_z$ that lies on $D$, we have $\psi(r_2)=0$, so
\[
|\cos\psi(r_1)-1| \le \Gamma_{max} L_{max} \le \frac{1}{100},
\]
so $\cos\psi(r_1)\ge \frac12$.
\end{proof}

This has easy consequences:

\begin{lemma}\label{lem:W-u_z-s_not_close}
 For any $z_1,z_2\in D$, $dist(W^u_{z_1},W^u_{z_2})\ge \frac12 |z_2-z_1|$.
\end{lemma}

\begin{proof}
 Trivial geometry using $\cos\psi\ge \frac12$.
\end{proof}

\begin{lemma}\label{lem:projection_to_W-u-not-short}
 Let $x_1,x_2\in W^u$, let $z_1,z_2\in D$ and let $r_1:=x_1+z_1$, $r_2:=x_2+z_2$ (so $r_1,r_2\in U$). Then $dist(r_1,r_2)\ge \frac12 dist_{W^u}(x_1,x_2)$.
\end{lemma}

\begin{proof}
 Trivial geometry using $\cos\psi\ge \frac12$.
\end{proof}

\begin{lemma}\label{lem:psi-Lipschitz}
 $\psi$ is Lipschitz continuous on $D$ with constant $2\Gamma_{max}$. That is, for any $r_1,r_2\in U$,
 \[|\psi(r_2)-\psi(r_1)|\le 2 \Gamma_{max} dist(r_1,r_2).\]
\end{lemma}

\begin{proof}
 Let $r_1:=x_1+z_1$, $r_2:=x_2+z_2$ where $x_1,x_2\in W^u$ and $z_1,z_2\in D$. Then, by construction, $\psi(r_1)=\psi(x_1)$ and $\psi(r_2)=\psi(x_2)$, so
 \[|\psi(r_2)-\psi(r_1)|=|\psi(x_2)-\psi(x_1)|\le\Gamma_{max} dist_{W^u}(x_1,x_2).\]
 Using Lemma~\ref{lem:projection_to_W-u-not-short} we get the statement.
\end{proof}

\subsection{Central-stable foliation}\label{sec:csfoliation}

The other foliation we use to get the product structure of $U$ consists of (homogeneous) central-stable manifolds. This is crucial, since we will use in our calculations that the points on such a manifold stay close to each other for any long time. For this, a high price has to be paid: this foliation has much worse regularity properties than the u-foliation above, and not every point can be covered with the product structure.

We will use $H$ to denote the set of those points in $W^u$ whose central-stable manifold crosses $U$ properly. That is, the definition will ensure that for every $x\in H$
\begin{itemize}
 \item The central-stable manifold of $x$ is long enough in every direction (meaning $r^{c-s}(x)$ is big enough) so that it surely reaches the boundary of $U$.
 \item The central-stable manifold of $x$ should not hit the circular faces of the tube $U$, but rather cross $W^u_z$ for every $z\in D$.
\end{itemize}

\begin{notation}
\begin{equation}\label{eq:H-def}
H:=\left\{x\in W^u \,\middle|\, r^{c-s}(x)\ge \frac{10\eps}{\sin c_{tr}} \text{ and } W^{c-s}(x)\cap W^u_z\neq \emptyset \text{ for every $z\in D$} \right\}.
\end{equation}
\end{notation}

It is easy to see that if $\eps$ is small, then $H$ contains the vast majority of points in $W^u$:

\begin{lemma}\label{lem:W-minus-H-small} There is a global constant $\Cl{const:W-minus-H-small}<\infty$ such that
 $m_{W^u}(W^u\setminus H)\le \Cr{const:W-minus-H-small} \eps$.
\end{lemma}

\begin{proof}
Due to transversality (Theorem~\ref{thm:transversality}), $u$-curves and central-stable manifolds have an angle at least $c_{tr}$, so if $r^{c-s}(x)\ge \frac{10\eps}{\sin c_{tr}}$ then $W^{c-s}(x)$ either crosses every $W^u_z$, or it hits one of the circular faces at the end of the tube. This latter can only happen if $x$ is less than $\frac{11\eps}{\sin c_{tr}}$ away from one of the endpoints of $W^u$. On the other hand, not reaching far enough also only happens to a small set of $x$ due to Theorem~\ref{thm:long_cs_manifolds}.
\end{proof}

However, we need a little more than that.

\begin{notation}
From now on, for every $x\in U$ we use $W^{c-s}_x$ to denote the central-stable manifold of $x$, \emph{intersected with $U$}.
\end{notation}

So
\[U_0:=\cup_{x\in H} W^{c-s}_x\subset U\]
is the set of points in $U$ that can be covered with these properly crossing central-stable manifolds.
For any $z\in D$, set $H_z:=W^u_z\cap U_0$. Actually, this $H_z$ is nothing else than $H$ shifted from $W^u$ to $W^u_z$ by the holonomy map. For future use, we introduce two notations for this holonomy:
\begin{notation}\label{not:holonomy}
 Let $h:H\times D\to U$ be defined so that for any $x\in H$ and $z\in D$ $h(x,z)$ is the single element of $W^{c-s}_x \cap W^u_z$. We will also use the notation $h_z(x):=h(x,z)$, so $h_z:H\to W^u_z$ is the usual holonomy map.
\end{notation}

With this notation, $U_0=h(H\times D)$ and $H_z=h_z(H)$.

\begin{lemma}\label{lem:W_z-minus-H_z-small}
 There is a global constant $\Cl{const:W_z-minus-H_z-small}<\infty$ such that for any $z\in D$
 $m_{W^u_z}(W^u_z\setminus H_z)\le \Cr{const:W_z-minus-H_z-small} \eps$.
\end{lemma}

\begin{proof}
Just like before. Due to transversality (Theorem~\ref{thm:transversality}), $u$-curves and central-stable manifolds have an angle at least $c_{tr}$, so if $y\in W^u_z$ and $r^{c-s}(y)\ge\frac{20\eps}{\sin c_{tr}}$, then the central-stable manifold of $y$ either crosses $U$ properly, or it hits one of the circular faces at the end of the tube. This latter can only happen if $y$ is less than $\frac{22\eps}{\sin c_{tr}}$ away from one of the endpoints of $W^u_z$. Not reaching far enough also only happens to a small set of $y$ due to Lemma~\ref{lem:W-u-z_is_good_u-curve} and Theorem~\ref{thm:long_cs_manifolds}. Now if the central-stable manifold of $y\in W^u_z$ crosses $U$ properly, then it also intersects $W^u$ at some $x\in W^u$, so it has to coincide with the central-stable manifold of this $x$ and surely $r^{c-s}(x)\ge \frac{10\eps}{\sin c_{tr}}$. So $x\in H$ and $y\in U_0$.
\end{proof}

We only use this through the following immediate corollary:
\begin{lemma}\label{lem:U-minus-U0-small}
 There is a global constant $\Cl{const:U-minus-U0-small}<\infty$ such that
 $\mu(U\setminus U_0)\le \Cr{const:U-minus-U0-small} \eps^3$.
\end{lemma}

\begin{proof}
By Lemma~\ref{lem:u-density-times-factor-density}, since $\cos\psi(r)\le 1$,
\begin{eqnarray*}
 \mu(U\setminus U_0) &=& \int_D \mu^{cond}_{W^u_z} (W^u_z\setminus H_z)\diff \mu^{factor}_D(z) = \\
 &=& \int_D \int_{W^u_z\setminus H_z} \frac{\diff \mu^{cond}_{W^u_z}}{\diff m_{W^u_z}}(r) \diff m_{W^u_z}(r) \frac{\diff \mu^{factor}_D}{\diff m_D}(z) \diff m_D(z) = \\
 &=& \int_D \int_{W^u_z\setminus H_z} \frac{1}{Leb(M)} \cos\psi(r) \diff m_{W^u_z}(r) \diff m_D(z) \le \\
 &\le& \frac{1}{Leb(M)} \int_D m_{W^u_z} (W^u_z\setminus H_z) \diff m_D(z).
\end{eqnarray*}
So by Lemma~\ref{lem:W_z-minus-H_z-small}
 \[ \mu(U\setminus U_0) \le \frac{1}{Leb(M)} \int_D \Cr{const:W_z-minus-H_z-small} \eps \diff m_D(z) =
 \frac{\Cr{const:W_z-minus-H_z-small}}{Leb(M)} \eps m_D(D)=\frac{\Cr{const:W_z-minus-H_z-small} \pi}{Leb(M)} \eps^3. \]
\end{proof}

\begin{definition}\label{def_pi_and_gamma}
Let $\pi$ denote the natural projection of $U_0$ to $H$ by the holonomy, meaning $\pi(r):=x$ for $r\in W^{c-s}_x$.
Let the measure $\gamma$ on $H$ be the push-down of $\mu_{|U_0}$ from $U_0$ to $H$ by $\pi$, and let $m:=Leb_{W^u}|_H$ denote the restriction of the Lebesgue measure of $W^u$ to $H$.
\end{definition}

\begin{lemma}\label{lem:beta-bounded}
 $\gamma \ll m$ with density
  \begin{equation}\label{eq:beta-def}
   \beta:=\frac{\diff\gamma}{\diff m}:H\to\IR^+.
  \end{equation}
 satisfying
 \[\frac{1}{C_\beta}\eps^2 \le \beta \le C_\beta \eps^2,\]
 where $C_\beta=C_\beta(\cR_Q,\cR_u)$ is a global constant.
\end{lemma}

The proof of the lemma also gives a formula for $\beta$, but we do not need that.

\begin{proof}
Let $h_z$ be the holonomy from $W^u$ to $W^u_z$. So for $B\subset H$, $(\pi^{-1}B)\cap W^u_z=h_z(B)$. So, disintegrating the $\mu$-measure of $\pi^{-1}B$ w.r.t. the u-foliation, we get
\begin{eqnarray*}
\gamma(B)&=&\mu_U(\pi^{-1}B)=\int_D \mu^{cond}_{W^u_z}(h_z (B))\diff \mu^{factor}_D(z)=\\
&=& \int_D \int_{W^u_z} \Ind_{h_z(B)}(\tilde x) \frac{\diff \mu^{cond}_{W^u_z}}{\diff m_{W^u_z}}(\tilde x) \diff m_{W^u_z}(\tilde x) \diff \mu^{factor}_D(z).
\end{eqnarray*}
In the inner integral we substitute $\tilde x=h_z(x)$, which is exactly pulling back the integral to $W^u$ from $W^u_z$. We get
\[
\gamma(B)=\int_D \int_{W^u} \Ind_B(x) \frac{\diff \mu^{cond}_{W^u_z}}{\diff m_{W^u_z}}(h_z(x)) Jh_z(x) \diff m_{W^u}(x) \diff \mu^{factor}_D(z).
\]
We can now exchange the integrals to get
\[
\gamma(B)=\int_{W^u} \Ind_B(x) \int_D \frac{\diff \mu^{cond}_{W^u_z}}{\diff m_{W^u_z}}(h_z(x)) Jh_z(x) \diff \mu^{factor}_D(z)\diff m_{W^u}(x),
\]
which means exactly that
\begin{equation}
\beta(x)=\frac{\diff\gamma}{\diff m}(x)=\int_D \frac{\diff \mu^{cond}_{W^u_z}}{\diff m_{W^u_z}}(h_z(x)) Jh_z(x) \diff \mu^{factor}_D (z).
\end{equation}
is indeed the density of $\gamma$ w.r.t. $m$. Using Lemma~\ref{lem:u-density-times-factor-density} this can be written as
\begin{equation}\label{eq:beta_formula}
\beta(x)=\frac{\diff\gamma}{\diff m}(x)=\int_D \frac{1}{Leb(M)} \cos\psi(h_z(x)) Jh_z(x) \diff m_D (z).
\end{equation}

To see the bounds on $\beta$, we use that $\frac{1}{2}\le \cos\psi(h_z(x))\le 1$ by Lemma~\ref{lem:u-density-regularity}, $\frac{1}{C_h}\le Jh_z(x) \le C_h$ by Theorem~\ref{thm:holonomy_dynholder} and $m_D(D)=\eps^2\pi$ by definition.
\end{proof}

\subsection{Product structure}\label{sec:product}

We used $U_0$ to denote the set of points where leaves of the $u$-foliation and long central-stable manifolds intersect. Indeed, points of $U_0$ can be identified with a pair of ``coordinates'' $(x,z)$, where $x$  identifies the central-stable manifold and $z$ identifies the $u$-curve. This can be formulated in many ways -- see Notation~\ref{not:holonomy}:
\begin{eqnarray*}
U_0 &=& \{h_z(x) | x\in H,z\in D\}=\\
&=& \bigcup_{z\in D} h_z(H) = \\
&=& \bigcup_{x\in H} W^{c-s}_x = \\
&=& \bigcup_{x\in H} \bigcup_{z\in D} \left(W^{c-s}_x \cap W^u_z\right) = \\
&=& h(H\times D).
 \end{eqnarray*}

For measurable subsets $A\subset H$ and $B\subset D$ we will use the notation $A\ast B$ to denote this approximate product in $U_0$, while reserving the notation $A\times B$ for the usual Cartesian product:
\[
A\ast B:=h(A\times B)= \bigcup_{x\in A} \bigcup_{z\in B} \left(W^{c-s}_x \cap W^u_z\right).
\]
With this notation, we have
\[U_0=H\ast D.\]

Lemma~\ref{lem:beta-bounded} immediately implies that
\begin{lemma}\label{lem:A-times-D-small}
For any measurable $A\subset H$ we have $\mu(A\ast D)\le C_\beta \eps^2 m_{W^u}(A)$,
where $C_\beta=C_\beta(\cR_Q,\cR_u)<\infty$ is the global constant from Lemma~\ref{lem:beta-bounded}.
\end{lemma}

\begin{proof} By Definition~\ref{def_pi_and_gamma} and Lemma~\ref{lem:beta-bounded} we can write
\[\mu(A\ast D)=\mu(\pi^{-1}(A))=\gamma(A) = \int_A 1 \diff \gamma = \int_A \beta(x)\diff m(x).\]
Lemma~\ref{lem:beta-bounded} also gives the bound for the integrand, so
\[\mu(A\ast D)\le \int_A C_\beta\eps^2 \diff m = C_\beta \eps^2 m_{W^u}(A).\]
\end{proof}

\subsection{Construction of the approximating density}\label{sec:construction}

We choose $q=q_\eps :D\to\IR^+$ be a smooth enough probability density. The specific form is not important, but for easier calculations we choose the function
\[
q(z)=q_\eps(z):=\frac{3}{\eps^2 \pi} \left(1-\frac{|z|}{\eps}\right) \quad \text{ for $z\in D$}.
\]
we will use the following regularity properties:
\begin{lemma}\label{lem:q-regularity}
\begin{enumerate}
 \item $q$ is a probability density on $D$ with respect to Lebesgue measure $m_D$,
 \item $q(z) \le \frac{3}{\pi} \frac{1}{\eps^2}$ for any $z\in D$,
 \item $|q(z_1)-q(z_2)| \le \frac{3}{\pi} \frac{1}{\eps^3}|z_1-z_2|$ for any $z_1,z_2\in D$,
 \item $q$ vanishes on the boundary of $D$.
\end{enumerate}
\end{lemma}

\begin{proof}
Easy calculation. Remember that $D=\{z\in\IR^2\,:\,|z|\le\eps\}$.
\end{proof}

Let $\tilde{q}$ denote the measure on $D$ with density $q$ (with respect to Lebesgue measure $m_D$).
We construct our approximating density as a ``product'' of $\phi$ in the $x$ direction and $q$ in the $z$ direction using the product structure.
Specifically, let $\tilde{G_0}$ be the measure on $U_0$ which is defined on approximate product sets $A\ast B$ as the push-forward of $\tilde{q} \otimes \tilde{\phi}$ from $H\times D$ to $U_0$ by $h$:
\begin{equation}\label{eq:tildeG_0-def}
\tilde{G_0}(A\ast B) = \tilde{G_0}(h(A\times B)):= (\tilde{q} \otimes \tilde{\phi})(A\times B)) = \int_A \phi \diff m \cdot \int_B q \diff m_D.
\end{equation}
Let $G_0:U_0\to \IR^+$ be defined as
\begin{equation}\label{eq:G_0-def}
G_0:=\frac{\diff \tilde{G_0}}{\diff \mu}.
 \end{equation}
This definition makes sense, since $\tilde{G_0}$ is indeed absolutely continuous with respect to $\mu$, as we will see in Lemma~\ref{lem:G0-formula}.

We will use this definition only through two consequences. The first consequence is an explicit formula that allows us to check the regularity of $G_0$, and we obtain it by disintegrating (\ref{eq:tildeG_0-def}) w.r.t. the u-foliation.

\begin{lemma}\label{lem:G0-formula}
 For $x\in H$, $z\in D$ and $r=h_z(x)$
 \[G_0(r)=Leb(M) \phi(x)q(z) \frac{1}{Jh_z(x)}\frac{1}{\cos\psi(r)}.\]
 (for almost every $r$).
\end{lemma}

\begin{proof}
For any measurable $A\subset H$ and $B\subset D$ we calculate $\tilde{G_0}(A\ast B)$ from (\ref{eq:G_0-def}) by disintegrating w.r.t. the u-foliation:
\begin{eqnarray*}
\tilde{G_0}(A\ast B) &=&
\int_B \int_{H_z(A)} G_0(r) \diff \mu^{cond}_{W^u_z}(r) \diff\mu^{factor}_D (z) = \\
&=& \int_B \int_{H_z(A)} G_0(r) \frac{\diff \mu^{cond}_{W^u_z}}{\diff m_z}(r) \diff m_z(r) \diff\mu^{factor}_D (z).
\end{eqnarray*}

In the inner integral we use the substitution $r=h_z(x)$ and notice that $\diff m_z(r) = Jh_z(x) \diff m(x)$ by definition (\ref{eq_Jh-def}):
\begin{eqnarray*}
\tilde{G_0}(A\ast B) &=&
\int_B \int_A G_0(h_z(x)) \frac{\diff \mu^{cond}_{W^u_z}}{\diff m_z}(h_z(x)) Jh_z(x) \diff m(x) \diff\mu^{factor}_D (z) = \\
&=& \int\limits_B \int\limits_A G_0(h_z(x)) \frac{\diff \mu^{cond}_{W^u_z}}{\diff m_z}(h_z(x)) Jh_z(x) \diff m(x) \frac{\diff\mu^{factor}_D}{\diff m_D}(z) \diff m_D (z)
\end{eqnarray*}
Comparing this with (\ref{eq:tildeG_0-def}) we find that
\begin{eqnarray*}
& &\int_{A\times B} \phi(x) q(z) \diff (m\otimes m_D)(x,z)= \\
& &= \int_{A\times B} G_0(h_z(x)) \frac{\diff \mu^{cond}_{W^u_z}}{\diff m_z}(h_z(x)) Jh_z(x) \frac{\diff\mu^{factor}_D}{\diff m_D}(z) \diff (m\otimes m_D)(x,z)
\end{eqnarray*}
for every measurable $A\subset H$ and $B\subset D$, so the integrands have to be equal:
\[
\phi(x) q(z)=G_0(h_z(x))  \frac{\diff \mu^{cond}_{W^u_z}}{\diff m_z}(h_z(x)) Jh_z(x) \frac{\diff\mu^{factor}_D}{\diff m_D}(z)
\]
for
$m\otimes m_D$-a.e. $(x,z)\in H\times D$, which also means $\mu$-a.e. $r\in U_0$.
Using Lemma~\ref{lem:u-density-times-factor-density} we get exactly the statement to be proven.
\end{proof}

The second consequence
says that for functions that are constant along central-stable manifolds, integrating against $G_0$ on $U_0$ is exactly the same as integrating against $\phi$ on $H$. We get it by projecting down to $H$ along central-stable manifolds. Remember that $\pi$ denotes the natural projection from $U_0$ to $H$ so when $r=h_z(x)\in U_0$ for some $x\in H$ and $z\in D$, we have $\pi(r)=x$. So a function that is constant along central-stable manifolds is a function which depends on $r$ through $\pi(r)$ only.

\begin{lemma}\label{lem:G_0-similarto-phi}
 For $f:U_0\to \IR$ (or $f:H\to\IR$) and any $A\subset H$,
 \[ \int\limits_{A\ast D} (f\circ \pi )G_0 \diff\mu =\int\limits_A f \phi \diff m.\]
\end{lemma}

\begin{proof}
We first use the definition of $G_0$ to replace integration w.r.t. $\mu$ by integration $w.r.t.$ $\tilde{G_0}$. Then we use the definition of $\tilde{G_0}$ to perform the integral substitution $r:=h(x,z)$:
\[
\int\limits_{A\ast D} \left(f \circ\pi\right)G_0 \diff\mu = \int\limits_{A\ast D} f(\pi(r)) \diff\tilde{G_0}(r) =
\int\limits_{A\times D} f(\pi(h(x,z))) \diff(\tilde{\phi}\otimes \tilde{q})(x,z).
\]
Since $\pi(h(x,z))=x$, the integrand depends on $x$ only (that's how it was designed), so the integral factorizes and the integral w.r.t. $z$ becomes $1$:
\[
\int\limits_{A\ast D} \left(f \circ\pi\right)G_0 \diff\mu = \int\limits_A f(x) \left[\int\limits_D 1 \diff\tilde{q}(z)\right] \diff\tilde{\phi}(x) = \int\limits_A f \diff\tilde{\phi}.
\]
Using the definition of $\tilde{\phi}$ gives the statement.
\end{proof}

\begin{remark} It is not hard to see that the foliation of $U_0$ with central-stable manifolds is also integrable (w.r.t. $\mu |_{U_0}$). We will not use this fact, so we do not prove it. However, for better understanding, we give the following statement (which we will not use), demonstrating how the function $G_0$ is ``well related'' to the central-stable foliation. It says that ``the integral of $G_0$ on each central-stable manifold is exactly what it should be'', and we obtain it by disintegrating (\ref{eq:tildeG_0-def}) w.r.t. the central-stable foliation.
Let $\nu_x$ denote the conditional measures of $\mu |_{U_0}$ on the $W^{c-s}_x$ (the existence of which we do not show). The factor measure is $\gamma$. Then for $m$-a.e. $x\in H$
 \[\int_{W^{c-s}_x} G_0(r) \diff \nu_x(r) = \frac{\phi(x)}{\beta(x)}.\]
Indeed, applying (\ref{eq:tildeG_0-def}) and (\ref{eq:G_0-def}) with $B=D$, and then disintegrating w.r.t. the central-stable foliation, we get
\[\int_A \phi(x)\diff m(x) = \tilde{G_0}(A\ast D)=\int_{A\ast D} G_0\diff \mu = \int_A \int_{W^{c-s}_x} G_0(r)\diff \nu_x(r) \diff \gamma(x).\]
By (\ref{eq:beta-def}) $\beta(x)=\frac{\diff \gamma(x)}{\diff m(x)}$, so this gives
\[\int_A \phi(x)\diff m(x) = \int_A \int_{W^{c-s}_x} G_0(r)\diff \nu_x(r) \beta(x) \diff m(x)\]
for every measurable $A\subset H$, which means that
\[\phi(x) = \int_{W^{c-s}_x} G_0(r)\diff \nu_x(r) \beta(x) \]
for $m$-a.e. $x\in H$.
\end{remark}

The following proposition is the key to the approximation of the singular measure with the density. It says that if we substitute the density $\phi$ on $W^u$ with the density $G_0$ on $M$, we make little mistake when integrating $F \circ \Phi^t$, if only $F$ is sufficiently regular. This is a strong statement, since $F\circ \Phi^t$ is far from inheriting the regularity of $F$, at least in unstable directions. So the essence of the proposition is that this particular approximating density is insensitive to irregularity in the unstable direction.

\begin{proposition}\label{prop:integral_compare_on_H1-U1}
There is a constant $1\le C_\pi=C_\pi(\cR_Q,\cR_u)<\infty$ such that for any measurable $A\subset H$,
\[
\left| \int\limits_{A\ast D}\left(F\circ \Phi^t\right)G_0\diff\mu - \int\limits_A \left(F\circ \Phi^t\right) \phi \diff m \right|\le
\int\limits_{A\ast D} \left[ \left(osc_{C_\pi \eps} F\right) \circ \Phi^t\right]G_0\diff\mu.
\]
\end{proposition}

\begin{proof}
  Let $C_\pi=\frac{10}{\sin c_{tr}}$, where $c_{tr}$ is the transversality bound from Theorem~\ref{thm:transversality}.
  Let us consider the error we make if we substitute the function $F(\Phi^t(r))$ with the constant $F(\Phi^t(x))$ on each $W^{c-s}_x$. In other words, we are comparing the function $F\circ \Phi^t$ to the function $F\circ \Phi^t \circ \pi$. By the construction of $H$ in (\ref{eq:H-def}), for any $r\in U_0$, $r$ and $\pi(r)$ are $C_\pi \eps$ close, and they are also on the same central-stable manifold, so their distance does not increase in time. This means that $dist(\Phi^t(r),\Phi^t(\pi(r)))\le C_\pi \eps$, or in other words $\Phi^t(\pi(r))\in B_{C_\pi \eps} (\Phi^t(r))$. This implies that
 \begin{equation}\label{eq:F-Phi-pi}
  \left|F\circ \Phi^t - F \circ \Phi^t \circ\pi\right|\le (osc_{C_\pi \eps} F)\circ \Phi^t.
 \end{equation}
 Using Lemma~\ref{lem:G_0-similarto-phi} from right to left with $f= F\circ \Phi^t$, we get
 \begin{eqnarray*}
 \left| \int\limits_{A\ast D}\left(F\circ \Phi^t\right)G_0\diff\mu - \int\limits_A \left(F\circ \Phi^t\right) \phi \diff m \right| = \\
 \left| \int\limits_{A\ast D}\left(F\circ \Phi^t\right)G_0\diff\mu - \int\limits_{A\ast D}\left(F\circ \Phi^t\circ \pi\right)G_0\diff\mu \right| \le \\
 \le \int\limits_{A\ast D}\left| F\circ \Phi^t - F\circ \Phi^t\circ \pi\right| G_0\diff\mu.
 \end{eqnarray*}
 Substituting (\ref{eq:F-Phi-pi}) gives the result.
 \end{proof}

The above proposition points to a technical difficulty we have to fight: we will need to show that the right hand side is small when $t$ is large and $\eps$ is chosen appropriately. This does not follow immediately from norm estimates. Indeed, if $F$ were H\"older continuous, then $osc_{C_\pi \eps} F$ would be uniformly small, and a good upper bound would be immediate. However, we only assume that $F$ is generalized H\"older, so $osc_{C_\pi \eps} F$ is only small on average. The upper bound on the right hand side -- which is a time correlation function -- will follow from the correlation decay in Theorem~\ref{thm:BDL_extended}. For this, the generalized H\"older regularity of $osc_{C_\pi \eps} F$ needs to be shown. This is done in the separate paper \cite{T17}. The main theorem there is the following:

\begin{theorem}\label{thm:osc_f-genHolder}
 For any Lebesgue measurable $D\subset\IR^d$, any bounded $f:D\to\IR$, any $r>0$ and any $0<\alpha\le 1$
 \[|osc_r f|_{\alpha;gH} \le 2(\sup_D f - \inf_D f) \mu(Conv(D))\left(\frac{2d+1}{r}\right)^\alpha,\]
 where $Conv(D)$ denotes the convex hull of $D$.
\end{theorem}

We use this via the following corollary:

\begin{corollary}\label{cor:osc_F-genholder-norm}
 There is a global constant $\Cl{const:osc_F-genholder-norm}=\Cr{const:osc_F-genholder-norm}(\cR_Q)<\infty$ such that for any $0<\alpha\le 1$ and $0<\eps\le diam(M)$
 \[var_{\alpha}(osc_{C_\pi\eps} F) \le \frac{\Cr{const:osc_F-genholder-norm}}{\eps^{\alpha}} (\sup_M F - \inf_M F).\]
\end{corollary}

\begin{proof}
 By the definition of $var_\alpha$ in (\ref{eq:var}),
 \[var_{\alpha}(osc_{C_\pi\eps} F) = |osc_{C_\pi\eps} F|_{\alpha;gH} + \sup_M(osc_{C_\pi\eps} F) - \inf_M(osc_{C_\pi\eps} F).\]
 The second term is $\sup_M(osc_{C_\pi\eps} F)\le \sup_M F -\inf_M F$, while the third is $\inf_M(osc_{C_\pi\eps} F)\ge 0$, so
 \begin{equation} \label{eq:var-osc-F}
 var_{\alpha}(osc_{C_\pi\eps} F) \le |osc_{C_\pi\eps} F|_{\alpha;gH} + \sup_M F - \inf_M F.
 \end{equation}
 To bound the first term, we would like to apply Theorem~\ref{thm:osc_f-genHolder}
 with $d=3$, $D=M\subset \IR^3$, $f=F$, $r=C_\pi \eps$ and $\mu=\frac{1}{Leb(M)}Leb$. The only minor problem is that this theorem is about functions on subsets of $\IR^d$, while our $F$ has domain $M\subset\IT^3$. This can easily be treated at the cost of some non-optimal constant factor $C=125$, see Remark~\ref{rem:osc-f-genHolre-on-torus}. We get
 \[|osc_{C_\pi\eps} F|_{\alpha;gH}\le 2(\sup_M F - \inf_M F)\frac{1}{Leb(M)}125 Leb(\IT^3) \left(\frac{2\cdot 3+1}{C_\pi \eps}\right)^\alpha.\]
 Since $0<\alpha\le 1$ and $C_\pi\ge 1$, an upper bound independent of $\alpha$ can be given (although this is not important for us): $\left(\frac{2\cdot 3+1}{C_\pi}\right)^\alpha \le 7$. We assumed $\eps\le diam(M)$, so
 $1\le\left(\frac{diam(M)}{\eps}\right)^\alpha\le \frac{\max\{diam(M),1\}}{\eps^\alpha}$. Writing these back to (\ref{eq:var-osc-F}), we get the statement of the corollary with $\Cr{const:osc_F-genholder-norm}= 14\frac{125 Leb(\IT^3)}{Leb(M)}+\max\{diam(M),1\}$.
\end{proof}

\begin{remark}\label{rem:osc-f-genHolre-on-torus}
Theorem~\ref{thm:osc_f-genHolder} is about $f:D\to\IR$ where $D\subset\IR^d$, while we have $F:M\to\IR$ where $M\subset\IT^3$.
A non-optimal, but easy way out is the following:

If $r:=C_\pi \eps\ge diam(M)$ (which is unrealistic anyway), then $osc_r F\equiv \sup_M F - \inf_M F$ is constant, so
$|osc_r F|_{\alpha;gH}=0$.

So assume $r=C_\pi \eps < diam(M)$. We view $M\subset \IT^3$ as $M\subset [0,1]^3\subset\IR^3$, and extend $F$ from $M$ to $M+\IZ^3$ periodically. Then set $D:=(M+\IZ^3)\cap[-2,3]^3\subset\IR^3$, and restrict $F$ to $D$. In words: we extend $F$
from a single copy of $M$ to the neighbouring $5\cdot 5\cdot 5$ fundamental cells, $5^3=125$ copies of $M$ all together, to get some $\tilde F:D\to\IR$ where $D\subset\IR^3$ is still bounded. On the central $3\cdot 3 \cdot 3=27$ copies of $M$ in $[-1,2]^3$, $osc_r \tilde F$ is the same as $ocr_r F$. So, as long as $\delta<diam(M)$, $osc_\delta (osc_r \tilde F)$ and $osc_\delta (osc_r F)$ coincide on the central copy $M\subset[0,1]^3$. So, for $\delta<diam(M)$,
\[\int_M osc_\delta (osc_r F) = \int_{D\cap[0,1^3]} osc_\delta (osc_r \tilde F) \le \int_D osc_\delta (osc_r \tilde F).\]
(For $\delta \ge diam(M)$ the oscillation doesn't grow further, meaning $osc_\delta (osc_r F)=osc_{diam(M)} (osc_r F)$, so the same bound trivially holds.)

Now we can apply Theorem~\ref{thm:osc_f-genHolder} to the extended function $\tilde F:D\to\IR$ to get the bound on
$|osc_{C_\pi\eps} F|_{\alpha;gH}$. Clearly $Leb(Conv(D))\le Leb([-2,3]^3)=125 Leb(\IT^3)$.
\end{remark}

\subsection{Regularity of the approximating density}\label{sec:regularity}

\begin{proposition}\label{prop:G_0-dynHolder}
If $\phi$ is $\Theta_\phi$-dynamically H\"older continuous, then $G_0$ is uniformly dynamically H\"older continuous when restricted to any element of the u-foliation: There exist constants $C_{G;u}=C_{G;u}(\cR_Q,\cR_u)<\infty$ and $\Theta_G=\Theta_G(\Theta_\phi,\cR_Q,\cR_u) < 1$ such that for any $z\in D$ and $r_1,r_2\in H_z$
 \[|G_0(r_1)-G_0(r_2)|\le C_{G;u} \frac{1}{\eps^2} ||\phi||_{\Theta_\phi;dH} \Theta_G^{s^+(r_1,r_2)}.\]
 In particular, $C_{G;u}$ and $\Theta_G$ do not depend on $\phi$ and $\eps$ (but $\Theta_G$ depends on $\Theta_\phi$).
\end{proposition}

\begin{proof}
 Lemma~\ref{lem:G0-formula} gives an explicit formula for $G_0$:
 \[G_0(r)=Leb(M) \phi(x)q(z) \frac{1}{Jh_z(x)}\frac{1}{\cos\psi(r)},\]
 where $x=\pi(r)$ is the projection of $r$ to $H$ by the holonomy. All factors are dynamically H\"older, thus so is the product. Quantitatively:
 \begin{enumerate}[a.)]
  \item Since we fix $z$, $q(z)$ is just a constant, and Lemma~\ref{lem:q-regularity} ensures that $0\le q(z)\le \frac{3}{\pi}\frac{1}{\eps^2}$.
  \item $\phi(x)=\phi(\pi(r))$ has the same dynamical H\"older regularity as $\phi$, due to Lemma~\ref{lem:sep-time-constant-on-W^cs}, so for $\phi\circ\pi:H_z\to \IR$
  \[\sup_{H_z}(\phi\circ \pi)=\sup_H \phi \le \sup_{W^u} \phi \quad , \quad |\phi\circ\pi|_{\Theta_\phi;dH}=|\phi_{|H}|_{\Theta_\phi;dH}\le |\phi|_{\Theta_\phi;dH}.\]
  \item $\frac{1}{Jh_z(x)}=Jh_{H_z\to H}(r)$ is exactly the Jacobian of the holonomy from $H_z$ to $H$, so Theorem~\ref{thm:holonomy_dynholder} ensures that
  \[\sup_{H_z}\frac{1}{Jh_z}\le C_h \quad , \quad \left|\frac{1}{Jh_z}\right|_{\Theta_h;dH}\le C_h.\]
  \item By Lemma~\ref{lem:u-density-regularity} $\cos\psi(r)$ is $\alpha$-H\"older with $\alpha=1$ (which is Lipschitz continuity), so it is also dynamically H\"older by Lemma~\ref{lem:Holder-dynHolder} (ii) with some $\Theta_{\cos\psi}=\Theta_{\cos\psi}(\alpha=1,\cR_Q,\cR_u)=\Theta_{\cos\psi}(\cR_Q,\cR_u)<1$ and $C_{\cos\psi}=C_{\cos\psi}(\alpha=1,\cR_Q,\cR_u)=C_{\cos\psi}(\cR_Q,\cR_u)\in\IR$. In turn, Lemma~\ref{lem:dynHolder-product}(\ref{it:1_over_f-dynHolder}) ensures that $\frac{1}{\cos\psi(r)}$ is also dynamically H\"older:
  \[\left|\frac{1}{\cos\psi(r)}\right|\le 2 \quad , \quad \left|\frac{1}{\cos \psi}\right|_{\Theta_{\cos\psi};dH} \le 4 C_{\cos\psi}.\]
 \end{enumerate}
So let us choose
\[\Theta_G=\Theta_G(\Theta_\phi,\cR_Q,\cR_u):=\max\{\Theta_\phi,\Theta_h,\Theta_{\cos\psi}\}<1.\]
With this choice, Lemma~\ref{lem:dynHolder-product}(\ref{it:dynHolder-dynHolder}) ensures that the above three dynamical H\"older regularity statements remain valid with
$\Theta_\phi$, $\Theta_h$ and $\Theta_{\cos\phi}$ replaced by $\Theta_G$. Now Lemma~\ref{lem:dynHolder-product}(\ref{it:dynHolder-product}) and the definition (\ref{eq:dynHolder-norm}) of the dynamical H\"older norm ensure that
\begin{eqnarray*}
|{G_0}_{|H_z}|_{\Theta_G;dH}&\le& Leb(M) \frac{3}{\pi}\frac{1}{\eps^2}\left( \sup \phi C_h 4 C_{\cos\psi} + \sup\phi 2 C_h + C_h 2 |\phi|_{\Theta_\phi;dH} \right) \\
&\le& Leb(M) \frac{3}{\pi}\frac{1}{\eps^2} C_h (4C_{\cos\psi}+2+2) ||\phi||_{\Theta_\phi;dH}.
\end{eqnarray*}
Choosing $C_{G;u}:=\frac{12}{\pi}Leb(M) C_h (C_{\cos\psi} + 1)$ completes the proof.
\end{proof}

\begin{proposition}\label{prop:G_0-Holder-on-W^cs}
 With $\alpha_{G_0}:=\frac13$, $G_0$ is uniformly $\alpha_{G_0}$-H\"older continuous when restricted to any element of the central-stable-foliation: There is a constant $C_{G;cs}=C_{G;cs}(\cR_Q,\cR_u)<\infty$ such that for any $x\in H$ and any $r_1,r_2\in W^{c-s}_x$
 \[|G_0(r_1)-G_0(r_2)|\le C_{G;cs} \sup_{W^u}\phi\frac{1}{\eps^3} dist_{W^{c-s}_x}(r_1,r_2)^{\alpha_{G_0}}.\]
\end{proposition}

\begin{proof}
 Lemma~\ref{lem:G0-formula} gives an explicit formula for $G_0$:
\begin{equation}\label{eq:G0-formula-for-const_x}
G_0(r)=Leb(M) \phi(x)q(z) \frac{1}{Jh_z(x)}\frac{1}{\cos\psi(r)},
\end{equation}
when $r=h_z(x)\in W^{c-s}_x$. {For} $x$ fixed, the first two {terms} are constant,
the third and last are Lipschitz. {Let us estimate $\frac{1}{Jh_z(x)}$}.

Let $z_1,z_2\in D$ and let $r_1=h_{z_1}(x)$, $r_2=h_{z_2}(x)$. Then
\[ \frac{Jh_{z_2}(x)}{Jh_{z_1}(x)}=Jh_{H_{z_1}\to H_{z_2}}(r_1)\]
is exactly the Jacobian of the holonomy from $H_{z_1}\subset W^u_{z_1}$ to $H_{z_2}\subset W^u_{z_2}$. We estimate this using Theorem~\ref{thm:holonomy_cs-regularity}:
\begin{equation}\label{eq:Jh_over_Jh-close-to-1}
\left| \frac{Jh_{z_2}(x)}{Jh_{z_1}(x)}-1 \right| \le C_{h2} \left(|\psi(r_2)-\psi(r_1)| + dist_{W^{c-s}_x}(r_1,r_2)^{\frac13} \right).
\end{equation}
Now by Lemma~\ref{lem:psi-Lipschitz}
\begin{eqnarray}
|\psi(r_2)-\psi(r_1)| &\le& 2 \Gamma_{max} dist(r_1,r_2) \le 2 \Gamma_{max} dist(r_1,r_2)^{\frac23} dist(r_1,r_2)^{\frac13} \le \\
  & \le & 2 \Gamma_{max} diam(M)^{\frac23} dist_{W^{c-s}_x}(r_1,r_2)^{\frac13}.
\end{eqnarray}
We {substitute} to (\ref{eq:Jh_over_Jh-close-to-1}), and use $\frac{1}{C_h} \le Jh_{z_2}(x)$ from Theorem~\ref{thm:holonomy_dynholder} to get
\[
\left| \frac{1}{Jh_{z_1}(x)}-\frac{1}{Jh_{z_2}(x)} \right| \le C_{h2} C_h \left( 2 \Gamma_{max} diam(M)^{\frac23} + 1 \right) dist_{W^{c-s}_x}(r_1,r_2)^{\frac13}.
\]
So $\frac{1}{Jh_z(x)}$ is H\"older continuous in its variable $r=h_z(x)$ along $W^{c-s}_x$, with exponent $\frac13$ and constants
\[
\left| \frac{1}{Jh_{z}(x)} \right| \le C_h \quad , \quad \left| \frac{1}{Jh_{z}(x)} \right|_{\frac13;H} \le C_{h2} C_h \left( 2 \Gamma_{max} diam(M)^{\frac23} + 1 \right).
\]
As mentioned before, the other factors in (\ref{eq:G0-formula-for-const_x}) are easy:
 \begin{enumerate}[a.)]
  \item Since we fix $x$, $\phi(x)$ is just a constant, and of course $0\le \phi(x)\le \sup_{W^u}\phi$.
  \item By Lemma~\ref{lem:q-regularity} $q$ is $\alpha$-H\"older with $\alpha=1$ (which is Lipschitz continuity) on $D$, so by Lemma~\ref{lem:Holder-Holder} it is also $\frac13$-H\"older and
  \[|q|_{\frac13;H}\le diam(D)^{\frac23} |q|_{1;H}=(2\eps)^{\frac23} \frac{3}{\pi} \frac{1}{\eps^3}=\frac{3\cdot 2^{2/3}}{\pi}\frac{1}{\eps^{2\frac13}}.\]
  Now Lemma~\ref{lem:W-u_z-s_not_close} says that $|z_2-z_1|\le 2 dist(r_1.r_2)\le 2 dist_{W^{c-s}_x}(r_1,r_2)$, so $q(z)$ is also $\frac13$-H\"older as a function of $r=h_z(x)$:
  \[|q(z_2)-q(z_1)|\le \frac{3\cdot 2^{2/3}}{\pi}\frac{1}{\eps^{2\frac13}} |z_2-z_1|^{\frac13} \le \frac{6}{\pi}\frac{1}{\eps^{2\frac13}} dist_{W^{c-s}_x}(r_1,r_2)^{\frac13}.\]
  We also have $q(z)\le \frac{3}{\pi}\frac{1}{\eps^2}$ by Lemma~\ref{lem:q-regularity}.
  \item By Lemma~\ref{lem:psi-Lipschitz} $\psi(r)$ is Lipschitz on $U$, thus so is $\cos\psi(r)$. So it is also $\frac13$-H\"older by Lemma~\ref{lem:Holder-Holder} (\ref{it:Holder-Holder}). Also $\cos\psi(r)\ge\frac12$ by Lemma~\ref{lem:u-density-regularity}, so Lemma~\ref{lem:Holder-Holder} (\ref{it:1_over_f-Holder}) says that $\frac{1}{\cos\psi(r)}$ is also $\frac13$-H\"older. Quantitatively,
  \begin{eqnarray*}
   \left| \frac{1}{\cos\psi} \right| &\le& 2 ,  \\
   \left| \frac{1}{\cos\psi} \right|_{\frac13;H} &\le& 4 | \cos\psi |_{\frac13;H} \le 4 diam(M)^{\frac23} | \cos\psi |_{1;H}\le \\
   &\le& 4 diam(M)^{\frac23} | \psi |_{1;H}\le 4 diam(M)^{\frac23} 2\Gamma_{max}.
  \end{eqnarray*}
 Since $dist(r_1,r_2)\le dist_{W^{c-s}_x}(r_1,r_2)$, $\frac{1}{\cos\psi}$ is also $\frac13$-H\"older on $W^{c-s}_x$ with the same constants.
 \end{enumerate}
Putting the estimates for the factors in (\ref{eq:G0-formula-for-const_x}) together, Lemma~\ref{lem:Holder-Holder} (\ref{it:Holder-product}) says that $G_0$ is $\frac13$-H\"older and
\[|G_0|_{\frac13;H} \le Leb(M) \sup_{W^u} \phi\left( \frac{3}{\pi}\frac{1}{\eps^2} C_h 8 diam(M)^{\frac23} \Gamma_{max} +
                                              \frac{3}{\pi}\frac{1}{\eps^2} C_{h2} C_h \left( 2 \Gamma_{max} diam(M)^{\frac23} + 1 \right) 2 +
                                              \frac{6}{\pi}\frac{1}{\eps^{2\frac13}} C_h 2
                                              \right)
\]
Since $\eps\le diam(M)$, choosing
\[
C_{G;cs}:=
Leb(M) \frac{6}{\pi} C_h \left(2 \Gamma_{max} (2+C_{h2}) diam(M)^{\frac53} + C_{h2} diam(M) + 2 diam(M)^{\frac23} \right)
\]
completes the proof.
\end{proof}

\begin{proposition}\label{prop:G_0-bounded} $G_0$ is uniformly bounded: there is a constant $C_{G;b}=C_{G;b}(\cR_Q,\cR_u)<\infty$ such that
 \[0\le G_0(r)\le \frac{C_{G;b}}{\eps^2} \sup_{W^u}\phi\]
 for every $r\in U_0$.
\end{proposition}

\begin{proof}
 This is actually included in the proofs of both previous lemmas. Lemma~\ref{lem:G0-formula} gives an explicit formula for $G_0$:
 \[G_0(r)=Leb(M) \phi(x)q(z) \frac{1}{Jh_z(x)}\frac{1}{\cos\psi(r)},\]
 where $x=\pi(r)$ is the projection of $r$ to $H$ by the holonomy. All the factors multiplying $\phi$ have known bounds. Quantitatively:
 \begin{enumerate}[a.)]
  \item Lemma~\ref{lem:q-regularity} ensures that $0\le q(z)\le \frac{3}{\pi}\frac{1}{\eps^2}$.
  \item $\frac{1}{Jh_z(x)}=Jh_{H_z\to H}(r)$ is exactly the Jacobian of the holonomy from $H_z$ to $H$, so Theorem~\ref{thm:holonomy_dynholder} ensures that
  $0<\frac{1}{Jh_z}\le C_h$.
  \item By Lemma~\ref{lem:u-density-regularity}, $0<\frac{1}{\cos \psi(r)}\le 2$.
 \end{enumerate}
Choosing $C_{G;b}:=\frac{6}{\pi} Leb(M) C_h$ completes the proof.
\end{proof}

\subsection{Smoothing the approximating density}\label{sec:smoothing}

Our approximating density $G$ will be a slight modification of $G_0$ to ensure that it has the required regularity.
First, we restrict $G_0$ from $U_0$ to a smaller set $U_1$ to improve its regularity from dynamically H\"older continuous to truly H\"older continuous:

\begin{proposition}\label{prop:G_0-Holder-on-H_1}
 There exist $\Cl{const:G_0-Holder-on-H_1}=\Cr{const:G_0-Holder-on-H_1}(\cR_Q,\cR_u)<\infty$, $\alpha_G=\alpha_G(\cR_Q,\cR_u,\Theta_{\phi})\le 1$ and a set $H_1\subset H$ with the following properties: Let $U_1=H_1\ast D$. Then
 \begin{enumerate}
  \item \label{it:W-minus-H1-small} $m_{W^u}(W^u\setminus H_1)\le \Cr{const:G_0-Holder-on-H_1}\eps$,
  \item \label{it:U-minus-U1-small} $\mu(U\setminus U_1)\le \Cr{const:G_0-Holder-on-H_1}\eps^3$,
  \item \label{it:G_0-Holder-on-H_1} $G_0$ restricted to $U_1$ is $\alpha_G$-H\"older continuous: for any $r_1,r_2 \in U_1$
  \[|G_0(r_1)-G_0(r_2)|\le \Cr{const:G_0-Holder-on-H_1} \frac{||\phi||_{\Theta_\phi;dH}}{\eps^3} |r_1-r_2|^{\alpha_G}.\]
 \end{enumerate}
\end{proposition}

\begin{proof}
The main input is the dynamical H\"older continuity of $G_0$, as stated in Proposition~\ref{prop:G_0-dynHolder}. We will construct $H_1$ by cutting out some neighbourhood of every singularity from $H$. If we do this appropriately, the total set we cut out will be small, and $G_0$ restricted to the remaining set will be H\"older (and not only dynamically H\"older). To obtain this, take $c>0$ and $\theta<1$ to be specified later. For every singularity of order $n\ge 0$, we cut out a neighbourhood (in the metric on $W^u_z$) of radius at least $\frac{c}{2}\theta^n$ (meaning an interval of length $c\theta^n$) from \emph{every} $W^u_z$.

To make the argument precise, we have to take into account secondary singularities, meaning that in principle, for every $n$, there are infinitely many intervals we need to cut out around singularities of order $n$. As usual, this only costs some power of $\frac{1}{\eps}$ in the measure of the set cut out, since the infinitely many secondary singularities accumulate at finitely many primary ones, so the intervals overlap heavily. An easy (not optimal) way to do this is  the following: Let $y_k\in W^u_z$ be on the $k$-th secondary singularity near a primary singularity. Then, by alignment (Theorem~\ref{thm:alignment}) $y_k$ it is at most some $\frac{\Cl{const:secondary}}{k^2}$ far from $y\in W^u_z$, in the metric of $W^u_z$, where $y$ is on the primary singularity -- or, possibly, an endpoint of $W^u_z$. (Here $\Cr{const:secondary}=\Cr{const:secondary}(\cR_Q,\cR_u$.) So if
\begin{equation}\label{eq:k_crit}
k>k_{crit}(n):=\frac{\sqrt{\Cr{const:secondary}}}{\sqrt{\frac{c}{2}\theta^n}},
 \end{equation}
then this distance is less than $\frac{c}{2}\theta^n$, meaning that the entire $\frac{c}{2}\theta^n$-neighbourhood of $y_k$ is contained in the $c \theta^n$-neighbourhood of $y$. So, with some generosity, we cut out neighbourhoods of radius $c\theta^n$
around the primary singularity and the first $k_{crit}(n)$ secondary singularities (and possibly the endpoints of $W^u_z$), and these finitely many intervals cover the $\frac{c}{2}\theta^n$-neighbourhood of every (primary and secondary) singularity.

Now let $y\in W^u_{z_y}$ be a singular point of order $n(y)$ (either primary or secondary). Due to the continuation property of singularity curves, Lemma~\ref{lem:sing-continuation}, the singularity containing $y$ -- or its continuation -- intersects $W^u$ in a single point $x$, which is singular of order $n(x)\le n(y)$. \footnote{Actually, if $y$ is close to the end of the $u$-curve $W^u_z$, it may happen that $W^u_z$ terminates before intersecting the singularity. However, this can only happen if all points of $H_z\subset W^u_z$ are on the same side of $y$, since central-stable manifolds cannot cross singularities. As it will be clear below, such singular points $y$ are of no interest for us.}

Let $I(y)\subset W^u_{z_y}$ be the neighbourhood of radius $c\theta^{n(y)}$ around $y$ in $W^u_{z_y}$, in the metric of $W^u_{z_y}$. Now the set we cut out from $H$ near $x$ is
\[I_x:=\bigcup_{y} h_{z_y}^{-1}(I(y)\cap H_{z_y}),\]
where the union is over all $y$ that give the same $x$ as above. This is an uncountable union, but the members of the union are all intervals around $x$ in $W^u$ (intersected by $H$), so the union is just the longest interval. (More precisely, the longest half-interval has to be taken in both directions.) So
\[
m_{W^u}(I_x)\le 2 \sup_y m_{W^u}(h_{z_y}^{-1}(I(y)\cap H_{z_y}))\le 2\sup_y C_h m_{W^u_{z_y}}(I(y))\le C_h \sup_y c\theta^{n(y)}
\le C_h c \theta^{n(x)}
\]
by the absolute continuity, Theorem~\ref{thm:holonomy_dynholder}. Again, $C_h=C_h(\cR_Q,\cR_u)$. Now we set
\[H_1:=H\setminus \bigcup_x I_x,\]
where the union is over all singular points $x\in W^u$. Since the number of primary singularities of order $n$ is at most $K_{max}^n$ (see Section~\ref{sec:regpar}), the total length we cut out is at most
 \[m_{W^u}(H\setminus H_1)\le\sum_x m_{W^u}(I_x) \le \sum_{n=0}^\infty k_{crit}(n) K_{max}^n C_h c \theta^n.\]
Using (\ref{eq:k_crit}), this gives
 \[m_{W^u}(H\setminus H_1)\le \sum_{n=0}^\infty \Cl{const:secondary2} K_{max}^n \sqrt{c \theta^n} = \Cr{const:secondary2}\frac{\sqrt{c}}{1-K_{max}\sqrt{\theta}}\]
 with $\Cr{const:secondary2}=\sqrt{2 \Cr{const:secondary}C_h}$ if $K_{max}\sqrt{\theta}<1$. So let us choose $\theta\le\theta(\cR_Q):=\frac{1}{4K_{max}^2}$, which means that
 \[m_{W^u}(H\setminus H_1)\le \Cr{const:secondary2}\frac{\sqrt{c}}{1-\frac12}= 2 \Cr{const:secondary2}\sqrt{c}.\]
 Now take $c:=\eps^2$. So item~\ref{it:W-minus-H1-small} is shown, because $W^u\setminus H$ is also small by Lemma~\ref{lem:W-minus-H-small}.

 Item~\ref{it:U-minus-U1-small} follows from Lemma~\ref{lem:U-minus-U0-small} and Lemma~\ref{lem:A-times-D-small}:
 \[\mu(U\setminus U_1)=\mu(U\setminus U_0)+\mu((H\setminus H_1)\ast D)\le \Cr{const:U-minus-U0-small} \eps^3 + C_\beta \eps^2 2 \Cr{const:secondary2} \eps.\]

 We are left to prove the last item, which is H\"older continuity of the restriction. We start with H\"older continuity along each
 $W^u_z$. By construction, if $r_1,r_2\in W^u_z\cap U_1=h_z(H_1)$ and their separation time is $s^+(r_1,r_2)=n$, then there is a singularity of order $n$ separating them, around which we already cut out an interval of length $c\theta^n$, meaning that $dist_{W^u_z}(r_1,r_2)\ge c\theta^n$. In short,
\begin{equation}\label{eq:dist-and-separation_time}
 dist_{W^u_z}(r_1,r_2)\ge c\theta^{s^+(r_1,r_2)}.
\end{equation}
Now H\"older continuity of the restriction of $G_0$ to any $h_z(H_1)$ follows from Proposition~\ref{prop:G_0-dynHolder}: for any $r_1,r_2\in h_z(H_1)$
\begin{eqnarray*}
 |G_0(r_1)-G_0(r_2)|\le C_{G;u} \frac{1}{\eps^2} ||\phi||_{\Theta_\phi;dH} \Theta_G^{s^+(r_1,r_2)}=
 C_{G;u} \frac{1}{\eps^2} ||\phi||_{\Theta_\phi;dH} \left(\theta^{s^+(r_1,r_2)}\right)^\frac{\ln\Theta_G}{\ln\theta} \le \\
 \le C_{G;u} \frac{1}{\eps^2} ||\phi||_{\Theta_\phi;dH} \left( \frac1c  dist_{W^u_z}(r_1,r_2) \right)^\frac{\ln\Theta_G}{\ln\theta} =
 C_{G;u} \frac{1}{\eps^2} ||\phi||_{\Theta_\phi;dH}\left( \frac1c \right)^{\alpha'} dist_{W^u_z}(r_1,r_2)^{\alpha'}
\end{eqnarray*}
with $\alpha':=\frac{\ln\Theta_G}{\ln\theta}>0.$ We can choose $\theta=\theta(\cR_Q,\Theta_G)$ sufficiently small to make sure that that $\alpha'\le\frac12$, so by our earlier choice $c=\eps^2$ we get (using $\eps\le diam(M)$) that
\begin{equation}\label{eq:G_0-Holder-on-W^u}
 |G_0(r_1)-G_0(r_2)|\le \Cl{const:G0-Holder-on-u} \frac{||\phi||_{\Theta_\phi;dH}}{\eps^3} dist_{W^u_z}(r_1,r_2)^{\alpha'}
\end{equation}
(with $\Cr{const:G0-Holder-on-u}=C_{G;u}\max\{diam(M),1\}$.)

So we are able to compare function values if $r_1$ and $r_2$ are on the same $u$-curve. On the other hand, for two points on the same central-stable manifold, Proposition~\ref{prop:G_0-Holder-on-W^cs} can be applied directly. For arbitrary $r_1,r_3\in U_1$ we combine the two by setting $r_2$ to be the only intersection point of the $u$-curve of $r_1$ and the central-stable manifold of $r_3$: if $r_1=h_{z_1}(x)$ and $r_3=h_{z_2}(y)$ with $x,y\in H_1$ and $z_1,z_2\in D$, then $r_2:=h_{z_1}(y)$. Transversality (by Theorem~\ref{thm:transversality}) of $W^u_{z_1}$ and $W^{c-s}_y$ guarantees that, with some $\Cl{const:transversal-distances}=\Cr{const:transversal-distances}(\cR_Q.\cR_u)$
 \[dist_{W^u_{z_1}}(r_1,r_2)\le \Cr{const:transversal-distances} |r_1-r_3|\]
 and
 \[dist_{W^{c-s}_x}(r_2,r_3)\le \Cr{const:transversal-distances} |r_1-r_3|.\]
 So (\ref{eq:G_0-Holder-on-W^u}) and Proposition~\ref{prop:G_0-Holder-on-W^cs} give
 \begin{eqnarray*}
 |G_0(r_1)-G_0(r_3)|\le |G_0(r_1)-G_0(r_2)| + |G_0(r_2)-G_0(r_3)|\le \\
 \le \Cr{const:G0-Holder-on-u} \frac{||\phi||_{\Theta_\phi;dH}}{\eps^3} dist_{W^u_z}(r_1,r_2)^{\alpha'} + \frac{C_{G-cs}\sup\phi}{\eps^3} dist_{W^{c-s}_x}(r_2,r_3)^{\alpha_{G_0}}\le \\
 \le \Cr{const:G0-Holder-on-u} \Cr{const:transversal-distances}^{\alpha'} \frac{||\phi||_{\Theta_\phi;dH}}{\eps^3} |r_1-r_3|^{\alpha'}+ \frac{C_{G-cs}\sup\phi}{\eps^3} \Cr{const:transversal-distances}^{\alpha_{G_0}} |r_1-r_3|^{\alpha_{G_0}}.
 \end{eqnarray*}
Choosing $\alpha_G:=\min\{\alpha', \alpha_{G_0}\}$, we use $|r_1-r_3|\le diam(M)$ and $\sup\phi\le ||\phi||_{\Theta_\phi;dH}$ from (\ref{eq:dynHolder-norm}) to get the result.
 \end{proof}

\begin{remark}\label{rem:G0-zero-on-boundary}
So far we only defined and considered $G_0$ at points of the product set $U_0=H\ast D$. However, Lemma~\ref{lem:G0-formula} shows that $G_0(r)=0$ whenever $r\in U_0\cap W^u_z$ with $|z|=\eps$, since then $q(z)=0$. So it is reasonable to set $G_0=0$ on all of $\bigcup\{W^u_z\,|\, |z|=\eps\}$, the curved surface of $U$. With this extension, $G_0$ is still clearly H\"older along the $u$-curves $W^u_z$, and Proposition~\ref{prop:G_0-Holder-on-H_1} remains true, with the proof unchanged.
\end{remark}

Now we extend $G_0|_{U_1}$ from $U_1$ to all of $M$.

\begin{proposition}\label{prop:G_Holder}
There is a global constant $\Cl{const:G-regularity}=\Cr{const:G-regularity}(\cR_Q,\cR_u)<\infty$ and a function $G:M\to\IR^+$ such that
 \begin{enumerate}
  \item \label{it:G-is-extension-of-G0} $G=G_0$ on $U_1$,
  \item \label{it:suppG-about-U_1} $\mu(supp(G)\setminus U_1)\le \Cr{const:G-regularity} \eps^3$,
  \item \label{it:G-is-Holder} $G$ is H\"older continuous with constants as $G_0$: for any $r_1,r_2 \in M$
  \[|G(r_1)-G(r_2)|\le \Cr{const:G_0-Holder-on-H_1} \frac{||\phi||_{\Theta_\phi;dH}}{\eps^3} |r_1-r_2|^{\alpha_G},\]
  where $\Cr{const:G_0-Holder-on-H_1}(\cR_Q,\cR_u)<\infty$ and $\alpha_G(\cR_Q,\cR_u,\Theta_{\phi})\le 1$ are given by Proposition \ref{prop:G_0-Holder-on-H_1},
  \item \label{it:G-bound} $0\le G \le \frac{\Cr{const:G-regularity}}{\eps^2} \sup_{W^u}\phi$,
  \item \label{it:IntG_almost_1} $\left| \int_{M}G\diff\mu -1\right|\le \Cr{const:G-regularity}\eps \sup_{W^u}\phi$,
  \item \label{it:G-Holder-norm} $||G||_{\alpha_G;H}\le \frac{\Cr{const:G-regularity}}{\eps^3} ||\phi||_{\Theta_\phi;dH}$.
 \end{enumerate}
\end{proposition}

\begin{proof}
Let us attach two semi-spheres of radius $\eps$ to the two flat faces of $U$ to get the set $U^+$, which is now a neighbourhood of $W^u$. We set $G:=0$ outside this $U^+$, so $supp(G)\subset U^+$. So
\[ \mu(supp(G)\setminus U_1) \le \mu(U^+\setminus U_1) = \mu(U^+\setminus U) + \mu(U\setminus U_1) \le \frac{4\eps^3 \pi}{3} + \Cr{const:G_0-Holder-on-H_1} \eps^3 \]
by Proposition~\ref{prop:G_0-Holder-on-H_1}, so item~\ref{it:suppG-about-U_1} holds. Setting $G=0$ outside $U^+$ does not spoil H\"older continuity of $G_0$: the function $G^*:(M\setminus U^+)\cup U_1$ defined as
\[ G^*(r):=\begin{cases}
              G_0(r), & \text{ if $r\in U_1$}\\
              0, & \text{ if $r\in M\setminus U^+$}
           \end{cases}
\]
satisfies $0\le G^* \le \sup G_0\le \frac{C_{G;b}}{\eps^2} \sup_{W^u}\phi$ by Proposition~\ref{prop:G_0-bounded}, and it
is H\"older continuous with the same constants as $G_0$ on $U_1$ as in Proposition~\ref{prop:G_0-Holder-on-H_1}, item~\ref{it:G_0-Holder-on-H_1}. This is so because for any point $r$ where $G^*$ is defined and possibly non-zero, the nearest point outside $U^+$ is surely on $\bigcup\{W^u_z\,|\, |z|=\eps\}$, so the regularity of $G^*$ follows from Remark~\ref{rem:G0-zero-on-boundary}.

Now the abstract extension lemma Lemma~\ref{lem:Holder-extension} ensures that $G^*$ can be extended to some $G$ defined on all of $M$ with all the required H\"older continuity and upper bound, so items \ref{it:G-is-extension-of-G0}, \ref{it:G-is-Holder} and \ref{it:G-bound} are shown.

To see item~\ref{it:IntG_almost_1} we use that $\phi$ is a probability density so $\int_{W^u}\phi\diff m_{W^u} =1$ and that $\int_{U_1}G\diff\mu=\int_{H_1}\phi\diff m_{W^u}$ by (\ref{eq:tildeG_0-def}), (\ref{eq:G_0-def}) and item \ref{it:G-is-extension-of-G0}. We get that
\[ \left| \int_{M}G\diff\mu -1\right| = \left| \int_{M}G\diff\mu - \int_{W^u}\phi\diff m_{W^u} \right| \le \int_{M\setminus U_1} G\diff\mu + \int_{W^u\setminus H_1}\phi\diff m_{W^u}. \]
The first term is bounded due to items \ref{it:suppG-about-U_1} and \ref{it:G-bound} of the present proposition, while the second term is bounded due to item~\ref{it:W-minus-H1-small} of Proposition~\ref{prop:G_0-Holder-on-H_1}, giving the statement of item~\ref{it:IntG_almost_1}.

Finally, item~\ref{it:G-Holder-norm} follows from items \ref{it:G-is-Holder}, \ref{it:G-bound} and the definitions (\ref{eq:Holder-norm}) and (\ref{eq:dynHolder-norm}).
\end{proof}

\subsection{Using the approximating density}\label{sec:using}

\subsubsection{Total error of the approximation}\label{sec:total-error}

\begin{lemma}\label{lem:G-phi_error_small}
There is a global constant $\Cl{const:G-phi_error_small}=\Cr{const:G-phi_error_small}(\cR_Q,\cR_u)<\infty$ such that
\[
\left| \int_M (F\circ \Phi^t) G \diff \mu - \int_M F\circ \Phi^t \diff \tilde\phi \right|\le \int\limits_{M} \left[ \left(osc_{C_\pi \eps} F\right) \circ \Phi^t\right]G\diff\mu + \Cr{const:G-phi_error_small} (\sup|F|) (\sup_{W^u}\phi) \eps,
\]
where $C_\pi(\cR_Q,\cR_u)<\infty$ if from Proposition~\ref{prop:integral_compare_on_H1-U1}.
\end{lemma}

\begin{proof}
We understand the integrals on $U_1$ and $H_1$ well, and the measure of the rest is small. So we write
 \begin{eqnarray*}
 &\,& \left| \int_M (F\circ \Phi^t) G \diff \mu - \int_M F\circ \Phi^t \diff \tilde\phi \right| = \\
 &=& \left| \int_{supp(G)} (F\circ \Phi^t) G \diff \mu - \int_{W^u} F\circ \Phi^t \diff \tilde\phi \right| \le \\
 &\le& \left| \int_{U_1} (F\circ \Phi^t) G_0 \diff \mu - \int_{H_1} F\circ \Phi^t \diff \tilde\phi \right| + \\
 &\,& \left| \int_{supp(G)\setminus U_1} (F\circ \Phi^t) G \diff \mu \right| + \left| \int_{W^u\setminus H_1} F\circ \Phi^t \diff \tilde\phi \right|.
 \end{eqnarray*}
The first term is estimated using Proposition~\ref{prop:integral_compare_on_H1-U1}.
For the second we use Proposition~\ref{prop:G_Holder}, items \ref{it:suppG-about-U_1} and \ref{it:G-bound}.
For the third we use Proposition~\ref{prop:G_0-Holder-on-H_1}, item~\ref{it:W-minus-H1-small}.
We get
 \begin{eqnarray*}
 &\,&\left| \int_M (F\circ \Phi^t) G \diff \mu - \int_M F\circ \Phi^t \diff \tilde\phi \right| \le \\
 &\le& \int\limits_{U_1} \left[ \left(osc_{C_\pi \eps} F\right) \circ \Phi^t\right]G_0\diff\mu + \\
 &\,& + (\sup |F|) (\sup_M G) \mu(supp(G)\setminus U_1) + (\sup|F|) (\sup_{W^u}\phi) m_{W^u}(W^u\setminus H_1)\le \\
 &\le& \int\limits_{M} \left[ \left(osc_{C_\pi \eps} F\right) \circ \Phi^t\right]G\diff\mu + \\
 &\,& + (\sup |F|) \frac{\Cr{const:G-regularity}}{\eps^2} (\sup_{W^u}\phi) \Cr{const:G-regularity}\eps^3 +
   (\sup|F|) (\sup_{W^u}\phi) \Cr{const:G_0-Holder-on-H_1} \eps = \\
 &=& \int\limits_{M} \left[ \left(osc_{C_\pi \eps} F\right) \circ \Phi^t\right]G\diff\mu +
   (\Cr{const:G-regularity}^2+\Cr{const:G_0-Holder-on-H_1}) (\sup|F|) (\sup_{W^u}\phi) \eps.
 \end{eqnarray*}
\end{proof}

\subsubsection{Completing the proof}\label{sec:end-of-proof}

We now have all the regularity estimates to complete the proof of our main Theorem~\ref{thm:main} using the approximating density $G$.

\begin{remark}
 The ``moreover'' part of Theorem~\ref{thm:main} discusses the dependence of the correlation decay on the billiard domain $Q$. To prove it, we needed and need to keep track of the $Q$-dependence of our ``constants''. Up to this point, every constant and every estimate depended on $Q$ through $\cR_Q$ (from (\ref{eq:RQ})) only. It is only the remaining final step where a (possibly) more complicated dependence appears, through the application of Theorem~\ref{thm:BDL} (more precisely, its corollary Theorem~\ref{thm:BDL_extended}).
\end{remark}

\begin{proof}[Proof of Theorem~\ref{thm:main}]
We need to estimate $\left| \int_M F\circ \Phi^t \diff \tilde\phi \right|$ from above. We do this using Lemma~\ref{lem:G-phi_error_small}, which implies
\begin{equation}\label{eq:main-thm-cov-estimate}
\left| \int_M F\circ \Phi^t \diff \tilde\phi \right|\le
\left| \int_M (F\circ \Phi^t) G \diff \mu \right| +
\int\limits_{M} \left[ \left(osc_{C_\pi \eps} F\right) \circ \Phi^t\right]G\diff\mu + \Cr{const:G-phi_error_small} (\sup|F|) (\sup_{W^u}\phi) \eps.
\end{equation}

The first term is estimated using Theorem~\ref{thm:BDL_extended} and Proposition~\ref{prop:G_Holder}, item~\ref{it:G-Holder-norm}. Since $\int_M F\diff\mu=0$, the result is
\begin{equation}\label{eq:int-F-estimate}
\left| \int_M (F\circ \Phi^t) G \diff \mu \right| \le
\cC_{BDL;g} var_{\alpha_F}(F) ||G||_{\alpha_G;H} e^{- a''t} \le
\cC_{BDL;g} var_{\alpha_F}(F) \frac{\Cr{const:G-regularity}}{\eps^3} ||\phi||_{\Theta_\phi;dH} e^{- a''t} .
\end{equation}
Here $\cC_{BDL;g}=\cC_{BDL;g}(Q,\alpha_F,\alpha_G)$ and $a''=a''(Q,\alpha_F,\alpha_G)$ both depend on $Q$, $\cR_u$, $\alpha_F$ and $\Theta_\phi$, since $\alpha_G=\alpha_G(\cR_Q,\cR_u,\Theta_\phi)$.

To estimate the second term, we use Theorem~\ref{thm:BDL_extended} again, with the same $\alpha_F$ and $\alpha_G$ (this choice is for convenience only). We get
\begin{equation}\label{eq:int-osc-F-estimate}
\int\limits_{M} \left[ \left(osc_{C_\pi \eps} F\right) \circ \Phi^t\right]G\diff\mu \le
\int_M osc_{C_\pi \eps} F \diff \mu \int_M G \diff \mu +
\cC_{BDL;g}\, var_{\alpha_F}(osc_{C_\pi\eps} F) ||G||_{\alpha_G;H} e^{- a''t}.
\end{equation}
The definitions (\ref{eq:gH-seminorm-def}) and (\ref{eq:var}) give
\[\int_M osc_{C_\pi \eps} F \diff \mu \le |f|_{\alpha_F;gH} (C_\pi \eps)^{\alpha_F}\le C_\pi^{\alpha_F} var_{\alpha_F} \eps^{\alpha_F}.\]
Item~\ref{it:IntG_almost_1} of Proposition~\ref{prop:G_Holder} and (\ref{eq:dynHolder-norm}) give
\[\int_M G \diff \mu \le 1 + \Cr{const:G-regularity}\eps \sup_{W^u}\phi \le
(L_{max} + \Cr{const:G-regularity} diam(M)) ||\phi||_{\Theta_\phi;dH}.\]
Corollary~\ref{cor:osc_F-genholder-norm} with $\alpha=\alpha_F$ and (\ref{eq:var}) give
\[
var_{\alpha_F}(osc_{C_\pi\eps} F) \le \frac{\Cr{const:osc_F-genholder-norm}}{\eps^{\alpha_F}} (\sup_M F - \inf_M F)\le
\frac{\Cr{const:osc_F-genholder-norm}}{\eps^{\alpha_F}} var_{\alpha_F}(F).
\]
Finally, item~\ref{it:G-Holder-norm} of Proposition~\ref{prop:G_Holder} gives
$||G||_{\alpha_G;H} \le \frac{\Cr{const:G-regularity}}{\eps^3} ||\phi||_{\Theta_\phi;dH}$.
{Substituting} to (\ref{eq:int-osc-F-estimate}), we get
\begin{equation}\label{eq:int-osc-F-estimate2}
\int\limits_{M} \left[ \left(osc_{C_\pi \eps} F\right) \circ \Phi^t\right]G\diff\mu \le
\C \cdot (1+\cC_{BDL;g}) var_{\alpha_F}(F) ||\phi||_{\Theta_\phi;dH} \left( \eps^{\alpha_F} + \frac{e^{- a''t}}{\eps^{3+\alpha_F}} \right).
\end{equation}
The third term of (\ref{eq:main-thm-cov-estimate}) is trivially bounded using (\ref{eq:var}) and (\ref{eq:dynHolder-norm}) as
\[
\Cr{const:G-phi_error_small} (\sup|F|) (\sup_{W^u}\phi) \eps \le
\Cr{const:G-phi_error_small} var_{\alpha_F}(F) ||\phi||_{\Theta_\phi;dH} \eps.
\]
{Substituting} this, (\ref{eq:int-F-estimate}) and (\ref{eq:int-osc-F-estimate2}) back to (\ref{eq:main-thm-cov-estimate}), we {obtain}
\begin{equation}\label{eq:main-thm-cov-estimate2}
\left| \int_M F\circ \Phi^t \diff \tilde\phi \right|\le
\Cl{const:main-thm-cov-estimate2} \cdot (1+\cC_{BDL;g}) var_{\alpha_F}(F) ||\phi||_{\Theta_\phi;dH} \left( \eps^{\alpha_F} + \frac{e^{- a''t}}{\eps^{3+\alpha_F}} \right).
\end{equation}

This holds for every $t\ge 0$ and every $0<\eps\le\eps_0$. Now to minimize the sum
$\eps^{\alpha_F} + \frac{e^{- a''t}}{\eps^{3+\alpha_F}}$,
we choose $\eps:=\eps(t)$ so that the two terms are equal, whenever this is allowed by the restriction
$\eps\le\eps_0$. Specifically, let
\begin{equation}\label{eq:a-def}
a:=\frac{\alpha_F}{3+2\alpha_F}a''
\end{equation}
and
\[\eps:=\eps(t):=\min\left\{e^{-\frac{a''}{3+2\alpha_F}t},\eps_0\right\}.\]
With this choice, if $e^{-\frac{a''}{3+2\alpha_F}t} \le \eps_0$, then
\[\eps^{\alpha_F} + \frac{e^{- a''t}}{\eps^{3+\alpha_F}}=2\eps^{\alpha_F} = 2 e^{-at}.\]
If $e^{-\frac{a''}{3+2\alpha_F}t} \ge \eps_0$, then a short calculation gives
\[
\eps^{\alpha_F} + \frac{e^{- a''t}}{\eps^{3+\alpha_F}}=
\left[ \left(\frac{\eps_0}{e^{-\frac{a''}{3+2\alpha_F}t}}\right)^{\alpha_F} +
  \left(\frac{e^{-\frac{a''}{3+2\alpha_F}t}}{\eps_0}\right)^{3+\alpha_F} \right] e^{-at} \le
\left[ 1 + \left(\frac{1}{\eps_0}\right)^{3+\alpha_F} \right] e^{-at}
.\]
{Substituting} to (\ref{eq:main-thm-cov-estimate2}), we get
\begin{equation}
\left| \int_M F\circ \Phi^t \diff \tilde\phi \right|\le
\Cr{const:main-thm-cov-estimate2} \cdot (1+\cC_{BDL;g})
\left[ 2 + \left(\frac{1}{\eps_0}\right)^{3+\alpha_F} \right]
var_{\alpha_F}(F) ||\phi||_{\Theta_\phi;dH} e^{-at}.
\end{equation}
So the main statement of the theorem is proven with
\begin{equation}\label{eq:cC-def}
\cC=\Cr{const:main-thm-cov-estimate2} \cdot (1+\cC_{BDL;g})
\left[ 2 + \left(\frac{1}{\eps_0}\right)^{3+\alpha_F} \right].
\end{equation}

To see the ``moreover'' part of the theorem, note that the $Q$-dependence of $\cC_{BDL;g}$ and $a''$ is described by the ``moreover'' part of Theorem~\ref{thm:BDL_extended}. With this, (\ref{eq:a-def}) gives
\[a=\frac{\alpha_F}{3+2\alpha_F} \frac{\alpha_F}{\alpha_F+1} a'(Q,\alpha_G(\cR_Q,\cR_u,\Theta_\phi))\]
and (\ref{eq:cC-def}) gives
\[\cC=\cC(\cR_Q,\cR_u,\alpha_F,\cC_{BDL}(Q,\alpha_G(\cR_Q,\cR_u,\Theta_\phi))).\]
Now the ``moreover'' part is shown with $\alpha=\alpha_G$.
\end{proof}

\section{Possible extension, open problems}\label{sec:extensions}

\subsection{Dependence of constants on the billiard table}\label{sec:ext_Q-dependence}

As mentioned in the introduction, equidistribution theorems, like the one in this paper, are sometimes applied to a class of models simultaneously. In such situations, it is useful to know that the same estimate holds for all the models in the class. Unfortunately, our main Theorem~\ref{thm:main} says little about the dependence of the ``constants'' $\cC$ and $a$ on the billiard table $Q$. The only reason for this is that our main reference, Theorem~\ref{thm:BDL} from \cite{BDL16} does not say anything about the $Q$-dependence of the constants $\cC_{BDL}$ and $a'$. If this dependence was better understood -- say, we would know that $\cC_{BDL}$ and $a'$ depend on $Q$ through $\cR_Q$ only, -- then a better understanding of the $Q$-dependence of $\cC$ and $a$ in Theorem~\ref{thm:main} would be automatic -- see the ``moreover'' part of Theorem~\ref{thm:main}.

However, there is good reason that \cite{BDL16} does not discuss the $Q$-dependence of constants. At the heart of their functional analytic proof is a compactness argument, which allows only finitely many eigenvalues of the transfer operator in a neighbourhood of the imaginary axis.
This ensures a spectral gap, but gives no control on the size of that gap. Getting explicit bounds seems difficult at least.

Another possible way to control the $Q$-dependence of equidistribution would be to give up exponential decay, and prove stretched exponential equidistribution only. This is clearly not optimal, but also very useful technically, and can be proven with different methods. Indeed, if we want stretched exponential decay only, we can, instead of \cite{BDL16}, refer to Theorem 1.1 for \cite{Ch07}:

\begin{theorem}\label{thm:Ch}
Consider a billiard like the one in Theorem~\ref{thm:main}. Assume $0 < \alpha \le 1$. Then there exist $a' = a'(Q,\alpha) > 0$ and $\cC_{Ch}=\cC_{Ch}(Q,\alpha) < \infty$ such that for any
$F,G:M\to \IR$ generalized $\alpha$-H\"older functions with $\int_M F\diff\mu=0$ and any $t\ge 0$ one has
 \[\left| \int_M (F\circ \Phi^t) G \diff \mu  \right| \le \cC_{Ch} var_\alpha(F) var_\alpha(G) e^{- a'\sqrt{t}}. \]
Here $var_\alpha(.)$ denotes the generalized $\alpha$-H\"older seminorm defined in (\ref{eq:var}).
(Ch stands for Chernov.)
\end{theorem}

This theorem, combined with lemmas \ref{lem:holder_genholder} and \ref{lem:genholder-genholder} has the following obvious corollary:

\begin{corollary}\label{cor:Ch-extended}
A planar billiard flow with finite horizon and no corner points enjoys stretched exponential correlation decay for a pair of H\"older and a generalized H\"older observables. Quantitatively, let $0 < \alpha_F,\alpha_G \le 1$. Then there exists an $a'' = a''(Q, \alpha_F,\alpha_G) > 0$ and a $\cC_{Ch;g}=\cC_{Ch;g}(Q,\alpha_F,\alpha_G) < \infty$ such that if $F:M\to \IR$ is $\alpha_F$-generalized H\"older and $G:M\to \IR$ is $\alpha_G$-H\"older, then for any $t\ge 0$
 \[\left| \int_M (F\circ \Phi^t) G \diff \mu - \int_M F \diff \mu \int_M G \diff \mu \right| \le \cC_{Ch;g} var_{\alpha_F}(F) ||G||_{\alpha_G;H} e^{- a''\sqrt{t}}. \]
(Here again Ch stands for Chernov.)

Moreover, $a''(Q,\alpha_F,\alpha_G)=a'(Q,\min\{\alpha_F,\alpha_G\})$ and
\[\cC_{Ch;g}(Q,\alpha_F,\alpha_G)= \max\{(diam(M))^{|\alpha_G-\alpha_F|},1\} \, \cC_{Ch}(Q,\min\{\alpha_F,\alpha_G\})\]
where $a'$ and $\cC_{Ch}$ are from Theorem~\ref{thm:Ch}.
\end{corollary}

Now the stretched exponential version of our main Theorem~\ref{thm:main} follows:

\begin{theorem}\label{thm:stretched} Suppose the billiard table $Q$ satisfies assumptions \ref{asm:dispersing}, \ref{asm:finite-horizon} and \ref{asm:no-corner}. Let $0<\Theta_\phi<1$, $0<\alpha_F \le 1$ and $\eps_0>0$. Then there exist $\cC<\infty$ and $a>0$ with the following properties:

Let $(W^u,\phi)$ be a standard pair with a dynamically $\Theta_\phi$-H\"older $\phi$. Assume that $W^u$ is a \emph{good} $u$-curve. Let $F:M\to \IR$ be generalized $\alpha_F$-H\"older continuous. Then, for every $t\ge 0$,
 \[\left| \int_{W^u} \left(F\circ \Phi^t\right) \phi \diff m_{W^u} - \int_M F\diff\mu \right| \le \cC \cdot ||\phi||_{\Theta_\phi;dH}\cdot var_{\alpha_F}F\cdot  e^{-a \sqrt{t}}.\]
Here
 \begin{itemize}
  \item $a=a(Q,\cR_u,\Theta_\phi,\alpha_F)<\infty$ and $\cC=\cC(Q,\cR_u,\Theta_\phi,\alpha_F)<\infty$ depend on the billiard table $Q$, the regularity of good $u$-curves quantified in $\cR_u$, and the regularity classes of $\phi$ and $F$ given by $\Theta_\phi$ and $\alpha_F$. They do not depend on $W^u$, $\phi$ and $F$.
 \end{itemize}
Moreover, $\cC$ depends on $Q$ only through $\cR_Q$ from (\ref{eq:RQ}) and $\cC_{Ch}(Q,\alpha)$ from Theorem~\ref{thm:Ch} with some $\alpha=\alpha(\cR_Q,\cR_u,\Theta_\phi,\alpha_F)>0$.
Similarly, $a$ depends on $(Q,\cR_u,\Theta_\phi)$ only through $a'(Q,\alpha)$ from Theorem~\ref{thm:Ch} with the same $\alpha=\alpha(\cR_Q,\cR_u,\Theta_\phi,\alpha_F)>0$. That is,
\[\cC=\cC(\cR_Q,\cR_u,\alpha_F,\cC_{Ch}(Q,\alpha(\cR_Q,\cR_u,\Theta_\phi,\alpha_F)))\]
and
\[a=a(\alpha_F, a'(Q,\alpha(\cR_Q,\cR_u,\Theta_\phi,\alpha_F))).\]
\end{theorem}

Note that there are only two differences between this theorem and the main Theorem~\ref{thm:main}. First, there is $\sqrt{t}$ instead of $t$ in the exponent. Second, the description of the $Q$-dependence of constants in the ``moreover'' part now refers to Theorem~\ref{thm:Ch} instead of Theorem~\ref{thm:BDL}.

\begin{proof}
 The proof of Theorem~\ref{thm:main} from Section~\ref{sec:end-of-proof} {applies
 up to the proof of the main statement in (\ref{eq:cC-def}), with the following minor modifications:}
 \begin{itemize}
  \item every $t$ should be replaced by $\sqrt{t}$,
  \item every $\cC_{BDL;g}$ should be replaced by $\cC_{Ch;g}$,
  \item every reference to Theorem~\ref{thm:BDL_extended} should be replaced by a reference to Corollary~\ref{cor:Ch-extended}.
 \end{itemize}
To see the ``moreover'' part, we use the ``moreover'' part of Corollary~\ref{cor:Ch-extended}. With this, (\ref{eq:a-def}) gives
\[a=\frac{\alpha_F}{3+2\alpha_F} a'(Q,\min\{\alpha_F,\alpha_G(\cR_Q,\cR_u,\Theta_\phi)\})\]
and (\ref{eq:cC-def}) gives
\[\cC=\cC(\cR_Q,\cR_u,\alpha_F,\cC_{Ch}(Q,\min\{\alpha_F,\alpha_G(\cR_Q,\cR_u,\Theta_\phi)\})).\]
Now the ``moreover'' part is shown with $\alpha=\min\{\alpha_F,\alpha_G\}$.
\end{proof}

Just like in the case of our main Theorem~\ref{thm:main}, if we were able to better understand the functions $a'$ and $\cC_{Ch}$ in Theorem~\ref{thm:Ch}, then we would also know the $Q$-dependence of $a$ and $\cC$ better. This problem is currently open: as a first step, a detailed analysis of the proof in \cite{Ch07} is needed to see what $Q$-dependence can be read out of it.

\subsection{Corner points}\label{sec:ext-corner}

A weakness of the main Theorem~\ref{thm:main} is that it requires Assumption~\ref{asm:no-corner}, so it does not cover the case of corner points. Again, the main reason is that our main reference Theorem~\ref{thm:BDL} does not cover corner points either. However, several regularity properties that we use, are also known only in the case of no corner points. We do not attempt to generalize these here. Instead, we discuss the points where Assumption~\ref{asm:no-corner} is used, and state two conditional theorems, describing what needs to be checked to cover the case of corner points.

If there are corner points, there is no lower bound $\tau_{min}$ for the free flight. Instead of Assumption~\ref{asm:no-corner}, we require Assumption~\ref{asm:no-cusp} below to rule out cusps.~\footnote{According to the terminology of \cite{ChM06}, billiards satisfying assumptions \ref{asm:dispersing}, \ref{asm:finite-horizon} and \ref{asm:no-cusp} belong to category C.}

\begin{assumption}[No cusps]
\label{asm:no-cusp}
 For every corner point $r\in \Gamma_i\cap\Gamma_j$, the angle of $\Gamma_i$ and $\Gamma_j$ at $x$ is non-zero.
\end{assumption}

Then there is some lower bound $\xi_{min}$ for the angle of the two smooth components of $\partial Q$ that meet at a corner. Also, let $d_{min}$ be the minimal free flight between two scatterers that do \emph{not} form a corner. (If $d_{min}$ is very small, there is a ``bottleneck'' situation between two scatterers that almost touch. Such a situation is close to a cusp.)

So the (hopefully sufficient) set of regularity parameters for the table $Q$ is now
\begin{equation}\label{eq:RQ-corner}
\cR_Q^{corner}:=\{\tau_{max},d_{min},\xi_{min},\kappa_{min},\kappa_{max},\kappa'_{max},K_{max},A_{min},A_{max},d_Q\}
\end{equation}
(compare (\ref{eq:RQ})).

The inputs of our argument -- apart from the main reference \cite{BDL16} -- are the regularity properties of $u$-curves stated in sections \ref{sec:hyp-prop} and \ref{sec:holonomy}. Of these, the elementary ones in Section~\ref{sec:hyp-prop} are known under Assumption~\ref{asm:no-cusp} (no cusps) instead of Assumption~\ref{asm:no-corner}, so there is nothing to do, see also Remark~\ref{rem:aligncorner}. On the other hand, the advanced regularity properties of Section~\ref{sec:holonomy} are proven in Section~\ref{sec:Ahom} with reference to Theorem 5.67, Proposition 5.48 and Theorem 5.42 from \cite{ChM06}. These are definitely known only for billiards with no corner points.

To get the three statements of Section~\ref{sec:holonomy} from these statements of \cite{ChM06}, our arguments (in Section~\ref{sec:Ahom}) do not use Assumption~\ref{asm:no-corner}. That is, to prove them for {billiards with corner points}, it would be enough to prove the corner point generalizations of the three statements of \cite{ChM06}. This justifies the following conditional extension of our main theorem:

\begin{theorem}\label{thm:corner} Suppose the billiard table $Q$ satisfies assumptions \ref{asm:dispersing}, \ref{asm:finite-horizon} and \ref{asm:no-cusp}. Assume that
\begin{enumerate}
 \item the generalization of Theorem~\ref{thm:long_cs_manifolds}, or alternatively, Theorem 5.67 from \cite{ChM06} is proven for this class of billiards, with $\cR_Q$ replaced by $\cR_Q^{corner}$,
 \item the generalization of Theorem~\ref{thm:holonomy_dynholder}, or alternatively, Proposition 5.48 from \cite{ChM06} is proven for this class of billiards, with $\cR_Q$ replaced by $\cR_Q^{corner}$,
 \item the generalization of Theorem~\ref{thm:holonomy_cs-regularity}, or alternatively, Theorem 5.42 from \cite{ChM06} is proven for this class of billiards, with $\cR_Q$ replaced by $\cR_Q^{corner}$,
 \item the generalization of Theorem~\ref{thm:BDL} (which is Corollary 1.3 in \cite{BDL16}) is proven for this class of billiards.
\end{enumerate}
Then there is exponential equidistribution for the evolution of the standard pair under the flow:

Let $0<\Theta_\phi<1$, $0<\alpha_F \le 1$ and $\eps_0>0$. Then there exist $\cC<\infty$ and $a>0$ with the following properties:

Let $(W^u,\phi)$ be a standard pair with a dynamically $\Theta_\phi$-H\"older $\phi$. Assume that the $2\eps_0$-neighbourhood of $W$ is disjoint from the boundary of $M$. Let $F:M\to \IR$ be generalized $\alpha_F$-H\"older continuous. Then, for every $t\ge 0$,
 \[\left| \int_{W^u} \left(F\circ \Phi^t\right) \phi \diff m_{W^u} - \int_M F\diff\mu \right| \le \cC \cdot ||\phi||_{\Theta_\phi;dH}\cdot var_{\alpha_F}F\cdot  e^{-a t}.\]
Here
 \begin{itemize}
  \item $a=a(Q,\cR_u,\Theta_\phi,\alpha_F)<\infty$ and $\cC=\cC(Q,\cR_u,\Theta_\phi,\alpha_F)<\infty$ depend on the billiard table $Q$, the regularity of good $u$-curves quantified in $\cR_u$, and the regularity classes of $\phi$ and $F$ given by $\Theta_\phi$ and $\alpha_F$. They do not depend on $W^u$, $\phi$ and $F$.
 \end{itemize}
Moreover, $\cC$ depends on $Q$ only through $\cR_Q^{corner}$ from (\ref{eq:RQ-corner}) and $\cC_{BDL}(Q,\alpha)$ from the generalization of Theorem~\ref{thm:BDL} with some $\alpha=\alpha(\cR_Q^{corner},\cR_u,\Theta_\phi)>0$.
Similarly, $a$ depends on $(Q,\cR_u,\Theta_\phi)$ only through $a'(Q,\alpha)$ from the generalization of Theorem~\ref{thm:BDL} with the same $\alpha=\alpha(\cR_Q^{corner},\cR_u,\Theta_\phi)>0$. That is,
\[\cC=\cC(\cR_Q^{corner},\cR_u,\alpha_F,\cC_{BDL}(Q,\alpha(\cR_Q^{corner},\cR_u,\Theta_\phi)))\]
and
\[a=a(\alpha_F, a'(Q,\alpha(\cR_Q^{corner},\cR_u,\Theta_\phi))).\]
\end{theorem}

Similarly to the situation in Section~\ref{sec:ext_Q-dependence}, we don't know if the generalization of Theorem~\ref{thm:BDL} to the case of corner points can reasonably be done. An alternative can be to treat the corner points with the method of \cite{Ch07}, and get stretched exponential equidistribution only. This would be good enough for most applications. The analogue of Theorem \ref{thm:corner} for this case is straightforward.

\subsection{Standard pairs close to scatterers}\label{sec:no-eps0}

Theorem~\ref{thm:main}, as we stated and proved it, requires that the $u$-curve $W^u$ is at least $\eps_0$ far away from collision points. This is clearly a weakness from the application point of view. However, it is not a serious restriction, because if $\eps_0$ is small enough, then any short $u$-curve $W^u$, which is possibly too close to a scatterer (or even in the process of being reflected from it), will soon evolve (under $\Phi^{t_0}$ with some small $t_0$) either into a curve which is more than $\eps_0$ away, or possibly a bounded number of such curves. (At most two if there are no corner points.) At the same time, $\phi$ evolves into some density $\phi'$ on $\Phi^{t_0} W^u$. Now we can apply Theorem~\ref{thm:main} to the pair $(\Phi^{t_0} W^u,\phi')$, provided that it is a standard pair, and the regularity can be quantitatively checked.

This requires a study of the time evolution of standard pairs (for finite times), especially the evolution of their regularity parameters. This was not needed for the proof of our theorem, and is thus not done here. The notion of standard pairs (including the notion of $u$-curves) was designed with this need in mind, and we expect the notion to be time-invariant in some good sense. We leave this to a future paper that applies this result.

\section{Acknowledgement} This research was supported by Hungarian National Foundation for Scientific
Research grant No. K 104745 and OMAA-92\"ou6 project.

\renewcommand\thesection{\Alph{section}}
\setcounter{section}{0}
\section{Appendix}

\subsection{Holonomy}\label{sec:Ahom}

Here we provide the proofs for the three theorems stated in Section~\ref{sec:holonomy}, namely theorems \ref{thm:long_cs_manifolds}, \ref{thm:holonomy_dynholder} and \ref{thm:holonomy_cs-regularity}.

\begin{notation}
 Since the arguments in this section are independent from the bulk of the paper, we will use slightly different notation for better readability. In particular, we will use $\phi$ for the angle of the velocity with the normal vector of the scatterer at collision points, as usual in the billiard literature.
\end{notation}

These theorems formulate regularity properties of $u$-curves (of the flow phase space). For simplicity we reduce their proofs to regularity properties of discrete time $u$-curves. The direct reduction to those of $u$-curves of the billiard ball map raises a technical problem since  the transition between the discrete time and the flow $u$-curves has bad regularity near tangent collisions: the flight time that connects them has unbounded derivative. There is an additional fact we want to emphasize: our reference \cite{ChM06} states the theorems in a form which is a little weaker than what they actually prove, and what we need. In particular, they use the term ``constant'' for numbers which depend on the billiard table $Q$ only, but they do not discuss the form of this dependence. To see that their constants actually depend on $Q$ through $\cR_Q$ from (\ref{eq:RQ}), we need to look into the proofs.

Taking the aforementioned circumstances into consideration our approach will still be to reduce our statements to the discrete time analogues by introducing so-called \emph{transparent walls} and relying on regularity properties of  $u$-curves living on them. In this way we can still use the results of \cite{ChM06} by also avoiding unbounded derivatives.
This means that we extend the discrete time phase space $\cM$ with an extra component $\cM_{tr}=\partial Q_{tr}\times [-\frac{\pi}{2},\frac{\pi}{2}]$.
The motion of the particle remains unchanged, but when $\cM_{tr}\subset M$ is reached, that is treated as a collision
(note that $\cM_{tr}$ is in the phase space, i.e. we only keep track of passing through the transparent wall
in one direction). Since the statements we want to prove are local, it is enough to set up the transparent wall near the $u$-curves we study. See Figure~\ref{fig:transparent_wall}.
\begin{figure}[hbt]
 \psfrag{T_tr^2-q}{$T_{tr}^2 q$}
 \psfrag{T_tr-q}{$T_{tr}q$}
 \psfrag{pQ_tr}{$\partial Q_{tr}$}
 \psfrag{gamma}{$\gamma$}
 \psfrag{W}{$W$}
 \psfrag{q}{$q$}
 \centering
\includegraphics[width=5cm]{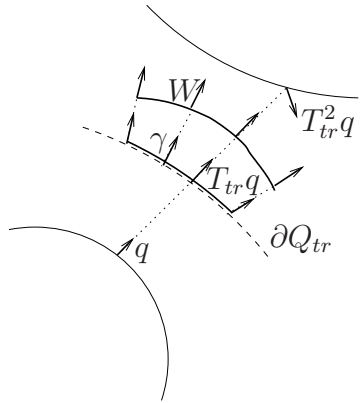}
 \caption{A transparent wall with the trace of a $u$-curve}
 \label{fig:transparent_wall}
\end{figure}

The new map $T_{tr}:\cM\cup \cM_{tr}\to \cM\cup \cM_{tr}$ is very similar to the billiard dynamics, and the theorems we use from \cite{ChM06} remain valid. The new domain $\partial Q \cup \partial Q_{tr}$ will depend on the $u$-curves under consideration, but we take care that the constants in the estimates remain uniform.

In return, we can always choose the new wall so that it has a big enough angle with the trajectories considered. This makes it easy to show the necessary regularity properties for the transition map between the flow curves and their traces on the transparent wall.

Before the specific calculations, we ask the reader to recall the notations $\cR_Q$ and $\cR_u$ (formulas (\ref{eq:RQ}) and (\ref{eq:Ru})).

\begin{convention}
\label{conv:GeomConstantsHere}
As discussed in Section~\ref{sec:const-convention}, $C$ denotes constants that depend only on $Q$ and $\cR_u$, also called geometric constants. In this section, dependence on $Q$ is always through $\cR_Q$ only, so $C=C(\cR_Q,\cR_U)$, even if this is not indicated. On the other hand, the exact value of $C$ may change from line to line (even within a line).
\end{convention}

Throughout, we consider \textit{good $u$-curves} $W$ for the flow (cf.~Definitions~\ref{def:u-curve} and~\ref{def:gooducurve}). In particular
\[
W=\{p(s)=(x(s),y(s),v_x(s),v_y(s))|s\in I \}.
\]
This generates
$W_0=\{(x(s),y(s))\}$, a curve in the configuration domain. In fact, we have
\[(v_x(s),v_y(s))=(\cos\omega(s),\sin\omega(s)),\]
as the {the speed is one}. Note that the curvature of $W_0$ (as a curve in the configuration space) is bounded from below by some constant $B_{\rm{min}}=B_{\rm{min}}(\cR_Q,\cR_u)$.
To describe the tangent vectors in ${\mathcal T}_{p}W$ it is convenient to use the Jacobi coordinates $(d\xi,d\eta,d\omega)$, where $d\xi=-\sin\omega dx+\cos\omega dy$ and
$d\eta=\cos\omega dx+\sin\omega dy$. Note that the arc length along $W$ can be computed by integrating $\sqrt{d\xi^2+d\eta^2+d\omega^2}$.

Now we construct the transparent wall associated to the good $u$-curve $W$. Take one particular point $p(s_0)=(x(s_0),y(s_0))$ in the interior of $W_0$, and consider the unit (velocity) vector $v(s_0)= (v_x(s_0),v_y(s_0))$. Consider, furthermore the point $p'=p(s_0)-L v(s_0)$ in the configuration space. Let $\partial Q_{\rm{tr}}$ denote the circular arc of length $10 L$, and curvature $B_{\rm{min}}/2$ centred at $p'$ {with normal vector $v(s_0)$ at $p'$}.  Then $\partial Q_{\rm{tr}}$ is our transparent wall, which will be regarded as a scatterer. Accordingly, the billiard map phase space $\cM$ can be extended with the associated component $\cM_{\rm{tr}}$, which is parametrized by the usual $(r,\phi)$ (configuration and velocity) coordinates. $\partial Q_{\rm{tr}}$ has the following properties:

\begin{itemize}
\item The trajectories emerging from the points of $W$ necessarily intersect $\partial Q_{\rm{tr}}$ for some negative time.
\item Let $\gamma=\{q(s)=(r(s),\phi(s))\}$ denote the resulting trace of $W$ on $\cM_{\rm{tr}}$. Then $\gamma$ is a $u$-curve for this Poincar\'e section.
\item From now on, with a slight abuse of notation, let $s$ denote arc length along $\gamma$.  We have $W=\{p(s)=(q(s),\tau(s))\}$ where, by bounded curvature, $\tau(s)$ is $C^1$ and $\tau'(s)$ is Lipschitz with some constant $C$ independent of $W$. The canonical projection $\pi$ provides a one-to-one correspondence between $\gamma$ and $W$.
\item The above properties hold not only for $W$ but for any other good $u$-curve $\bar{W}$ sufficiently close to $W$: $\bar{W}$ will leave a trace $\bar{\gamma}=\pi(\bar{W})$ on $\partial Q_{\rm{tr}}$ when flown backwards in time; $\bar{\gamma}$ is a $u$-curve for the Poincar\'e section $\cM_{\rm{tr}}$ and $\bar{W}=\{\bar{p}(s)=(\bar{q}(s),\bar{\tau}(s))\}$, with $\bar{q}(s)$ tracing $\bar{\gamma}$, and $\bar{\tau}$ $C^1$ with a uniformly Lipschitz continuous derivative.
\end{itemize}

\begin{convention}
\label{conv:GeomConstantsHere2}
In what follows we will work with the billiard table that has the additional {transparent wall} $\partial Q_{\rm{tr}}$. The regularity parameters (cf. (\ref{eq:RQ})) of this
extended {billiard} configuration are denoted by $\cR_{Q,\partial Q_{\rm{tr}}}$. These depend, a priori, on $\partial Q_{\rm{tr}}$. Nonetheless, by construction, and as $\partial Q_{\rm{tr}}$ is associated to a good $u$-curve, we have $\cR_{Q,\partial Q_{\rm{tr}}}=\cR_Q\cup\cR_u$. Note that $\cR_u$ includes the regularity parameter $\eps_0$, which is the minimum distance of the good $u$-curves from the scatterers.
\end{convention}

The proofs of Theorems~\ref{thm:long_cs_manifolds},~\ref{thm:holonomy_dynholder} and Theorem~\ref{thm:holonomy_cs-regularity} are essentially the reduction to analogous statements for the billiard map, discussed in \cite{ChM06} as Theorem 5.67, Proposition 5.48 and Theorem 5.42, respectively.  In fact, in the formulation of these statements in \cite{ChM06}, the constants $C$ depend on the billiard domain $Q$.
In particular, a priori, in our setting this implies a dependence on the $u$-curve $W_1^u$ via the choice of the transparent wall $\partial Q_{\rm{tr}}$. However, {a careful study of} the relevant arguments in \cite{ChM06} {reveals} that the constants $C$ depend only on the regularity parameters of the billiard domain (in the terminology of Section~\ref{sec:regpar}). To see this, we also refer to Appendix A.2. in \cite{ChD09A} -- in particular Extension 1 -- where the analogous property of the constants is explicitly stated.\footnote{The statement in \cite{ChD09A} concerns the rate of correlation decay, yet the proof of exponential decay of correlations heavily relies on the regularity properties of the holonomy map.}
Thus, in our case, we have $C=C(\cR_{Q\cup\partial Q_{\rm{tr}}})$ in the statements. By Convention~\ref{conv:GeomConstantsHere2}, this means $C=C(\cR_Q,\cR_u)$, so not only are the estimates uniform in $W^u$, but the constants only depend on $Q$ through $\cR_Q$, as desired by the theorems we are proving.

To reduce our statements to their discrete time analogues (with the help of the transparent wall $\partial Q_{\rm{tr}}$), we first have to discuss the regularity of the projection $\pi$ of $W$ to $\gamma\subset\cM_{\rm{tr}}$.

\begin{lemma}
\label{lem:JacobianMapFlow}
The map $\pi$ is absolutely continuous and its Jacobian $J_{\gamma\to W}(s)$ satisfies
\[C^{-1}<J_{\gamma\to W}(s)<C\quad \text{and} \quad |J_{\gamma\to W}(s)-J_{\gamma\to W}(s')|\le C|s-s'|\]
for some uniform $C$.
\end{lemma}

\begin{proof}
Below we obtain an explicit formula for the Jacobian $J_{\gamma\to W}(s)$ in several steps.
For any $q(s)\in\gamma$ and $p(s)=\pi^{-1}(q(s))\in W$, we need to relate the line element $ds$ along $\gamma$ to the length of the corresponding tangent vector, $(d\xi,d\eta,d\omega)$, of $W$ at $p(s)$. For brevity the dependence on $s$ is omitted {whenever there is no risk of ambiguity} and the following notations are introduced:
\begin{itemize}
 \item $m=\frac{d\phi}{dr}$, the slope of $\gamma$ in the Poincar\'e section.
 \item $K=K(r)$, the curvature of $\partial Q_{\rm{tr}}$ at the point of collision.
 \item $W_+$ the post-collision front (or local orthogonal manifold) emerging from $\gamma$, and  $(d\xi_+,0,d\omega_+)$, its tangent vector.
 \item $B_+$, the curvature of the post-collision front. In particular, $B_+=\frac{d\omega_+}{d\xi_+}$.
 \end{itemize}

 The following relations are standard, see \cite{ChM06}:
 \[
 ds=\sqrt{1+m^2} \cdot dr;\qquad d\xi_+=\cos\phi \cdot dr; \qquad m=B_+\cos \phi -K.
 \]

 By construction
 \begin{itemize}
 \item $\cos\phi(r)$ is bounded away from $0$,
\item $m(r)$ is bounded from above, while $B_+(r)$ is bounded away from $0$ and infinity,
\item all these quantities depend on $r$ {in a Lipschitz continuous manner}, with uniformly bounded Lipschitz constant,
\item $ds$, $dr$ and $d\xi_+$ are uniformly equivalent.
\end{itemize}
Hence $d\xi_+=\frac{\cos\phi(r)}{\sqrt{1+m(r)^2}}\cdot ds$, and as $d\omega_+=B_+(r) d\xi_+$, the Jacobian $J_{\gamma\to W_+}$ is bounded away from $0$ and infinity, and Lipschitz with uniformly bounded Lipschitz constant.

The free flight evolution of $W_+$ into $W$ results in
\[
d\xi=d\xi_++\tau d\omega_+=(1+\tau B_+) d\xi_+, \qquad  d\omega_+=d\omega.
\]
We still need to compute $d\eta$. By \cite{ChM06}, section 3.3, we have
\[
d\eta=\sin\phi dr+d\tau=(\sin\phi +\tau'(r))dr.
\]
Now we are in the position to compute the Jacobian $J_{\gamma\to W}=\frac{\sqrt{d\xi^2+d\eta^2+d\omega^2}}{ds}$. As
\begin{eqnarray}
 \label{eq:J0expl}
   \sqrt{d\xi^2+d\eta^2+d\omega^2}&=&\sqrt{(1+\tau B_+)^2 +B_+^2 +\frac{(\sin\phi+\tau')^2}{\cos^2\phi}} \cdot \cos\phi \, dr \\ \nonumber
    &=& \frac{\sqrt{(1+\tau B_+)^2\cos^2\phi +B_+^2 \cos^2\phi +(\sin\phi+\tau')^2}}{\sqrt{1+m^2}}\cdot ds
 \end{eqnarray}

our analysis implies that the Jacobian $J_{\gamma\to W}(s)$ is uniformly bounded away from $0$ and infinity, and depends on $s$ {in a Lipschitz continuous manner}, with uniformly bounded Lipschitz constant.
\end{proof}
\medskip

\begin{proof}[Proof of Theorem~\ref{thm:long_cs_manifolds}]
Let us mention first that it would be possible to prove Theorem~\ref{thm:long_cs_manifolds} by discussing the flow directly, as in the proof of the related Proposition 6.10 from \cite{ChM06}. Yet, in accordance with the other arguments presented here, we prefer to {deduce it from the corresponding statement for the map, i.e. Theorem 5.67 in \cite{ChM06}}.

Consider a good $u$-curve $W$, the associated transparent wall $\partial Q_{\rm{tr}}$ and the trace $\gamma=\pi(W)$. For $x\in W$, let $r^{c-s}(x)$ denote the inner radius of the homogeneous local central stable manifold centred at $x$, while for $p\in\gamma$, let $r^s(p)$ denote the inner radius of the homogeneous stable manifold (of the billiard map $T$ with the new transparent wall) centred at $p$. Fix some $\eps>0$ and introduce
\[
B_{W,\eps}=\{x\in W|r^{c-s}(x)\le \eps\};\qquad
B_{\gamma,\eps}=\{x\in \gamma|r^s(x)\le \eps\}.
\]
Now we refer to Theorem~5.67 from \cite{ChM06} which ensures that
\begin{equation}\label{CM:longs}
m_{\gamma}(B_{\gamma,\eps})\le C\eps.
\end{equation}
On the other hand, Lemma~\ref{lem:JacobianMapFlow} ensures that
\begin{equation}\label{longs-longcs-measure}
m_{W}(B_{W,\eps})\le C m_{\gamma}(\pi(B_{W,\eps})).
\end{equation}
By construction of $\partial Q_{\rm{tr}}$, the flow time between $\gamma$ and $W$ is uniformly bounded, and there exists some constant $\Cl{const:BP1}$ such that
if $r^s(\pi(x))>\Cr{const:BP1}\eps$, then $r^{c-s}(x)>\eps$. Hence
\begin{equation}\label{longs-longcs-set}
\pi(B_{W,\eps})\subset B_{\gamma,\Cr{const:BP1}\eps}.
\end{equation}
Theorem~\ref{thm:long_cs_manifolds} follows by {combining} (\ref{CM:longs}), (\ref{longs-longcs-measure}) and (\ref{longs-longcs-set}).\hfill
\end{proof}

\begin{proof}[Proof of Theorem~\ref{thm:holonomy_dynholder}]
We use the notations introduced above and those of Theorem~\ref{thm:holonomy_dynholder}. Let $W_1^u$ and $W_2^u$ be two good $u$-curves, and let
$\partial Q_{\rm{tr}}$ be the transparent wall constructed from $W_1^u$ as discussed above. Let us assume that $W_1^u$ and $W_2^u$ are sufficiently close that
both leave traces on $\partial Q_{\rm{tr}}$, to be denoted by $\gamma_1$ and $\gamma_2$, respectively. Let $H_1\subset W_1^u$
denote the set of points for which the (homogeneous) central-stable manifold (of the flow) extends to $W_2^u$. Analogously, let
$H^0_1\subset \gamma_1$ denote the set of points for which the (homogeneous) stable manifold (of the map) extends to $\gamma_2$.
Let $h:H_1\to W_2^u$ and $h_0:H^0_1\to \gamma_2$ denote the holonomy maps given by sliding along central-sable and stable manifolds, respectively, and let
$H_2=h(H_1)$ and $H_2^0=h_0(H_1^0)$ denote the ranges of these holonomy maps, respectively.  We have $H^0_i=\pi(H_i)$ for $i=1,2$.
For $p,p'\in H_1$ the separation times on $H_1, H_2, H_1^0, H_2^0$ are equal, i.e.~$s^+(p,p')=s^+(\pi(p),\pi(p'))=s^+(h(p),h(p'))=s^+(h_0(\pi(p)),h_0(\pi(p')))$.

To express the Jacobian $J_h$, it is worth introducing the restrictions of the Jacobians $J_1= J_{\gamma_1\to W_1^u}|_{H_1^0}$ and $J_2= J_{\gamma_2\to W_2^u}|_{H_2^0}$. Then we have
\begin{equation}\label{eq:Jreduce}
J_h=(J_1)^{-1} \cdot J_{h_0} \cdot J_2.
\end{equation}

Theorem~\ref{thm:holonomy_dynholder} follows form the following claim.

\smallskip\noindent \textit{Claim.} There exist uniform constants $C>0$ and $\Theta<1$ such that $C^{-1}<J(p)<C$ and $|\log J(p)- \log J(p')|\le C\Theta^{s^+(p,p')}$, for
$J=(J_1)^{-1}$, $J_2$ and $J_{h_0}$.

\smallskip
Concerning $J_1$ and $J_2$ first note that the analysis of Lemma~\ref{lem:JacobianMapFlow} applies to $J_{\gamma_i\to W_i^u}$ for $i=1$ and $2$ as well. Also, for
$p, p'\in H_1$ let $s$ and $s'$ denote the arc length parameters of the projections $\pi(p),\pi(p')\in \gamma_i$, respectively. Then, by uniform expansion $|s-s'|\le C\Theta^{s^+(p,p')}$. Hence Lemma~\ref{lem:JacobianMapFlow} implies the Claim for $(J_1)^{-1}$ and $J_2$.

As for $J_{h_0}$, the Claim is explicitly stated in \cite{ChM06}, Proposition 5.48. This completes the proof of the Claim, and hence Theorem~\ref{thm:holonomy_dynholder} follows.\hfill
\end{proof}

\begin{proof}[Proof of Theorem~\ref{thm:holonomy_cs-regularity}]
Fix $p_1\in W_1$. Introduce the transparent wall $\partial Q_{\rm{tr}}$ corresponding to $W_1$ as above. Throughout, we use the notations from the proof of Theorem~\ref{thm:holonomy_dynholder}.
Note that $p_2=h(p_1)$. Projections onto $\cM_{\rm{tr}}$ are denoted as $q_i=\pi(p_i)$ and $\gamma_i=\pi(W_i)$; $(i=1,2)$.
For brevity let us introduce $J_i=J_i(q_i)$ $(i=1,2)$ and $J_0=J_{h_0}(q_1)$. As discussed in the proof of Theorem~\ref{thm:holonomy_dynholder}, $C^{-1}\le J \le C$ for $J=J_0, J_1, J_2$. Furthermore, the decomposition of Formula (\ref{eq:Jreduce}) applies and
\[
|J_h(p)-1|\le |J_1^{-1}\cdot J_2|\cdot |J_0-1| + |J_1^{-1}\cdot J_2-1|.
\]

Hence the following two claims imply the statement of Theorem~\ref{thm:holonomy_cs-regularity}.

\smallskip \noindent \textit{Claim 1.} $|J_2-J_1|\le C(\delta+\alpha)$.

\smallskip \noindent \textit{Claim 2.} $|J_0-1|\le C(\delta+\alpha^{1/3})$.

Let $\tau_i=\tau(q_i)$, $\phi_i=\phi(q_i)$, $(i=1,2)$. Let, furthermore, $m_i$ denote the slope of $\gamma_i$ at $q_i$, and let $B_i$ denote the curvature of the corresponding outgoing front ($B_+$ in the proof of Lemma~\ref{lem:JacobianMapFlow}).

We make the following observations.
\begin{enumerate}[(i)]
\item As $d(p_1,p_2)=\delta$, we have $d(q_1,q_2)\le C\delta$. Hence $|\tau_2-\tau_1|\le C\delta$ and $|\phi_2-\phi_1|\le C\delta$.
\item Recall that $\alpha$ denotes the angle of the tangent vectors $v_1=T_{p_1}W_1$ and $v_2=T_{p_2}W_2$. Let $v_i = (d\xi_i,d\eta_i,d\omega_i)$, for $i=1,2$. As it only the directions of the vectors $v_1$ and $v_2$ that matter, we may fix the
$d\xi_i$ to have unit length. $d\xi_1$ and $d\xi_2$ correspond to the flow directions at $p_1$ and $p_2$, respectively, hence $|d\xi_1-d\xi_2|\le C\delta$. It follows that
$\left|\frac{d\omega_1}{d\xi_1}- \frac{d\omega_2}{d\xi_2}\right|\le C(\delta+\alpha)$ and
$\left|\frac{d\eta_1}{d\xi_1}- \frac{d\eta_2}{d\xi_2}\right|\le C(\delta+\alpha)$.
On the other hand, by the analysis presented in the proof of Lemma~\ref{lem:JacobianMapFlow},
$\frac{d\omega_i}{d\xi_i}=(\tau_i +B_i^{-1})^{-1}$ and
$\frac{d\omega_i}{d\xi_i}=\frac{\sin\phi_i+\tau_i'}{\cos\phi_i(1+\tau_i B_i)}$, for $i=1,2$.
\item It follows that $|(\tau_2+B_2^{-1})^{-1}-(\tau_1+B_1^{-1})^{-1}|\le C(\delta+\alpha)$. Furthermore, these quantities, just like the $B_i$ ($i=1,2$), are bounded away from zero and infinity by some geometric constants. This, along with the previous items, implies $|B_2- B_1|\le C(\delta+\alpha)$ and hence $|m_2-m_1|\le C(\delta+\alpha)$. Similarly, $|(\sin\phi_2 + \tau'(q_2)) - (\sin\phi_1 + \tau'(q_1))|\le C(\delta+\alpha)$ (and thus $|\tau'(q_2)-\tau'(q_1)|\le C(\delta+\alpha)$).
\end{enumerate}

To prove Claim 1, we use the explicit Formula (\ref{eq:J0expl}) for both $J_1(q_1)$ and $J_2(q_2)$. Note first that in this formula both the numerator and the denominator are bounded away from zero and infinity by geometric constants. Hence Claim 1 follows from the above items.

To prove Claim 2, we refer to Theorem 5.42 in \cite{ChM06}. This Theorem implies that
$|J_{h_0}- 1|\le C (|\sphericalangle(\gamma_1(q_1),\gamma_2(q_2))| +d(q_1,q_2)^{1/3})$, where
$d(q_1,q_2)$ is the distance of $q_1$ and $q_2$, while $\sphericalangle(\gamma_1(q_1),\gamma_2(q_2))$ is the angle that the direction of $\gamma_1$ at $q_1$ makes with the direction of $\gamma_2$ at $q_2$ -- which is bounded by $C|m_1-m_2|$.
Now, by the analysis above, $d(q_1,q_2)\le C\delta$ while $\sphericalangle(\gamma_1(q_1),\gamma_2(q_2))\le C(\alpha+\delta)$. This implies Claim 2, and thus completes the proof of Theorem~\ref{thm:holonomy_cs-regularity}.\hfill
\end{proof}

\subsection{Exponential correlation decay for generalized H\"older observables}\label{sec:genholder-EDC}

Here we present a simple extension of Theorem~\ref{thm:BDL}, which is from \cite{BDL16}, to allow generalized H\"older observables. In the first step we only allow one of the two observables to be generalized H\"older, because this is the version we will use.

\begin{theorem}\label{thm:BDL_extended}
A planar billiard flow with finite horizon and no corner points enjoys exponential correlation decay for a H\"older and a generalized H\"older observable. Quantitatively, let $0 < \alpha_F,\alpha_G \le 1$. Then there exists an $a'' = a''(Q, \alpha_F,\alpha_G) > 0$ and a $\cC_{BDL;g}=\cC_{BDL;g}(Q,\alpha_F,\alpha_G) < \infty$ such that if $F:M\to \IR$ is $\alpha_F$-generalized H\"older and $G:M\to \IR$ is $\alpha_G$-H\"older, then for any $t\ge 0$
 \[\left| \int_M (F\circ \Phi^t) G \diff \mu - \int_M F \diff \mu \int_M G \diff \mu \right| \le \cC_{BDL;g} var_{\alpha_F}(F) ||G||_{\alpha_G;H} e^{- a''t}. \]
(Here BDL stands for Baladi-Demers-Liverani.)

Moreover, $a''(Q,\alpha_F,\alpha_G)$ depends on $Q$ and $\alpha_G$ only through $a'(Q,\alpha_G)$ from Theorem~\ref{thm:BDL}: actually, $a''=\frac{\alpha_F}{\alpha_F+1} a'(Q,\alpha_G)$. Similarly, $\cC_{BDL;g}(Q,\alpha_F,\alpha_G)$ depends on $Q$ only through $\cR_Q$ from (\ref{eq:RQ}) and $\cC_{BDL}(Q,\alpha_G)$ from Theorem~\ref{thm:BDL}, so it has the form
$\cC_{BDL;g}=\cC_{BDL;g}(\cR_Q,\cC_{BDL}(Q,\alpha_G),\alpha_F,\alpha_G)$.
\end{theorem}

\begin{proof}
Both the left and the right hand side remain unchanged if we add a constant to $F$, so it is common to assume -- without loss of generality -- that $\int_M F\diff\mu=0$.  However, the H\"older norm -- which is used in Theorem~\ref{thm:BDL} -- does depend on additive constants. On the other hand, if $F$ takes both positive and negative values, we have
\[
\sup |F|\le \sup F - \inf F\le 2 \sup |F|,
\]
hence in this case the H\"older norm can be controlled by $var_{\alpha_F}(F)$. For convenience, in the argument below we assume, without loss of generality, that
\begin{equation}\label{eq:assume-F-symmetric}
\sup |F|=\sup F = -\inf F =\frac{\sup F - \inf F}{2}.
\end{equation}
Note that if $F$ is such that $\int_M F\diff\mu=0$, and a constant is added to satisfy condition (\ref{eq:assume-F-symmetric}), $\sup |F|$ is affected at most by a factor $2$, while both versions of $F$ take both positive and negative values (unless $F$ is identically zero).

To prove the statement, the natural idea is to choose some $r>0$, replace $F$ with the smoothed version $\tilde{F}(x):=\frac{1}{Leb(B_r(x))}\int_{B_r(x)}F\diff Leb$, and then apply Theorem~\ref{thm:BDL} to $\tilde{F}$ and $G$. This $\tilde{F}$ would be Lipschitz continuous whenever $F$ is (measurable and) bounded, if $F$ were defined on $\IR^d$ or $\IT^d$. The only thing we have to be careful about is that $F$ is only defined on the phase space $M = Q\times \IT\subset \IT^3$. For that reason, we first extend $F$ to the $r$-neighbourhood $B_r(M)$ of $M$ in a fairly arbitrary way, say
\[\hat{F}(x):=\begin{cases}
                F(x) & \text{if $x\in M$}, \\
                \inf\{F(y)\,|\,y\in B_r(x)\cap M\} & \text{if $x\in B_r(M)\setminus M$}.
               \end{cases}
\]
Now for $x\in M$ we can define
\[
\tilde{F}(x):=\frac{1}{Leb(B_r(x))}\int_{B_r(x)} \hat{F}\diff Leb=\frac{3}{4r^3\pi}\int_{B_r(x)} \hat{F}\diff Leb.
\]
These definitions ensure that $\inf \hat{F}=\inf F$, $\sup \hat{F}=\sup F$, so for any $x,y\in M$
\begin{align*}
\tilde{F}(y)-\tilde{F}(x)= &
\frac{3}{4r^3\pi}\left(\int_{B_r(y)\setminus B_r(x)}\hat{F}\diff Leb - \int_{B_r(x)\setminus B_r(y)}\hat{F}\diff Leb \right) \le \\
\le & \frac{3}{4r^3\pi} \left(\sup \hat{F} Leb(B_r(y)\setminus B_r(x)) - \inf \hat{F} Leb(B_r(x)\setminus B_r(y)) \right) \le \\
\le & \frac{3}{4r^3\pi} (\sup F - \inf F) r^2 \pi |y-x| = \frac{3}{4\pi} (\sup F - \inf F) \frac1r |y-x|,
\end{align*}
so $\tilde{F}$ is Lipschitz continuous and
\[
|\tilde{F}|_{1;H}\le \frac{3}{4\pi} (\sup F - \inf F) \frac1r.
\]
Now Lemma~\ref{lem:Holder-Holder} \item~\ref{it:Holder-Holder} says that $\tilde{F}$ is also $\alpha_G$-H\"older and with (\ref{eq:assume-F-symmetric}) we get
\begin{equation}\label{eq:Ftilde-holder-norm}
||\tilde{F}||_{\alpha_G;H}
\le \left(\frac{3}{4\pi} diam(M)^{1-\alpha_G} \frac1r + \frac12 \right) (\sup F - \inf F)
\le \left(\frac{3}{4\pi} diam(M)^{1-\alpha_G} \frac1r + \frac12 \right) var_{\alpha_F}F.
\end{equation}
On the other hand, the definitions imply that
\[\inf_{B_{2r}(x)} F \le \tilde{F}(x) \le \sup_{B_{2r}(x)} F,\]
so
\[|\tilde{F}(x)-F(x)|\le \sup_{B_{2r}(x)} F - \inf_{B_{2r}(x)} F = (osc_{2r}F)(x),\]
which means that
\begin{equation}\label{eq:Ftilde-error-L1-norm}
||\tilde{F}-F||_{L^1(\mu)} \le \int_M osc_{2r}F \diff\mu \le |F|_{\alpha_F;gH}(2r)^{\alpha_F}
\le 2^{\alpha_F} var_{\alpha_F} F \cdot r^{\alpha_F}.
\end{equation}
To prove the theorem, we write
\[
\left| \int_M (F\circ \Phi^t) G \diff \mu - \int_M F \diff \mu \int_M G \diff \mu \right| \le
\left| \int_M (\tilde{F}\circ \Phi^t) G \diff \mu - \int_M \tilde{F} \diff \mu \int_M G \diff \mu \right|
+ 2\sup|G|\cdot ||\tilde{F}-F||_{L^1(\mu)}.
\]
We estimate the first term using Theorem~\ref{thm:BDL} and (\ref{eq:Ftilde-holder-norm}), and the second term using (\ref{eq:Ftilde-error-L1-norm}) to get
\begin{multline}
\left| \int_M (F\circ \Phi^t) G \diff \mu - \int_M F \diff \mu \int_M G \diff \mu \right| \le \\
\le \cC_{BDL}(Q,\alpha_G) ||\tilde{F}||_{\alpha_G;H} ||G||_{\alpha_G;H}e^{-a'(Q,\alpha_G)t}
 + 2||G||_{\alpha_G;H}  2^{\alpha_F} var_{\alpha_F} F \cdot r^{\alpha_F} \le \\
\le var_{\alpha_F} {F}||G||_{\alpha_G;H}
\left(
\frac{3 \cC_{BDL}}{4\pi} diam(M)^{1-\alpha_G} \frac{e^{-a'(Q,\alpha_G)t}}{r} + \frac{\cC_{BDL}}{2} e^{-a'(Q,\alpha_G)t}
+ 2^{\alpha_F+1} r^{\alpha_F}
\right).
\end{multline}
Choosing $r=e^{-\frac{a'}{\alpha_F+1}t}\le 1$ and
\[a''=a''(Q,\alpha_F,\alpha_G):=\frac{\alpha_F}{\alpha_F+1} a'(Q,\alpha_G)\]
we get $r^{\alpha_F}=\frac{e^{-a't}}{r}=e^{-a'' t}$ and $e^{-a' t}\le e^{-a'' t}$, so the statement follows with
\[
\cC_{BDL;g}=\cC_{BDL;g}(Q,\alpha_F,\alpha_G):=
2^{\alpha_F+1} + \left(\frac{3}{4\pi} diam(M)^{1-\alpha_G} + \frac{1}{2} \right) \cC_{BDL}(Q,\alpha_G).
\]
\end{proof}

For completeness we state the extension for the case when both observables are only generalized H\"older. We will not use this version in this paper.

\begin{theorem}\label{thm:BDL_extended2}
A planar billiard flow with finite horizon and no corner points enjoys exponential correlation decay for generalized H\"older observables. Quantitatively, let $0 < \alpha_F,\alpha_G \le 1$. Then there exists an $a''' = a'''(Q, \alpha_F,\alpha_G) > 0$ and a $\cC_{BDL;gg}=\cC_{BDL;gg}(Q,\alpha_F,\alpha_G) < \infty$ such that if $F:M\to \IR$ is $\alpha_F$-generalized H\"older and $G:M\to \IR$ is $\alpha_G$-generalized H\"older, then for any $t\ge 0$
 \[\left| \int_M (F\circ \Phi^t) G \diff \mu - \int_M F \diff \mu \int_M G \diff \mu \right| \le \cC_{BDL;gg} var_{\alpha_F}(F)   var_{\alpha_G}(G) e^{- a'''t}.\]
(Here BDL stands for Baladi-Demers-Liverani.)
\end{theorem}

The proof is completely analogous to the proof of Theorem~\ref{thm:BDL_extended}, and we omit it.

\subsection{Extension of H\"older continuous functions}\label{sec:Holder-extension-1p}

\begin{lemma}\label{lem:Holder-extension-1p}
 Let $(X,d)$ be a metric space, $\emptyset\neq D\subset X$, $z\in X$ and $f:D\to\IR$ H\"older continuous with constants $0 \le C<\infty$ and $0<\alpha \le 1$ meaning that
\begin{equation}\label{eq:Holder-ext-0}
|f(x)-f(y)|\le C d(x,y)^{\alpha} \text{ for any $x,y\in D$}.
\end{equation}
 Then there exists a function $\tilde{f}:D\cup\{z\}\to\IR$ such that $\tilde{f}=f$ on $D$, $\inf f\le \tilde{f} \le \sup f$ and $\tilde{f}$ is H\"older continuous with the same constants as $f$:
 \begin{equation}\label{eq:Holder-ext-1}
 |\tilde{f}(x)-\tilde{f}(y)|\le C d(x,y)^{\alpha} \text{ for any $x,y\in D\cup\{z\}$}.
 \end{equation}
\end{lemma}

\begin{proof}
We must have $\tilde{f}=f$ on $D$, so (\ref{eq:Holder-ext-1}) is trivial if $x,y\neq z$, and also if $x=y=z$. So we only need to choose $f(z)\in[\inf f,\sup f]$ so that (\ref{eq:Holder-ext-1}) holds with $y=z$ for every $x\in D$.
For every $x\in D$, define the numbers
\[ a_x:=f(x)-C d(x,z)^{\alpha}\quad , \quad b_x:=f(x)+C d(x,z)^{\alpha}\]
and the interval $I_x:=[a_x,b_x]$. If we can choose $f(z)$ such that $f(z)\in I_x\cap [\inf f,\sup f]$ for every $x\in D$, we are done because $f(z)\in I_x$ is equivalent to (\ref{eq:Holder-ext-1}) with $y=z$. Set
\[a:=\sup_{x\in D} a_x\quad , \quad b:=\inf_{x\in D} b_x.\]
Clearly $a\le \sup f$ and $b\ge\inf f$, so it enough to show that
\begin{equation}\label{eq:Holder-ext-2}
 a\le b,
\end{equation}
since this means not only that $\cap_{x\in D} I_x=[a,b]\neq\emptyset$, but also that
\[\cap_{x\in D} I_x\cap [\inf f, \sup f]=[a,b]\cap[\inf f,\sup f]\neq\emptyset.\]
First we show that $I_x\cap I_y\neq\emptyset$ for any $x,y\in D$.
The function $g(u):= u^\alpha$ is concave and satisfies $g(0)\ge 0$. Such functions are known to be subadditive:
$(u+v)^\alpha\le u^\alpha + v^\alpha$ for any $u,v\ge 0$.
So the triangle inequality $d(x,y)\le d(x,z) + d(y,z)$ implies that
\[d(x,y)^\alpha \le (d(x,z)+d(y,z))^\alpha\le d(x,z)^\alpha + d(y,z)^\alpha,\]
so (\ref{eq:Holder-ext-1}) implies that
\[ f(y)-f(x)\le C d(x,z)^\alpha + C d(y,z)^\alpha. \]
Rearranging this inequality gives
\[ a_y= f(y)-C d(y,z)^\alpha\le f(x)+C d(x,z)^\alpha =b_x.\]
Similarly, we can get $a_x\le b_y$, so $I_x\cap I_y=[\max\{a_x,a_y\},\min\{b_x,b_y\}]\neq \emptyset$.

Now if $b<a$ would hold, we would also have some $x$ and $y$ with $b_x < a_y$, a contradiction (we assumed $D\neq\emptyset$). So
we have shown (\ref{eq:Holder-ext-2}) and the proof is complete.
\end{proof}

\begin{lemma}\label{lem:Holder-extension}
 Let $(X,d)$ be a separable metric space, $\emptyset\neq D\subset X$ and $f:D\to\IR$ H\"older continuous with constants $0 \le C<\infty$ and $0<\alpha \le 1$ meaning that
 \[ |f(x)-f(y)|\le C d(x,y)^{\alpha} \text{ for any $x,y\in D$}. \]
 Then there exists a function $\tilde{f}:X\to\IR$ such that $\tilde{f}=f$ on $D$, $\inf f\le \tilde{f} \le \sup f$ and $\tilde{f}$ is H\"older continuous with the same constants as $f$:
 \[ |\tilde{f}(x)-\tilde{f}(y)|\le C d(x,y)^{\alpha} \text{ for any $x,y\in X$}. \]
\end{lemma}

\begin{proof}
 Let $E=\{x_1,x_2,\dots\}$ be a countable dense subset in $X$. Let $D_0:=D$ and $D_n:=D\cup\{x_1,\dots,x_n\}$ for $n=1,2,\dots$. We define $\tilde{f}$ of each $x_i$ inductively: If it is already defined with the required properties on $D_{n-1}$, then Lemma~\ref{lem:Holder-extension-1p} gives $\tilde{f}(x_n)$ by the extension to $D_n$. The function $\tilde{f}$ defined on $D\cup E$ this way also has the required properties (for any $x,y\in D\cup E$, $\tilde{f}(x)$ and $\tilde{f}(y)$ are obtained in finitely many steps). Now we have a continuous function on a dense set, so the continuous extension to all of $X$ exists and is unique with the trivial definition
 $\tilde{f}(x):=\lim\limits_{D\cup E\ni y\to x} \tilde{f}(y)$,
 and obviously preserves the regularity properties.
\end{proof}

\end{document}